\documentclass[12pt,a4paper,twoside]{article}

\usepackage[active]{srcltx}
\usepackage[utf8]{inputenc}
\usepackage{hyperref}
\usepackage{url}
\usepackage[english]{babel}
\usepackage{authblk}
\usepackage{overpic}

\usepackage{comment}
\usepackage{xcolor}

\usepackage[a4paper,hcentering,vcentering]{geometry}
\usepackage{url,epsfig,amssymb,amsthm,amsmath, mathtools}
\usepackage{subfig,graphicx}

\usepackage{tabstackengine} 
\setstackEOL{\cr}

\numberwithin{equation}{section}

\newtheorem{theorem}{Theorem}[section]

\newtheorem{corollary}[theorem]{Corollary}
\newtheorem{lemma}[theorem]{Lemma}
\newtheorem{proposition}[theorem]{Proposition}
\newtheorem{remark}[theorem]{Remark}

\definecolor{JungleGreen}{cmyk}{0.99,0,0.52,0}
\definecolor{RawSienna}{cmyk}{0,0.72,1,0.45}
\definecolor{bulgarianrose}{rgb}{0.28, 0.02, 0.03}
\definecolor{Magenta}{cmyk}{0,1,0,0}
\definecolor{airforceblue}{rgb}{0.36, 0.54, 0.66}
\definecolor{darkpastelgreen}{rgb}{0.01, 0.75, 0.24} 
\definecolor{brightgreen}{rgb}{0.4, 1.0, 0.0} 

\newcommand{\C}{\ensuremath{\mathbb{C}}}

\newcommand{\F}{\ensuremath{\mathcal{F}}}
\newcommand{\G}{\ensuremath{\mathcal{G}}}

\newcommand{\Oo}{\ensuremath{\mathcal{O}}}
\newcommand{\N}{\ensuremath{\mathbb{N}}}

\newcommand{\R}{\ensuremath{\mathbb{R}}}
\newcommand{\T}{\ensuremath{\mathbb{T}}}
\newcommand{\Z}{\ensuremath{\mathbb{Z}}}
\newcommand{\Lout}{\ensuremath{\mathcal{L}^{\mathrm{out}}}}
\newcommand{\Fout}{\ensuremath{\mathcal{F}^{\mathrm{out}}}}
\newcommand{\Lin}{\ensuremath{\mathcal{L}^{\mathrm{in}}}}
\newcommand{\Fin}{\ensuremath{\mathcal{F}^{\mathrm{in}}}}
\newcommand{\Lext}{\ensuremath{\mathcal{L}^{\mathrm{gf}}}}
\newcommand{\Fext}{\ensuremath{\mathcal{F}^{\mathrm{gf}}}}
\newcommand{\Lfl}{\ensuremath{\mathcal{L}^{\mathrm{fl}}}}
\newcommand{\Ffl}{\ensuremath{\mathcal{F}^{\mathrm{fl}}}}

\newcommand{\bs}{\mathrm{s}}
\newcommand{\bu}{\mathrm{u}}
\newcommand{\Delin}{\ensuremath{\Delta^{\mathrm{in}}}}
\newcommand{\Cin}{\ensuremath{c}}
\newcommand{\Cout}{\ensuremath{c}}
\newcommand{\Cmch}{\ensuremath{c}}
\newcommand{\slope}{\ensuremath{\vartheta_0}}
\newcommand{\hslope}{\ensuremath{\hat \vartheta_0}}
\newcommand{\slopee}{\ensuremath{\hat \vartheta_1}}
\renewcommand{\Re}{\ensuremath{\mathrm{Re}\,}}
\renewcommand{\Im}{\ensuremath{\mathrm{Im}\,}}
\renewcommand{\a}{\ensuremath{\mathbf{a}}}
\renewcommand{\b}{\ensuremath{\mathbf{b}}}
\newcommand{\gin}{\ensuremath{\mathbf{g}}}
\newcommand{\E}{E}
\newcommand{\diagonal}{\ensuremath{\mathbf{D}}}
\newcommand{\Lhat}{\widehat{\mathcal{L}}}

\newcommand{\Ghat}{\widehat{\mathcal{G}}} 
\newcommand{\Ggf}{\Gamma^{\mathrm{fl}}}
\newcommand{\Dgf}{D^{\mathrm{gf}}}

\begin{document}

\title{Chaotic phenomena in generic unfoldings of the Hamilton Hopf bifurcation with emphasis on the restricted planar circular 3-body problem beyond the Gascheau-Routh mass ratio}%
\author[1,3]{Inmaculada Baldomá}
\author[1,3]{Pau Martín}
\author[2]{Donato Scarcella}
\affil[1]{Departament de Matemàtiques, Universitat Politècnica de Catalunya, Diagonal, 647, 08028 Barcelona, Spain.}
\affil[2]{Departament de Matemàtiques i Informàtica, Universitat de Barcelona (UB), Gran Via, 585, 08007 Barcelona, Spain.}
\affil[3]{Centre de Recerca Matemàtica (CRM), Carrer de l'Albareda, 08193 Bellaterra, Spain}

\affil[ ]{\textit{Contributing authors:} immaculada.baldoma@upc.edu; p.martin@upc.edu; donato.scarcella@ub.edu}

\date{\null}%
\maketitle

\begin{abstract}
In this work, we prove that a generic unfolding of an analytic Hamiltonian Hopf singularity (in an open set with codimension 1 boundary) possesses transverse homoclinic orbits for subcritical values of the parameter close to the bifurcation parameter. As a consequence, these systems display chaotic dynamics with arbitrarily large topological entropy. We verify that the Hamiltonian of the restricted planar circular three-body problem (RPC3BP) close to the Lagrangian point $L_4$ falls within this open set. The generic condition ensuring the presence of transversal homoclinic intersections is subtle and involves the so-called Stokes constant. Thus, in the case of the RPC3BP close to $L_4$, our result holds conditionally on the value of this constant.

\end{abstract}

\tableofcontents

\section{Introduction}

\subsection{Generic unfoldings of an analytic Hamiltonian Hopf bifurcation}

In this paper, we consider generic unfoldings of the analytic Hamiltonian Hopf bifurcation, with special emphasis on the bifurcation that occurs at the equilibria $L_4$ and $L_5$ of the restricted planar circular 3-body problem (RPC3BP) at the Gascheau-Routh mass ratio. See Section~\ref{sec:RPC3BP} for a detailed description of the RPC3BP.

The Hamiltonian Hopf singularity occurs when the linearized system of a Hamiltonian vector field at an equilibrium has a non-semisimple 1:-1 resonance (see~\cite{Meer85}, \cite{BHH07}, and also Section~\ref{sec:normalform}, below). When considering unfoldings  of this singularity, in the supercritical case, when the eigenvalues are pure imaginary, it is well known that generically KAM theory applies, thus making the equilibrium stable~\cite{Deprit67,MS86}. This also holds for the particular case of $L_4$ and $L_5$ in RPC3BP. We are interested in the subcritical case, when the equilibrium becomes a complex saddle.
In this situation, the invariant manifolds of the equilibrium for the formal normal form coincide for a codimension 1 open set of analytic Hamiltonians.
Again, this is also the case of $L_4$ and $L_5$ in RPC3BP (see, for instance, \cite{Sch94} and also Section~\ref{sec:normalform}, below). Our main theorem is the following.

\begin{theorem}\label{thm:main:intro}
Consider $\mathcal{H}$, the set of analytic Hamiltonians with a Hamiltonian Hopf singularity.
There is an open subset $\mathcal{J} \subset \mathcal{H}$, whose boundary has codimension 1, with the following property.
Any analytic generic unfolding of a Hamiltonian Hopf singularity belonging to $\mathcal{J}$ has transversal homoclinic orbits for all values of the parameter in the subcritical region close to the bifurcation.
\end{theorem}

The proof of this theorem is placed in Section~\ref{subsec:hamilton-Jacobi}, where we will deduce it from Theorem~\ref{thm:difference_between_manifolds}.

Of course, in here transversal means transversal in the energy level.

We remark that the codimension 1 condition of the theorem is explicit and computable. It involves a single coefficient of the versal normal form of the singularity. The theorem applies to those singularities for which the coefficient has a prescribed sign. See Theorem~\ref{Thm:NormalForm} and~\eqref{def:condicio_homoclinica}. It holds in the case of $L_4$ and $L_5$ in RPC3BP. This condition ensures that the invariant manifolds of the versal normal form have a homoclinic connection. For Hamiltonian Hopf singularities outside $\mathcal{J}$, the invariant manifolds of the normal form do not form a homoclinic loop.

The generic condition that an unfolding has to satisfy in order to have transverse homoclinic orbits is more subtle and is related to what is often known as the Stokes constant.  Hence, when considering  $L_4$ and $L_5$ in RPC3BP, we have a conditional statement. See Theorem~\ref{thm:main:intro:L4}, below.

The dynamical consequences of Theorem~\ref{thm:main:intro} are enormous. It implies, in particular, that chaotic behavior and large topological entropy are generic phenomena in an open set in the subcritical Hamiltonian Hopf bifurcation. See Section~\ref{sec:dinamica_caotica}.

Gaiv\~{a}o, in this PhD thesis~\cite{Gai10}, has the same statement for generic unfoldings of the Hamiltonian Hopf bifurcation satisfying a certain symmetry condition (besides the aforementioned codimension 1 condition). In fact, this symmetry condition is not satisfied at $L_4$ and $L_5$ in RPC3BP. Furthermore, our techniques, although sharing some basic common features with the ones in~\cite{Gai10} --- since the problem under consideration is an \emph{exponentially small problem} in some parameter, see Section~\ref{sec:exponencialment_petit} --- differ strongly. In particular, we are able to get rid of the symmetry condition. 

It is important to remark that, although Theorem~\ref{thm:main:intro} is a statement on generic unfoldings, our techniques are constructive and can be applied to detect transversal intersections in a \emph{given} family of Hamiltonian systems with a Hamiltonian Hopf bifurcation. We will constantly  have in mind the case of the RPC3BP, described in the next section, as a paradigmatic example.

\subsection{The restricted circular planar three-body problem}
\label{sec:RPC3BP}

The restricted circular 3-body problem models the motion of a ``massless'' body under the influence of two bodies with mass, the \emph{primaries}, which evolve in Keplerian circles. The problem is called planar when the motion of the massless body is co-planar with the primaries. 

By choosing appropriate units of mass, distance and time, we can assume that the masses of the primaries are $m_1 = 1-\mu$ and $m_2 = \mu$, with $\mu \in (0, \frac12]$, their distance is 1 and that they make a revolution around their common center of mass, at the origin, in $2\pi$ units of time. We choose a rotating frame of coordinates that fixes their position in time.  So, the primaries are fixed at positions $(-\mu, 0)$ and $(1-\mu,0)$, respectively.  We denote by $(q,p) \in \R^2 \times \R^2$ the position and momenta of the third body in the rotating frame. The equations of motion of the massless body are Hamiltonian with  Hamiltonian
\begin{equation}
\label{H1}
H(q,p) = \frac{|p|^2}{2} - q^\top \begin{pmatrix} 0 & 1 \\ -1 & 0\end{pmatrix} p -\frac{1-\mu}{ |q + (\mu, 0)|}-\frac{\mu}{|q - (1- \mu, 0)|},
\end{equation}
where $|\cdot|$ stands for the Euclidean norm and $\top$ for the transpose.  It has 2 degrees of freedom. We point out that the $H$ is analytic far from $(-\mu, 0)$ and $(1-\mu, 0)$, that is, outside collisions with the primaries.

It is well known that Hamiltonian~\eqref{H1} has five critical points, $L_i$, with $i=1,\dots 5$. See Figure~\ref{fig:equilibris}. The first three are located in the line connecting the primaries. For all values of the mass parameter $\mu$, these three equilibria are of center-saddle type. The other two, $L_4$ and $L_5$, are at the vertices of the equilateral triangles, with one of the sides being the line joining the primaries. Unlike the collinear equilibria, the stability of $L_4$ and $L_5$ changes at 
\begin{equation}
\label{def:Routh}
\mu = \mu_1 = \frac12 \left(1 - \frac19\sqrt{69} \right), 
\end{equation}
the so-called Gascheau-Routh critical mass ratio \cite{Gascheau43}.
When $0 < \mu < \mu_1$, these points are linearly stable, while they become complex saddles when $\mu_1 < \mu \le 1/2$. At $\mu = \mu_1$, a Hamiltonian Hopf bifurcation occurs. In this particular case, after the bifurcation, a homoclinic loop appears. More concretely, the versal normal form of the Hamiltonian, truncated up to order three, is integrable and possesses a homoclinic loop (see~\cite{MM03} and the references therein, also Section~\ref{sec:normalform}), hence making way to possible interesting homoclinic phenomena if the homoclinic loop is not preserved for the complete Hamiltonian and the manifolds intersect transversely. However, it is well known that the versal normal form up to any order is integrable. Hence, to decide if the manifolds of $L_4$ intersect transversally is what is often called a ``beyond all orders'' problem. In this work, we obtain a formula for the difference between the invariant manifolds, asymptotic for $\mu - \mu_1>0$ tending to $0$. If some coefficient is different from zero --- the Stokes constant ---, our formula implies that the manifolds intersect transversally. More concretely, our theorem concerning $L_4$ is the following.

\begin{figure}
\begin{center}
\includegraphics[width=0.6\textwidth ]{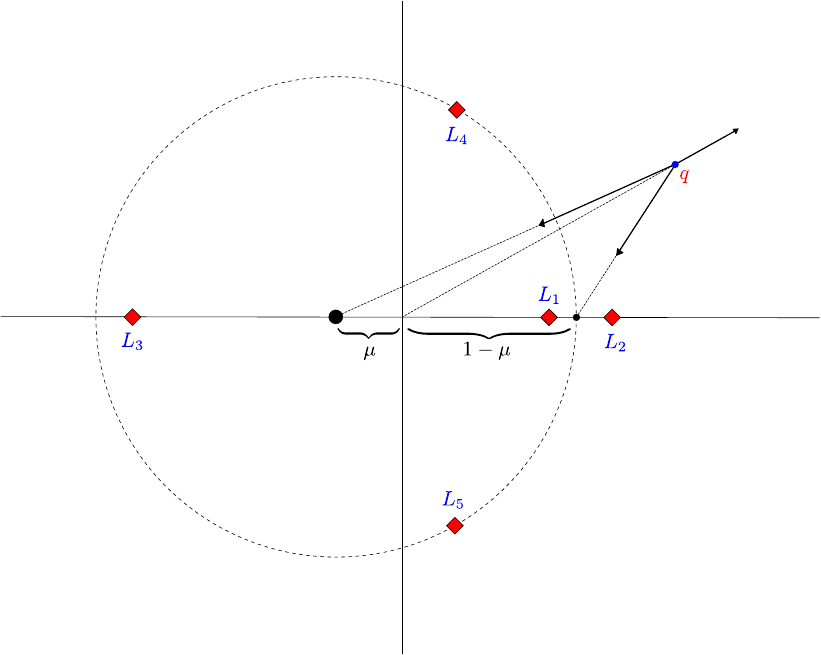}
    \end{center}
    \caption{The primaries, as black dots, and the equilibria, as red diamonds, of the RPC3BP in rotating coordinates. The black arrows represent the forces acting on the massless particle at $q$.}
    \label{fig:equilibris}
 \end{figure}

\begin{theorem}\label{thm:main:intro:L4}
There exist $\mu_0>\mu_1$ such that, for $\mu \in (\mu_1,\mu_0)$, the stable and unstable invariant manifolds of the Lagrangian point $L_4$ of the RPC3BP~\eqref{H1}, $W^{\bu,\bs}(L_4)$, intersect. In addition, there exists a constant $\Theta$, independent of $\mu$, such that at any homoclinic point $p\in W^{\bu}(L_4) \cap W^{\bs}(L_4)$ the angle $\varphi(p)$ between $W^{\bu}(L_4)$ and $W^{\bs}(L_4)$, measured at some section in the energy level $h(L_4)$,  
is
\begin{equation}
\label{teorema1:angle}
\varphi(p)= \frac{1}{(\mu-\mu_1)^{2}} e^{-\frac{2 \pi \sqrt 2  }{3 \sqrt{69} |\mu - \mu_1|^{1/2}}} \Big [\Theta + \mathcal{O}\left(\frac{1}{\log (\mu - \mu_1)}\right) \Big ] \cdot c(p)
\end{equation}
with $c(p)\neq 0$ a constant depending on the homoclinic point $p$.  
\end{theorem}

\begin{corollary} If $\Theta\neq 0$, then the invariant manifolds of $L_4$ intersect transversely in the energy level of $L_4$. 
\end{corollary}

The RPC3BP is quite often viewed as a first approximation of the motion in a system with two primaries. In the solar system, the mass ratio of the subsystem formed by Pluto and Charon lies beyond the Gascheau-Routh critical value. Many binary stars fall into this category.


The RPC3BP, as a simplified model of the full 3-body problem, has been the object of many studies. Since it has only two degrees of freedom, taking suitable Poincaré sections, it reduces to a two dimensional symplectic map, which greatly simplifies the task of finding complex dynamics in the model. At the same time, the low dimensionality prescribes the occurrence of other phenomena, like Arnold diffusion. However, some of the structures in the model, namely, hyperbolic sets, can quite often be continued to higher-dimensional models and act as guidance for the dynamics there.

Chaotic behavior linked to homoclinic phenomena has been found in different parts of the phase space of the model. Following the work of Moser~\cite{Moser01} for the Sitnikov problem, Llibre and Simó~\cite{SimoL80} proved the existence of oscillatory solutions, i.e., solutions $q(t)$ such that $\limsup_{t\in \R} |q(t)| = \infty$ while $\liminf_{t\in \R} |q(t)| <\infty$, for small values of $\mu$. This result was later extended to all values of $\mu$ in~\cite{GuardiaMS16}. These solutions are related to the invariant manifolds of the ``parabolic infinity'', which coincide for $\mu=0$ and give rise to transversal homoclinic solutions for $\mu>0$. Oscillatory orbits were recently obtained using variational methods in~\cite{paradela2022oscillatorymotionsrestricted3body}. Also related to the parabolic manifolds of infinity, the existence of Newhouse domains for the RPC3BP is established in~\cite{GMP25}.

Close to $L_3$, for $\mu$ close to 0, chaotic motions and quadratic tangencies are proven to exist in~\cite{baldomá2023coorbitalhomoclinicchaoticdynamics}, related to the invariant manifolds of $L_3$ \cite{BALDOMA2022108562,BALDOMA2023109218}.

\subsection{Chaotic dynamics in the Hamilton Hopf bifurcation and near $L_4$ in the RPC3BP}

\label{sec:dinamica_caotica}

Studying the dynamics of the RPC3BP close to $L_4$ and $L_5$, Str\"omgren~\cite{Stroemgren33}, from numerical computations,  conjectured the existence of families of periodic orbits depending on a parameter that accumulate to a transversal homoclinic connection to $L_4$, thus ``vanishing in thin air''. This phenomenon is called \emph{blue sky catastrophe}. Henrard, in~\cite{Hen73}, proved that the transversal intersection of the invariant manifolds of a Hamiltonian complex saddle (in particular, $L_4$ beyond the Gascheau-Routh mass ration), if it occurs, gives rise to the existence of such orbits. Later on, Devaney~\cite{Devaney76} showed that, besides the blue sky catastrophe, the existence of a transversal homoclinic orbit to Hamiltonian complex saddle implies the existence of horseshoes of $N$ symbols, for any $N$, with the correspondent consequences on the topological entropy of the system and the number of periodic orbits.
Theorem~\ref{thm:main:intro} implies that such behavior happens for generic unfoldings of any Hamiltonian Hopf singularity (in the open set $\mathcal{J}$ in Theorem~\ref{thm:main:intro}).

In the particular setting of the RPC3BP, Theorem~\ref{thm:main:intro:L4} implies, if $\Theta \neq 0$, that transversal homoclinic orbits to $L_4$ exist for any $\mu >\mu_1$ such that $\mu -\mu_1$ is small enough. Furthermore, one can see from the proof of Theorem~\ref{thm:main:intro:L4} that the parametrizations of the manifold that lead to formula~\eqref{teorema1:angle} are analytic for a rather large interval of $\mu$. Hence, in that interval, the angle will be non-zero except, at most, a finite number of values of $\mu$.

The transversal intersection of the manifolds can be deduced from Theorem~\ref{thm:main:intro:L4} if the constant $\Theta \neq 0$.
This constant is defined through an independent of~$\mu$ problem (see Section~\ref{sec:inner}). It is beyond the scope of this work to check that $\Theta \neq 0$. However, some remarks should be made about its value. First, since the RPC3BP can be seen, close to $L_4$, as an integrable Hamiltonian $H_0$ plus a perturbation, $H_1$, the value of $\Theta$ depends ``analytically'' on $H_1$. Roughly speaking, hence,  for a generic~$H_1$, it does not vanish. Of course, it may happen that $\Theta$ is 0 for the particular case of the RPC3BP, but the numerical evidence (going back to Str\"omgen) suggests it does not. Second, it seems feasible to prove the non-vanishing character of $\Theta$ by extending the computer assisted methods in~\cite{BGG25} to the present situation.

\subsection{The Hamiltonian Hopf bifurcation and exponentially small phenomena}
\label{sec:exponencialment_petit}

When the Hamiltonian Hop bifurcation takes place in a family of Hamiltonians depending on a parameter $\nu$ at $\nu=0$, if the singularity belongs to the open set in Theorem~\ref{thm:main:intro}, the versal normal form up to order 4 for subcritical values of $\nu$ close to 0 possesses a 2-dimensional homoclinic loop to the origin. 
Furthermore, since, for subcritical values of $\nu$ close 0, the imaginary part of the eigenvalues is 1 while their real part is small, two time scales appear. In this situation, it is well known that the versal formal normal form of the family is integrable up to any order. This implies that the stable and unstable invariant manifolds of the singularity coincide up to any order in $\nu$. Hence, determining if the system possesses transversal homoclinic orbits is a ``beyond all orders'' problem. 

With the aid of the normal form, the family can be written as an integrable Hamiltonian plus a perturbation. One can then try to compute the associated Melnikov function. It is immediate to check that this function provides an exponentially small in $\nu$ prediction of the difference between the invariant manifolds. However, our work here shows that this prediction is generically wrong: unless some additional smallness condition is assumed, Melnikov method fails. In particular, Melnikov method fails to predict the true splitting of the invariant manifolds of $L_4$ and $L_5$ in the RPC3BP. This is in contrast to what happens to the invariant manifolds of the parabolic infinity in the RPC3BP (see~\cite{GuardiaMS16}), where, although the problem has also two time scales, the authors were able to prove that the Melnikov function indeed describes the difference between the manifolds.

After the seminal ideas in~\cite{Lazutkin84}, rigorously applied later in~\cite{Gelfreich99} to the splitting of separatrices in the Chirikov standard map, the most successful way to tackle with the exponentially small splitting of invariant manifolds requires the use of complex extensions of the manifolds and, depending on the application, the study of the \emph{inner equation}. There is wide amount of literature. In the present paper, we will use the approach in~\cite{Sauzin01,LochakMS03} to deal with the Hamilton-Jacobi equation. We refer to~\cite{BaldomaFGS11} and the references therein for the treatment in Hamiltonian systems and~\cite{Baldoma06} for the inner equation.

\subsection{Structure of the paper}

In Section~\ref{sec:homoclinic_loop}, we introduce the problem and present the main results. First we describe the versal normal form and perform some suitable scalings. This part is already proven~\cite{Meer85}, see also~\cite{BHH07}, but we include it for completeness. Next we compute the homoclinic connection of the normal form and claim our main theorem in its technical form, Theorem~\ref{thm:difference_between_manifolds}. Finally, in Section~\ref{sec:L4:normalform} we deduce Theorem~\ref{thm:main:intro:L4} as a corollary of Theorem~\ref{thm:difference_between_manifolds}.

The rest of the paper is devoted to the proof of Theorem~\ref{thm:difference_between_manifolds}, as follows. 
In Section~\ref{sec:prova_del_teorema_principal}, we collect the main technical results concerning the invariant manifolds and their difference and deduce Theorem~\ref{thm:difference_between_manifolds} from them.

Subsequent sections contain the proof of each step: outer approximation of the manifolds, extension, inner equation and, finally, difference of the manifolds.

\section*{Acknowledgements}
I.B. and D.S. have been partially supported by the grant PID-2021-122954NB-100 funded by
MCIN/AEI/10.13039/501100011033 and “ERDF A way of making Europe”.  P.M. has been partially supported by the Spanish MINECO-FEDER Grant PID2021-123968NB-100.
This work is also supported by the Spanish State Research Agency, through the Severo Ochoa and Mar\'ia de Maeztu Program for Centers and Units of Excellence in R\&D (CEX2020-001084-M).

\section{Homoclinic loop and transverse intersections}
\label{sec:homoclinic_loop}
In this section, we recall the versal normal form of the Hamiltonian Hopf bifurcation, which has been widely studied; see, for instance,~\cite{Meer85} and~\cite{Soko74}. Versal normal forms are used to address the regularity issues arising from bifurcations. Typically, in the study of dynamical systems near an equilibrium point, the classical approach involves finding a suitable linear change of coordinates to bring the associated linearized system into its Jordan canonical form. Subsequently, a sequence of nonlinear transformations is employed to put higher-order terms into their normal forms. 
In our case, as is often the case in parameter-dependent systems where the linearized system has multiple eigenvalues, the linear symplectic change of coordinates bringing the linear part into its Jordan canonical form is not continuous when the Hopf bifurcation family encounters a non-semisimple Hamiltonian with a $1:-1$ resonance. This occurs at a critical value of the bifurcation parameter, where the linear part at the origin becomes a non-diagonalizable matrix with a pair of purely imaginary conjugate eigenvalues $\pm i\varpi$.
To address this issue in parameter-dependent systems, Arnold~\cite{A71} introduced the versal normal form. At the critical value of the parameter, the versal normal form coincides with the Jordan canonical form. However, for values near the critical value, additional terms are introduced to ensure the continuity of the linear transformation. 

In the versal normal form variables, at first order, a small homoclinic loop appears when the bifurcation parameter crosses the critical value. Following the ideas  in~\cite{MM03}, we perform a scaling normalizing the size of the homoclinic loop (see Proposition \ref{prop:main_rescaling}) and provide a suitable parameterization of this homoclinic loop in Lemma \ref{lem:homoclinic}. We also rewrite Theorem~\ref{thm:main:intro} in the scaled normal form variables (see Theorem~\ref{thm:difference_between_manifolds} in Section~\ref{subsec:hamilton-Jacobi}). 

Finally, in Section~\ref{sec:L4:normalform} we rewrite the Hamiltonian of the RPC3BP near $L_4$ (see~\eqref{H1}) in a versal normal form (see Theorem~\ref{Thm:NormalForm:L4}) and check that it satisfies the conditions of the main result.

From now on in this work, we will omit the dependence on the variables if there is no danger of confusion. 

\subsection{Normal form and rescaled Hamiltonian} \label{sec:normalform}

We say that $\mathbf{H}_0$ is a non-semisimple Hamiltonian with a $1:-1$ resonance if the associated vector field $X_{\mathbf{H}_0}$ has an equilibrium point, which can be assumed to be the origin, with a non diagonalizable linear part and having eigenvalues $\pm i \varpi \in i \mathbb{R}$ of multiplicity $2$. In other words,   
\begin{equation}\label{difH0}
DX_{\mathbf{H}_0} (0) = \begin{pmatrix}
    0 & -\varpi & 0 & 0 \\ \varpi & 0 & 0 & 0 \\ -1 & 0& 0& -\varpi \\0 &-1 & \varpi & 0
\end{pmatrix}, \qquad \varpi>0.
\end{equation}
\begin{remark}
The usual canonical form for $DX_{\mathbf{H}_0}(0)$ (see~\cite{Meer85}) is 
\[
DX_{\mathbf{H}_0} (0) = \begin{pmatrix}
    0 & -\varpi & 0 & 0 \\ \varpi & 0 & 0 & 0 \\ -\iota & 0& 0& -\varpi \\0 &- \iota & \varpi & 0
\end{pmatrix}, \qquad \varpi>0, \;\iota=\pm 1.
\]  
Notice that, performing, if necessary the symplectic scaling 
$\big  (x_1,\iota x_2, y_1, \iota y_2 \big)$  
and changing time by $\iota t$, $DX_{\mathbf{H}_0} (0) $ is of the form~\eqref{difH0}.  
\end{remark}

Consider a one-parameter family of $2-$degrees of freedom Hamiltonians $\mathbf{H}_{\mu} : \R^2 \times \R^2 \to \R$ with the standard symplectic form $dx_1 \wedge dy_1 + dx_2 \wedge dy_2$. We assume that $\mathbf{H}_{\mu}$ has the origin as an equilibrium point for all the values of the parameter $\mu$, and that, when $\mu=0$, $\mathbf{H}_0$ is a non-semisimple Hamiltonian with $1:-1$ resonance. We point out that condition~\eqref{difH0} is equivalent to assuming that the quadratic part $\mathbf{H}^2_{0}$ of $\mathbf{H}_0$ has the following form 
\begin{equation*}
   \mathbf{H}^2_{0}(x_1, x_2, y_1, y_2) = \varpi(x_1y_2 - x_2y_1) + {1 \over 2} (x_1^2 + x_2^2), \quad \mbox{with $\varpi>0$}
\end{equation*}
 whose the semisimple part is $\varpi(x_1y_2 - x_2y_1)$ and the nilpotent part is ${1 \over 2} (x_1^2 + x_2^2) $. The Hamiltonian Hopf bifurcation occurs when the equilibrium point at the origin associated with $\mathbf{H}_\mu$ changes its linear stability when the parameter $\mu$ crosses the critical value $\mu=0$, see for instance in Figure~\ref{fig:vaps}. That is, the eigenvalues of $DX_{\mathbf{H}_\mu}(0)$ change from $\pm i \varpi_1, \pm i \varpi_2$, with $\varpi_1\neq \varpi_2$ to $\pm a(\mu) \pm i b(\mu)$ with $a(\mu) \neq 0$, $a(0)=0, b(0)=\varpi$ when $\mu$ evolves either from $\mu>0$
 to $\mu<0$ (as in Figure~\ref{fig:vaps}) or from $\mu<0$ to $\mu>0$. 

\begin{figure}[t]
\begin{center}
\begin{overpic}[width=0.6\textwidth ]{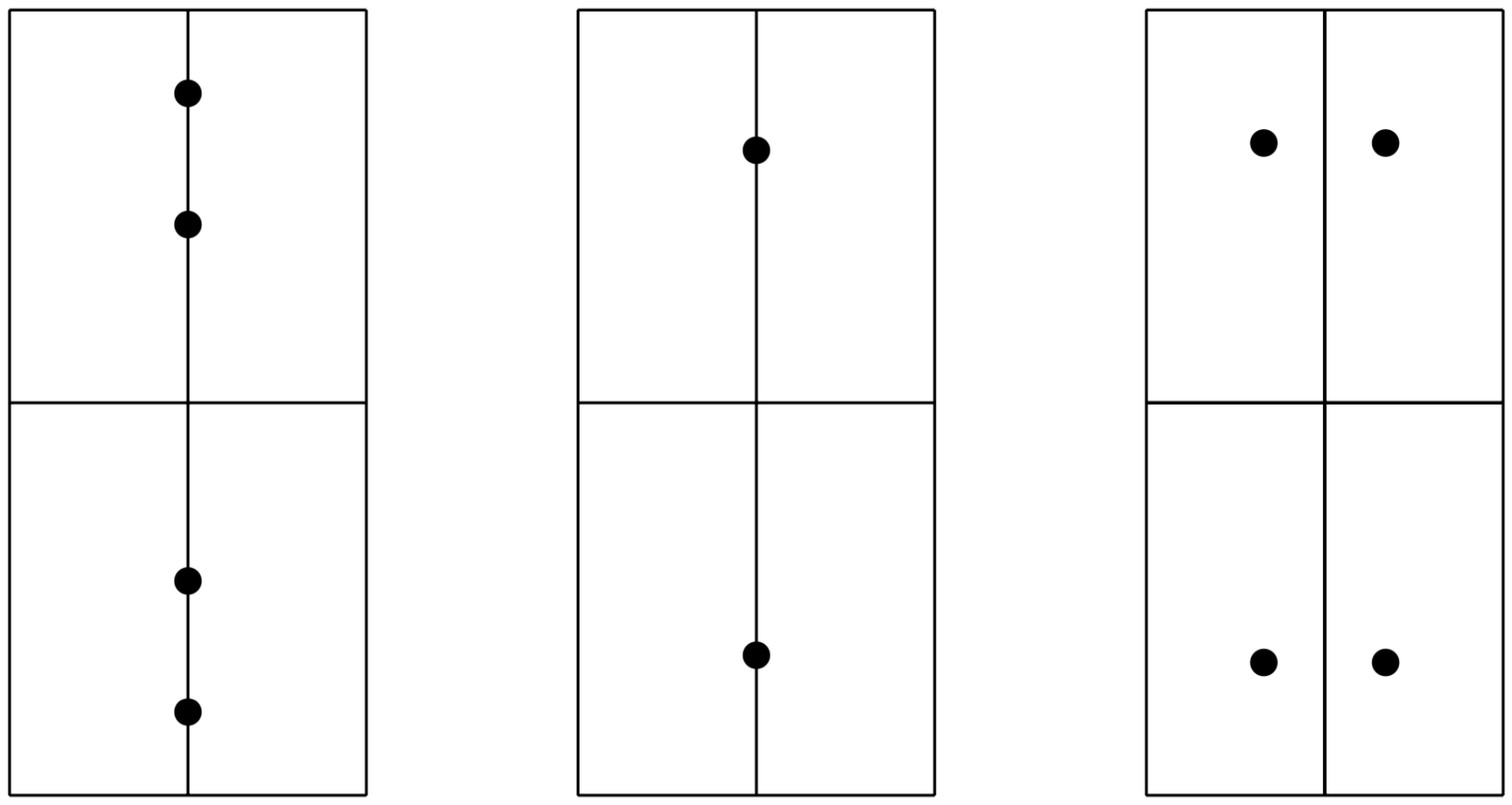}
		\put(45,4){\color{black} $\mu=0$ }
		\put(8,4){\color{black} $\mu>0$ } 
		\put(80,4){{\color{black} $\mu<0$ }}
	\end{overpic}
    \end{center}
    \caption{Evolution in the complex plane of the eigenvalues associated to the origin of $\mathbf{H}_{\mu}$ with respect to the bifurcation parameter $\mu$}
    \label{fig:vaps}
 \end{figure}

The following result has been proven in~\cite{Meer85}, see also~\cite{BHH07}. For a brief description of the proof, we refer to Section \ref{app:NF} of Appendix~\ref{HVNF}.  

\begin{theorem}\label{Thm:NormalForm} 
    Let $H$ be defined by $H(\cdot ;\mu)= \mathbf{H}_\mu$. 
    There exist $\mu_0>0$, a neighborhood $U$ of the origin in $\R^2 \times \R^2$ and an analytic family of analytic symplectic change of variables $\Phi:U \times [-\mu_0,\mu_0] \to  \R^2 \times \R^2$, such that 
$H\circ \Phi = \check{H}_0 + \check{H}_1$ with  
\[
\check{H}_0  =  \check{\omega}(\mu) S +  \frac{1}{2}  N + \frac{1}{2}  \nu_0(\mu) Q + {1 \over 4} \check{\gamma}(\mu)   Q^2+ {1 \over 4} \check{\alpha}(\mu)    S^2 + \frac{1}{2}\check{\beta}(\mu)    Q  \,  S  
\]
where $\check \omega = \varpi + \mathcal{O}(\mu)$, 
$
N  =  N(x_1,x_2)=x_1^2 + x_2^2$, $Q = Q(y_1,y_2)=y_1^2 + y_2^2$, $S  = S(x,y)=x_1 y_2-x_2y_1
$ 
and $\check{H}_1(x,y;\mu) = \mathcal{O}_6(x,y)$, uniformly for $\mu \in [-\mu_0,\mu_0]$.

In addition, {$H \circ \Phi$}, $\check \omega, \nu_0, \check \alpha$, $\check \beta$, and $\check \gamma$ are real analytic functions of $\mu \in (\mu_0,\mu_0)$.  

Finally, it is said that the family $H\circ \Phi$ undergoes a non degenerate Hopf bifurcation if $\nu_0(0)=0$, $\partial_\mu \nu_0(0)\neq 0$ and $\check \gamma(0)\neq 0$. 
\end{theorem}

 We point out that the Hamiltonian $\check H_0$ is invariant under the simultaneous rotations of the planes $(x_1, x_2)$ and $(y_1, y_2)$ induced by the Hamiltonian flow generated by $S$, that is $\{\check H_0,  S\}=0$, where $\{\cdot, \cdot\}$ denotes the usual Poisson bracket associated with the symplectic form $dx_1 \wedge dy_1 + dx_2 \wedge dy_2$. Furthermore, $\check{H}_0$ is an integrable with first integral $S$.

The conditions $\nu_0(0)=0$, $\partial_\mu \nu_0(0)\neq 0$ and $\check \gamma(0)\neq 0$, are generic non-degeneracy condition  ensuring that $H\circ \Phi$ unfolds a Hopf bifurcation at the origin.

\begin{remark}\label{new:H0:nu}
Assuming that $\partial_\mu \nu_0 (0) \neq 0$, taking $\nu=\nu_0(\mu)$, there exists an analytic function $\mu_0(\nu)$ , $\mu_0(0)=0$, that $\nu=\nu_0(\mu_0(\nu))$ if $\nu$ is small enough. Then $H\circ \Phi$ can be rewritten as 
$H\circ \Phi = \check{H}_0 + \check{H}_1$ with  
\[
\check{H}_0 =  \hat{\omega}(\nu) S +  \frac{1}{2}  N + \frac{1}{2}  \nu  Q + {1 \over 4} \hat{\gamma}(\nu)   Q^2+ {1 \over 4} \hat {\alpha}(\nu)    S^2 + \frac{1}{2}\hat{\beta}(\nu)    Q  \,  S  
\]
where $\hat \omega (\nu) = \check{\omega}(\mu(\nu))$, $\hat{\alpha}(\nu) = \check{\alpha}(\mu_0(\nu))$, $\hat{\beta}(\nu)=\check{\beta}(\mu_0(\nu))$ and $\hat{\gamma}(\nu) = \check{\gamma}(\mu_0(\nu))$.  
 
Notice that the characteristic polynomial of $DX_{ {\check{H}_0}}(0)$ can be explicitly computed as $
\lambda^4 + 2 (\hat{\omega}^2 + \nu) \lambda^2 + (\hat{\omega}^2-\nu)^2 $ so that the eigenvalues of $DX_{\check{H}_{0}}(0)$
satisfy
\[
\lambda^2 = -\hat{\omega}^2 - \nu \pm 2 \hat{\omega} \sqrt{\nu}.
\]
Therefore for $\nu>0$, the eigenvalues are purely imaginary 
$\pm i \sqrt{\hat{\omega}^2 + \nu \mp 2 \hat \omega\sqrt{\nu} }$ and when $\nu<0$, the eigenvalues are complex conjugated of the form $\pm \sqrt{-\nu} \pm i \hat \omega$. 
\end{remark}

So, from now on, we will restrict the analysis to the case $\nu<0$.  
If $\nu >0$, the origin is an elliptic point.

As mentioned before, the truncated Hamiltonian $\check{H}_0$ is integrable because $S$ is a first integral in involution with $\check{H}_0$, see~\cite{MM03} for a detailed study. 
The origin is an equilibrium point with eigenvalues $\pm \sqrt{-\nu} \pm i \hat{\omega}$ and, for $\nu <0$, it has 2-dimensional invariant manifolds given by 
\[
\check{W}^{\bu,\bs}_0(0) = \{(x,y) \in \mathbb{R}^4\,:\, S (x,y)= \check{H}_0(x,y)=0\},
\]
that is
\begin{equation*}
\check{W}^{\bu,\bs}_0(0) 
= \left \{ (x,y) \in \mathbb{R}^4\,:\,S =0,\; 
N+\nu Q + \frac{1}{2} \hat \gamma(\nu) Q^2 =0 \right\}.
\end{equation*}

\begin{figure}[h]
    \subfloat{
         \includegraphics[width=0.45\textwidth]{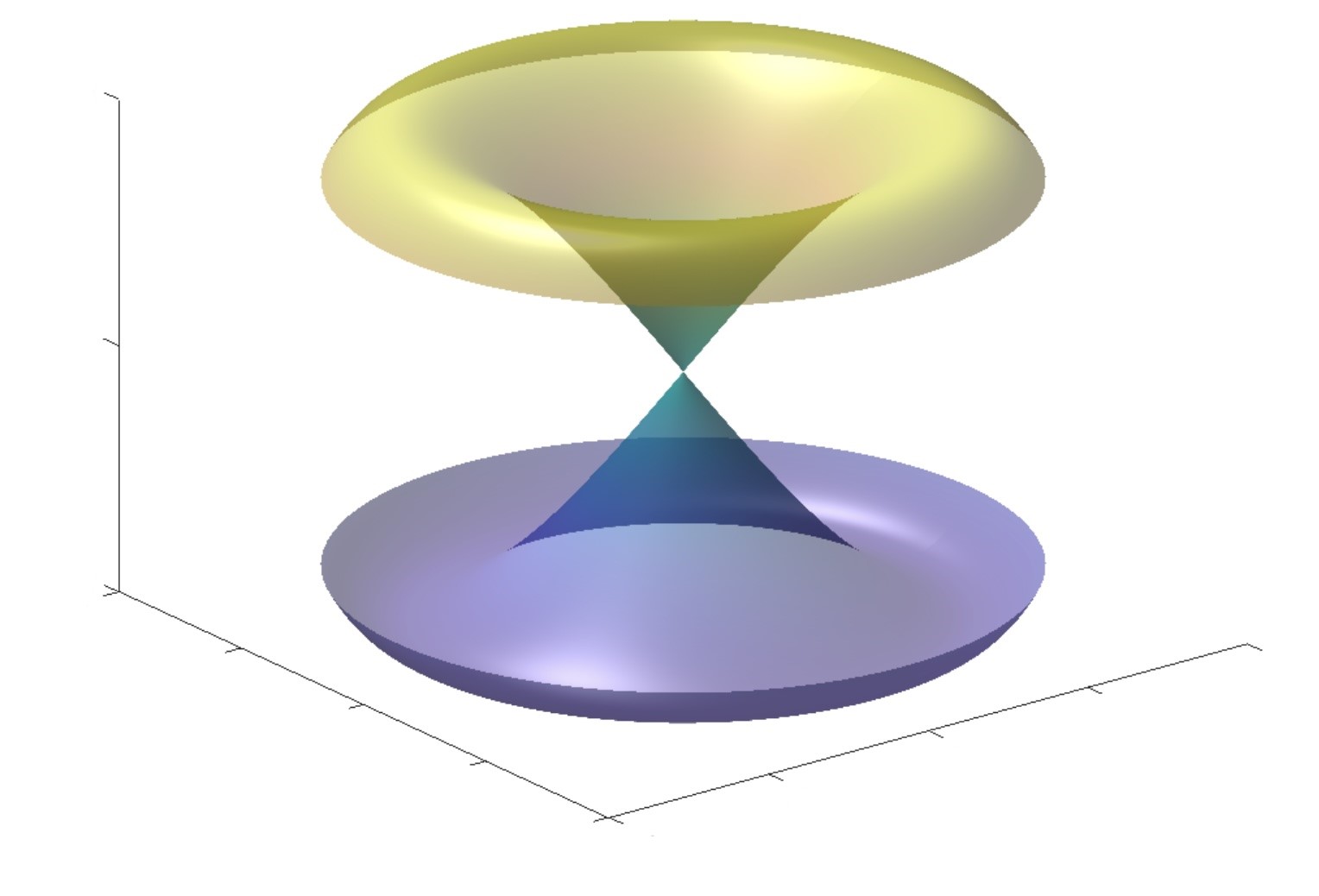}
         } \qquad 
          \subfloat{
         \includegraphics[width=0.45\textwidth]{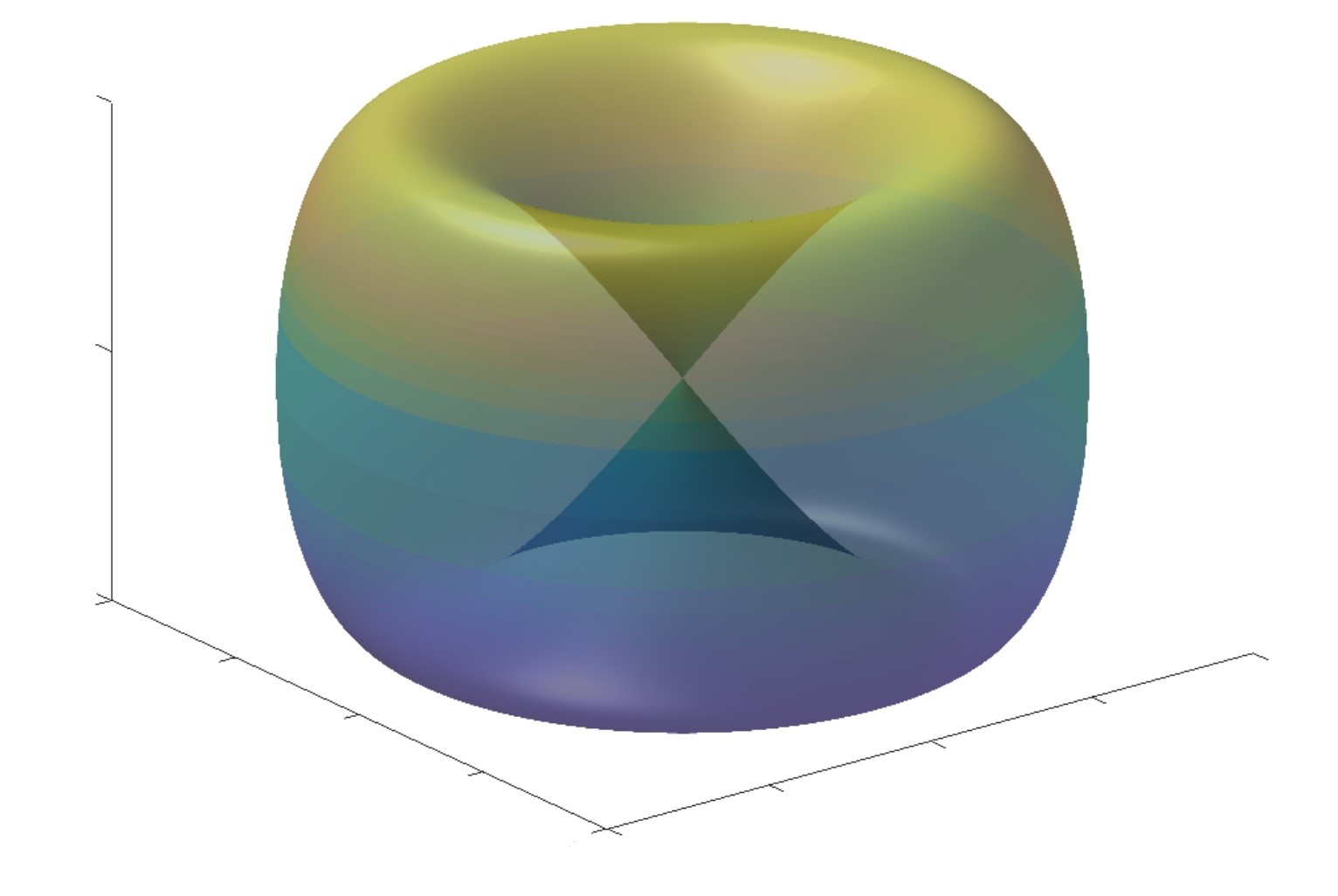}
         } 
\caption{The representation of the homoclinic surface $W^{\bu,\bs}_0(0)$. The axis are $(y_1,y_2, \sqrt{x^2_1+ x^2_2})$. On the left, it is represented by a piece of the homoclinic surface, just to clarify its geometry. On the right, the homoclinic surface is shown. }
\label{fig:homoclinica}
 \end{figure}

Clearly, a necessary condition for $\check{W}^{\bu,\bs}_0(0)$ to coincide, namely, to form a 2-dimensional homoclinic connection is (when $\nu<0$) 
\begin{equation}
\label{def:condicio_homoclinica}
\hat\gamma(0)>0
\end{equation}
(see Figure~\ref{fig:homoclinica}). Moreover, for $(x,y) \in \check W^{\bu,\bs}(0) $, 
$\nu  Q + \frac{1}{2} \hat \gamma(\nu)  Q^2 \leq 0$ and then, taking into account that $\hat \gamma(0)>0$, there exists a constant $c_1$ such that, if $\nu$ is small enough 
\[
0<Q = y_1^2 + y_2^2 \leq |\nu| 2 \hat \gamma^{-1} (\nu) \leq c_1^2 |\nu|
\]
that implies   
\[
N=x_1^2 + x_2^2 = |\nu |  Q - \frac{1}{2} \hat \gamma(\nu) Q^2 \leq   c_1^2 |\nu|^2.
\]
As a consequence, if $(x,y) \in \check{W}^{\bu,\bs}_0(0)$ the sizes of $x$ and $y$ are qualitative different. More precisely $\sqrt{x_1^2+ x_2^2} \leq {c_1} |\nu|$ and $\sqrt{y_1^2 + y_2^2}\leq {c}_1 |\nu|^{1/2}$.  
A suitable scaling, see~\cite{Meer85}, that normalizes the size of this homoclinic loop is done in order to study $\check{W}_0^{\bu,\bs}(0)$. This is summarized in the next proposition.
 

\begin{proposition}
\label{prop:main_rescaling} Assume that $\hat{\gamma}(0)>0$. 
Let $\varepsilon = \varepsilon(\nu)$, $a = a(\nu)$ be defined through
\begin{equation}
\label{def:epsilon}
\varepsilon^2 = -\frac{2 \nu}{\hat{\gamma}(\nu)}, \qquad a^2 =\frac{1}{2}\hat{\gamma}(\nu),
\end{equation}
with $\hat{\gamma}$ defined by Theorem~\ref{Thm:NormalForm} and Remark~\ref{new:H0:nu}. 
There exists $\varepsilon_0>0$ such that, for $\varepsilon\in (0,\varepsilon_0)$, the function $-\nu(\varepsilon)= \frac{1}{2} \varepsilon^2 \hat\gamma(0) + \mathcal{O}(\varepsilon^4)$ satisfying that $-\nu(\varepsilon) = {1 \over 2}\hat\gamma(\nu(\varepsilon)) \varepsilon^2$ is analytic\footnote{This fact follows straighforwardly using that $\hat\gamma(0)\neq 0$ and the implicit function theorem}. 

Consider the conformally symplectic scaling
\[
\widetilde \Phi(x,y) = (a(\nu(\varepsilon))\varepsilon^{2} x, \varepsilon y),
\]
where $a(\nu(\varepsilon))= \sqrt{{1 \over 2}\hat \gamma(\nu(\varepsilon))}$. Then

\begin{equation}
\label{def:rescaledHamiltonian}
H \circ \Phi \circ \widetilde \Phi =  \frac{1}{2}\hat{\gamma} (\nu(\varepsilon) )\varepsilon^4 \mathcal{H} := \frac{1}{2}\hat \gamma(\nu(\varepsilon) )\varepsilon^4 \big ( H_0 + H_1 \big ),
\end{equation}
with $H\circ \Phi$ the Hamiltonian in Remark \ref{new:H0:nu} (see also Theorem~\ref{Thm:NormalForm}), 
\begin{equation}
\label{def:H0H1}
H_0 ={\omega \over \varepsilon} S + {1 \over 2}\left(N -  Q +  Q^2\right),\qquad H_1 =  \alpha \varepsilon^2 S^2 +  \beta \varepsilon QS +  \widehat H_1,
\end{equation}
with  $N=N(x_1,x_2)$, $S=S(x,y)$, $Q=Q(y_1,y_2)$ defined as
\begin{equation}\label{def:NQS}
 N(x_1,x_2)=x_1^2 +x_2^2,\qquad S(x,y)=x_1y_2 - x_2 y_1,\qquad Q(y_1,y_2)=y_1^2+ y_2^2,
\end{equation}
the new values of $\omega = \omega(\varepsilon)$, $\alpha=\alpha(\varepsilon)$, $\beta=\beta(\varepsilon)$ are 
\begin{equation}\label{def:omegaalfabeta}
\begin{aligned}
\omega & = \omega(\varepsilon) = \sqrt{\frac{2}{\hat{\gamma}(\nu(\varepsilon))}} \hat \omega(\nu(\varepsilon)) = \sqrt{\frac{2}{\hat \gamma(0)}} \hat \omega (0) + \mathcal{O}(\varepsilon^2), \\
\alpha & =  \alpha(\varepsilon) = \frac{\hat{\alpha}(\nu(\varepsilon))}{4}  = \frac{\hat{\alpha}(0)}{4} + \mathcal{O}(\varepsilon^2),\\
\beta &=   \beta(\varepsilon)= \frac{\hat{\beta}(\nu(\varepsilon))}{\sqrt{2 \hat \gamma( \nu (\varepsilon))}} =  \frac{\hat{\beta}(0)}{\sqrt{2\hat{\gamma}(0)}} + \mathcal{O}(\varepsilon^2),
\end{aligned}
\end{equation}
and $
\widehat H_1 (x,y;\varepsilon) = \varepsilon^{-4}\tilde H_1 (\varepsilon^2 x, \varepsilon y;\varepsilon),
$
with $\tilde H_1$ analytic in a neighborhood of $(x,y,\varepsilon)=(0,0,0)$,
satisfying $\tilde H_1(x,y;\varepsilon) = \Oo_6(x,y)$ uniformly in $\varepsilon$.
\end{proposition}
\begin{proof}
 Let $a,\varepsilon$ defined by~\eqref{def:epsilon}. We emphasize that 
\begin{equation}\label{def:a:varepsilon}
\widetilde{a}(\varepsilon) := a( \nu(\varepsilon)) = \sqrt{\frac{1}{2} \hat\gamma(\nu(\varepsilon))} = \sqrt{ \frac{1}{2} \hat\gamma(0)} + \mathcal{O}(\varepsilon^2).
\end{equation}

In addition, we point out that 
$\widetilde{\Phi}$ is a symplectic change of coordinates with multiplier $a^{-1}\varepsilon^{-3}$. In the new variables, the Hamiltonian $H\circ \Phi = \check{H}_0 + \check{H}_1$ in Remark~\ref{new:H0:nu} (see also Theorem~\ref{Thm:NormalForm}) takes the following form   
\begin{align*}
\widehat{H}(x,y;\varepsilon) =& H \circ \Phi \circ \widetilde{\Phi}(x, y ;\nu)= \widetilde{a}(\varepsilon) \varepsilon^3 \hat \omega S + \frac{1}{2} \widetilde{a}^2(\varepsilon)  \varepsilon^4 N +   \frac{1}{2} \varepsilon^2 \nu(\varepsilon)Q + {1 \over 4} \varepsilon^4 \widetilde{\gamma}(\varepsilon)  Q^2 \\
&+ {1 \over 4}\widetilde a^2(\varepsilon)\varepsilon^6 \widetilde{\alpha}(\varepsilon)S^2 + {1 \over 2}\varepsilon^5 \widetilde a(\varepsilon)  \widetilde{\beta}(\varepsilon)QS + \check{H}_1(\widetilde a(\varepsilon) \varepsilon^2 x, \varepsilon y; \nu(\varepsilon))
\end{align*}
where $N,Q,S$ are defined in~\eqref{def:NQS}, we have denoted by $\widetilde{\gamma}(\varepsilon) = \hat\gamma(\nu(\varepsilon))$ and analogously for $\widetilde{\alpha}(\varepsilon), \widetilde{\beta}(\varepsilon)$. 
By rescaling time and using that $  \widetilde{a}^2(\varepsilon) \varepsilon^4 =  -\varepsilon^2 \nu(\varepsilon)= \frac{1}{2} \varepsilon^4 {\widetilde{\gamma}(\varepsilon)}$ (see~\eqref{def:a:varepsilon}) the previous Hamiltonian can be written as follows
\begin{align*}
\mathcal{H}(x,y;\varepsilon) =&  \frac{\hat \omega}{\widetilde{a}(\varepsilon) \varepsilon} S + \frac{1}{2} N - \frac{1}{2}Q + \frac{1}{2} Q^2 + \varepsilon^2 \frac{\widetilde{\alpha}(\varepsilon)}{4}S^2+ \varepsilon \frac{\widetilde{\beta}(\varepsilon)}{2\widetilde{a}(\varepsilon)}QS \\ & + \frac{1}{ \widetilde{a}^{2}(\varepsilon)} \varepsilon^{-4} \check{H}_1(\widetilde{a}(\varepsilon) \varepsilon^2 x, \varepsilon y; \nu(\varepsilon)  ).
\end{align*}
Notice that by~\eqref{def:a:varepsilon}, it is clear that the new parameters $\omega(\varepsilon)$, $\alpha(\varepsilon)$
and $\beta(\varepsilon)$ are defined by~\eqref{def:omegaalfabeta}. The proof is complete  taking 
\[
\widehat{H}_1(x,y;\varepsilon)=\varepsilon^{-4} \tilde{H}_1({x}, {y};\varepsilon) = \frac{\varepsilon^{-4}}{\widetilde{a}^{2}(\varepsilon) }\check{H}_1 (\widetilde{a}(\varepsilon) {x}, {y};\varepsilon) 
\]
and recalling that by Theorem~\ref{Thm:NormalForm}, $\check{H}_1(x,\eta;\nu) = \mathcal{O}_6 (x,\eta)$ uniformly in $\nu$. 
\end{proof}

\subsection{Transverse homoclinic points in the Hopf bifurcation. } \label{subsec:hamilton-Jacobi}

By Proposition~\ref{prop:main_rescaling}, if $\varepsilon \neq 0$ is small enough, the origin is a hyperbolic equilibrium of $\mathcal{H}$
in~\eqref{def:rescaledHamiltonian} with non real multipliers.  The rescaling given by Proposition~\ref{prop:main_rescaling} provides an integrable first-order Hamiltonian, $H_0$. 
The next lemma provides a suitable parametrization of the invariant manifolds of the origin of $H_0$.
\begin{lemma}
\label{lem:homoclinic}
The origin is a hyperbolic equilibrium of $H_0$ with 2-dimensional stable and unstable manifolds expressed as
\begin{equation}\label{defWusH0}
{W}^{\bu,\bs}_0(0) 
= \left \{ (x,y) \in \mathbb{R}^4\,:\, {S} =0,\; 
 {N} - {Q} + {Q}^2 =0 \right\}
\end{equation}
which coincide. The function $\Gamma_0:\R \times \T \to \R^4$ defined by
\[
\Gamma_0 = \begin{pmatrix} \delta_0 \\\gamma_0
\end{pmatrix},
\]
with
\[
\gamma_0(u,\theta) = \begin{pmatrix} r(u) \cos \theta \\ r(u)\sin \theta \end{pmatrix}, \qquad 
\delta_0(u,\theta) = \begin{pmatrix} R(u) \cos \theta\\ R(u) \sin \theta\end{pmatrix},
\]
and
\[
r(u) = \frac{1}{\cosh u}, \qquad R(u) = \frac{\sinh u}{\cosh^2 u}  = - \dot r(u),
\]
satisfies ${W}^{\bu,\bs}_0(0) = \Gamma_0(\mathbb{R} \times \mathbb{T})$.
\end{lemma}
\begin{proof}
The proof of this lemma is elementary by considering the polar symplectic change of coordinates 
\begin{equation}\label{changepolarsymplectic}
\begin{aligned}
&\mathrm{Pol} : \R^2 \times \R^+_* \times \T \to \R^2 \times \left(\R^2 \setminus \{(0,0)\}\right)\\
&\mathrm{Pol}(R,G,r, \theta) =  \left (R\cos \theta - {G\over r} \sin \theta, R \sin \theta + {G \over r}\cos\theta, r \cos \theta, r \sin \theta \right ).
\end{aligned}
\end{equation}
This change transforms the symplectic form $dx \wedge dy$ into $dR \wedge  dr + dG \wedge d\theta$. Moreover, $S=-G$ and in these new variables, the Hamiltonian $H_0$ is rewritten as
\begin{equation*}
\mathrm{H}_0(R, G, r, \theta)=- {\omega \over \varepsilon}G + {1\over 2} \left(R^2 + {G^2 \over r^2} -   r^2 +  r^4\right),
\end{equation*}
with equations of motion given by
\begin{equation*}
\dot R = \partial_r \mathrm{H}_0, \qquad \dot G = \partial_\theta \mathrm{H}_0=0, \qquad \dot r = - \partial_R \mathrm{H}_0, \qquad \dot \theta = - \partial_G \mathrm{H}_0.
\end{equation*}
We can describe $W^{\bu,\bs}_0(0)$ as 
\begin{equation}
\label{Homo}
W^{\bu,\bs}_0(0) = \{ G=0,\;R^2 -r^2 + r^4 =0, \; \theta \in \mathbb{T}\}.
\end{equation}
We point out that, because $G$ is an integral and we set $G=0$, the motion associated with the variable $\theta$ is linear. This means that for fixed $r$ and $R$ satisfying~\eqref{Homo}, we have a circle.  
Using that 
\begin{equation*}
\dot r = -R, \qquad R = \pm  r \sqrt{1 - r^2}
\end{equation*}
and fixing $r(0)=1$, we obtain $r(u)= (\cosh u)^{-1}$ and the result follows trivially undoing the change~\eqref{changepolarsymplectic}.
\end{proof}
\begin{remark}\label{rmk:homoclinic}
    We note that $r(u)$ is the solution of $\ddot{r} = r-2r^3$ satisfying $r(0)=1$, $\dot{r}(0)=0$.
\end{remark}
It is convenient to use $\Gamma_0$ as a starting approximation of the invariant manifolds of the origin associated with the full Hamiltonian $\mathcal{H}$. We remark that, since the manifolds are Lagrangian, there exist $\mathcal{S}^{\bu,\bs}(y)$, satisfying $\nabla \mathcal{S}(0) = 0$, such that $x = \nabla \mathcal{S}^{\bu,\bs}(y)$ are parametrizations of the unstable and stable manifolds whenever they can be expressed as graphs\footnote{Here $\nabla f$ denotes the gradient of a scalar function $f$.}. The functions $\mathcal{S}^{\bu,\bs}$ are characterized as the analytic solutions of the Hamilton-Jacobi equation
\begin{equation}
    \label{HJcartesiancoordinates}
    \mathcal{H}(\nabla \mathcal{S}(y), y;\varepsilon) = 0,
\end{equation}
such that $x= \nabla \mathcal{S}^{\bu}(y)$ and $x=\nabla \mathcal{S}^{\bs}(y)$ are tangent to the unstable and stable, respectively, subspaces at the origin. 

Since we want to find solutions of~\eqref{HJcartesiancoordinates} close to $\Gamma_0$, the parameterization of the invariant manifolds of the origin for $H_0$ ($W^{\bu,\bs}_0(0)$ in~\eqref{defWusH0}) provided in Lemma~\ref{lem:homoclinic}, we introduce
\begin{equation}
    \label{def:T}
    T^{\bu,\bs} = \mathcal{S}^{\bu,\bs} \circ \gamma_0.
\end{equation}
Then the Hamilton-Jacobi equation~\eqref{HJcartesiancoordinates} becomes
\begin{equation}
    \label{HJpolarcoordinates}
    \mathcal{H}((D\, \gamma_0)^{-\top}\nabla T, \gamma_0;\varepsilon) = 0.
\end{equation}
We are interested in solutions of~\eqref{HJpolarcoordinates} satisfying the boundary conditions
\begin{equation}\label{boundaryHJpolarcoordinates}
    \lim_{\Re u \to -\infty} \nabla T^{\bu}(u,\theta)=0, \qquad \lim_{\Re u \to +\infty} \nabla T^{\bs}(u,\theta)=0.
\end{equation}
Notice that a simple computation shows that the function
\begin{equation}
\label{def:T0}
    T_0(u,\theta) = \rho_0(u),
\end{equation}
where
$
\dot \rho_0 = - R^2
$,
satisfies that 
\[
\delta_0 = (D\, \gamma_0)^{-\top}\nabla T_0,
\]
that is, $\mathcal{S}_0 = T_0 \circ \gamma_0^{-1}$, provides through $x = \nabla \mathcal{S}_0(y)$ the parametrization of $W^{\bu,\bs}_0(0)$. The fact that $T_0$ (or $\mathcal{S}_0$) generates the invariant manifolds of the origin of $H_0$ implies that
\begin{equation}
    \label{HJpolarcoordinates_integrable}
    H_0((D\, \gamma_0)^{-\top}\nabla T_0, \gamma_0) = H_0(\delta_0,\gamma_0)=0.
\end{equation}

\begin{theorem}
    \label{thm:difference_between_manifolds}
    For any $0<u_0<u_0^*$, there exists $\varepsilon_0>0$ such that, for any $0 < \varepsilon  < \varepsilon_0$, equation~\eqref{HJpolarcoordinates} admits two real analytic solutions $T^{\bu,\bs}:[u_0,u_0^*]\times \T \to \R$ such that
    \[
 \Gamma^{\bu,\bs}(u, \theta) = \begin{pmatrix} (D\, \gamma_0)^{-\top}(u,\theta)\nabla T^{\bu,\bs}(u,\theta) \\ \gamma_0(u,\theta) \end{pmatrix}, \qquad (u,\theta) \in [u_0,u_0^*]\times \T,
 \]
 are parametrizations of the unstable and stable manifolds of the origin for the Hamiltonian $\mathcal{H}$.

 In addition, there exist a smooth function $a_0(\varepsilon) \in \R$ defined for $0<\varepsilon<\varepsilon_0$ and constants $a_1 , a_2  \in \R$ such that, for all $(u,\theta) \in [u_0,u_0^*]\times \T$, $0\le k, j \le 2$
    \[
    \partial_u^k \partial_\theta^j (T^\bu(u,\theta)- T^\bs(u,\theta) -a_0(\varepsilon))
    =  \varepsilon^{-3-k} e^{-{{\omega} \pi \over 2\varepsilon}} \left(f_1^{(k+j)} \left(\theta - {\omega \over\varepsilon} u\right)  + \Oo\left ( \frac{1}{|\log \varepsilon|} \right )\right)
    \]
where
\[
f_1(\sigma) = a_1  \cos \sigma + a_2 \sin \sigma.
\]
In particular,
letting $\nabla T^{\bu,\bs} = (\partial_u T^{\bu,\bs} , \partial_\theta T^{\bu,\bs} )^\top$, 
   \begin{multline*}
    \nabla T^\bu(u,\theta)- \nabla T^\bs(u,\theta) \\
    = \begin{pmatrix}
    \omega \varepsilon^{-4} e^{-{{\omega} \pi \over 2\varepsilon}} \left(a_1 \sin \left(\theta - {\omega \over\varepsilon} u\right) - a_2  \cos \left(\theta - {\omega \over\varepsilon} u\right) + \Oo\left ( {|\log \varepsilon|^{-1}} \right )\right) \\ \\
    \varepsilon^{-3} e^{-{{\omega} \pi \over 2\varepsilon}} \left(-a_1  \sin \left(\theta - {\omega \over\varepsilon} u\right) + a_2 \cos \left(\theta - {\omega \over\varepsilon} u\right) + \Oo\left (  {|\log \varepsilon|^{-1}} \right )\right)
    \end{pmatrix}.
    \end{multline*} 
\end{theorem}

\begin{proof}[Proof of Theorem~\ref{thm:main:intro}]
Assume that $a_1 a_2 \neq 0$. Gaiv\~ao, in~\cite{Gai10} showed that this happens generically. Now, since the unstable and stable invariant manifolds of the origin lie in the same energy level, if one of the components of $\nabla T^\bu(u,\theta)- \nabla T^\bs(u,\theta)$ vanishes, the other also does. It is clear that, if $a_1 a_2 \neq 0$, the leading term of the first component has two non-degenerate zeros for $\sigma = \theta-\omega u/\varepsilon$. An immediate application of the standard implicit function theorem implies that, for $\varepsilon$ small enough, the invariant manifolds must have transversal homoclinic orbits.
\end{proof}

\subsection{Homoclinic intersections around $L_4$ in the RPC3BP} \label{sec:L4:normalform}

This section is devoted to proving Theorem~\ref{thm:main:intro:L4} as a corollary of Theorem~\ref{thm:difference_between_manifolds}. 
We recall that the Hamiltonian of the RPC3BP near $L_4$ in a rotating framework is given by
\begin{equation}
\label{H1App}
H(q,p;\mu) = {|p|^2 \over 2} - q^{\top} \begin{pmatrix} 0 & 1 \\ -1 & 0\end{pmatrix} p -\left({1 - \mu \over |q + (\mu, 0)|}+{\mu \over |q - (1- \mu, 0)|}\right)
\end{equation}
where $(q,p) \in \R^2 \times \R^2$ are the position and momenta of the third body, $|\cdot|$ stands for the Euclidean norm and the symbol $\top$ for the transpose.

It is well known that the Lagrangian point $L_4$, at the top vertex of an equilateral triangle with base the segment between the primaries, has coordinates
\[
L_4= \left (\frac{1}{2} (1-2\mu), \frac{\sqrt{3}}{2}, -\frac{\sqrt{3}}{2}, \frac{1}{2} (1-2\mu )\right ).
\]
It is convenient to move $L_4$ to the origin. For this reason, we introduce the following symplectic change of coordinates defined in a small neighborhood of the origin and taking values in a small neighborhood of $L_4$
\begin{equation*}
 (Q, P) \longrightarrow (q,p)
\end{equation*}
such that 
\begin{equation*}
Q_1 = q_1 - {1 \over 2} (1 - 2\mu), \quad Q_2 = q_2 - {\sqrt{3} \over 2}, \quad P_1 = p_1 + {\sqrt{3} \over 2}, \quad P_2 = p_2 - {1 \over 2} (1 - 2\mu).
\end{equation*}
In these new variables, we can rewrite the Hamiltonian~\eqref{H1App} in the following form 
\begin{equation}
\label{H2}
\begin{aligned}
H(Q,P; \mu) =&   {P_1^2 + P_2^2 \over 2} - Q_1P_2 + Q_2P_1\\
&+{1 \over 8} Q_1^2 - {3 \sqrt{3} \over 4}(1 - 2\mu) Q_1 Q_2 - {5 \over 8} Q_2^2 + \mathcal{O}_3(Q,P;\mu),
\end{aligned}
\end{equation}
where the constant term is omitted.  We point out that $\mathcal{O}_3(Q,P;\mu)$ stands for terms,  depending on $\mu$,  of order at least $3$ in the new variables $(Q,P)$.  One can see that the linear part of the Hamiltonian system associated with the Hamiltonian~\eqref{H2} is given by
\begin{equation}
\label{A}
\begin{pmatrix}
0&1&1&0\\ 
-1&0&0&1\\ 
-{1 \over 4}&{3 \over 4}\left(\sqrt{3}  - 2\sqrt{3} \mu\right)&0&1\\
{3 \over 4}\left(\sqrt{3}  - 2\sqrt{3} \mu\right)&{5 \over 4}&-1&0
\end{pmatrix},
\end{equation}
with characteristic polynomial
\[
\lambda^4 + \lambda^2 + \frac{27}{4}\mu(1-\mu)=0.
\]

Let $\mu_1 = {1 \over 2} \left(1 - {1 \over 9}\sqrt{69} \right)$ be the Gascheau-Routh critical mass ratio, that is, the value such that $\frac{27}{4}\mu(1-\mu)=1$.  It is well known that at $\mu=\mu_1$, $L_4$ undergoes a Hamiltonian Hopf-zero bifurcation (see Figure~\ref{fig:vaps}). In other words, the matrix~\eqref{A} has purely imaginary eigenvalues for values of the mass ratio $\mu$ in the interval $0< \mu < \mu_1$. When $\mu = \mu_1$ the eigenvalues of the matrix are $\pm i {\sqrt{2} \over 2} $ and have multiplicity two and for $\mu_1 < \mu \le {1 \over 2}$, the eigenvalues are complex conjugated of the form $\pm a \pm ib$ with
\[
a=\frac{1}{2}\sqrt{-1 + \sqrt{27\mu(1-\mu)}}, \qquad b={1 \over 2} \sqrt{1 + \sqrt{27\mu(1-\mu)}}.
\]

We introduce the new parameters
\begin{equation*}
\nu = \nu_0(\mu):={1 \over 4} \left(1 - \sqrt{27\mu(1-\mu)}\right), \quad \hat{\omega} = \hat{\omega}(\nu):= \sqrt{{1 \over 2} -\nu}
\end{equation*}
so that the eigenvalues are $\pm a \pm i b$ with
\[
a=\sqrt{-\nu}, \qquad b= \hat{\omega}.
\]
We remark that $\nu_0(\mu_1) = 0$ and that $\mu >\mu_1$ corresponds to $\nu<0$, since 
\[
\partial_\mu \nu_0(\mu)  = -\frac{\sqrt{27}}{8 \sqrt{\mu(1-\mu)}} (1-2\mu)< 0, \qquad 0<\mu < \frac{1}{2}.  
\]
In addition, it is clear that 
\begin{equation*}
\nu =  -\frac{3\sqrt{69}}{8} (\mu- \mu_1) + \mathcal{O}((\mu-\mu_1)^2), \qquad \hat{\omega} = \frac{\sqrt {2}}{2} + \mathcal{O}(\mu- \mu_1).
\end{equation*}

The following proposition can be found in Section~$4$ of~\cite{Sch94}. For a brief idea of the proof, we refer to Section \ref{app:NF_L4} of Appendix~\ref{HVNF}.  


\begin{theorem}\label{Thm:NormalForm:L4} 
There exist $\nu_0<0$, a neighborhood $U$ of the origin in $\R^2 \times \R^2$ and an analytic family of conformally analytic symplectic change of variables $\Phi:U \times [\nu_0,0] \to  \R^2 \times \R^2$, with symplectic form $\frac{1}{2} d x \wedge dy$, such that $\Phi(0,\nu) = L_4$ and 
$H\circ \Phi = \check{H}_0 + \check{H}_1$ with  
\[
\check{H}_0 \circ \Phi = 2 \hat{\omega} S +   N + \nu   Q + {1 \over 2} \hat{\gamma}(\nu)   Q^2+ {1 \over 2} \hat{\alpha}(\nu)    S^2 + \hat{\beta}(\nu)   Q  \,  S  
\]
where 
$
N  =  N(x_1,x_2)=x_1^2 + x_2^2$, $Q= Q(y_1,y_2)=y_1^2 + y_2^2$, $S  = S(x,y)=x_1 y_2-x_2y_1
$ 
and $\check{H}_1(x,y;\nu) = \mathcal{O}_6(x,y;\nu)$, uniformly for $\nu \in [\nu_0,0]$.

In addition, {$H \circ \Phi$}, $\hat \alpha$, $\hat \beta$, and $\hat \gamma$ are real analytic functions of $\nu \in (\nu_0,0)$ and
\begin{equation*}
\begin{aligned}
\hat{\alpha}(\nu)&={-655+10\nu + 6496\nu^2-4960\nu^3 \over 216(1-2 \nu)(1-20\nu)(9-20\nu)}\\
\hat{\beta}(\nu)&={\sqrt{2-4\nu}(-515+6712\nu-13424\nu^2)\over 144(1-2 \nu)(1-20\nu)(9-20\nu)}\\
\hat{\gamma}(\nu)&={531-4586\nu+6932\nu^2+3776\nu^3-9920\nu^4 \over 216(1-2 \nu)(1-20\nu)(9-20\nu)}.
\end{aligned}
\end{equation*}
Therefore
\begin{equation*}
\hat{\alpha}(0) = -\frac{655}{1944} < 0, \qquad {\hat{\beta}(0) =-\frac{515 \sqrt{2}}{1944}< 0}, \qquad \hat{\gamma}(0) = \frac{531}{1944} > 0.
\end{equation*}
\end{theorem}
 
By Theorem~\ref{Thm:NormalForm:L4}, $\hat{\gamma}(0)>0$ and therefore we are on the conditions of Proposition~\ref{prop:main_rescaling}. The new Hamiltonian $H\circ\Phi \circ \widetilde{\Phi}$ is given by 
\begin{equation*}
H \circ \Phi \circ \widetilde \Phi =  \frac{1}{2} \hat{\gamma} (\nu(\varepsilon) )\varepsilon^4 \mathcal{H} := \frac{1}{2} \hat \gamma(\nu(\varepsilon) )\varepsilon^4 \big ( H_0 + H_1 \big ),
\end{equation*}
with $H_0, H_1$ as in~\eqref{def:H0H1}
with  $N=N(x_1,x_2)=x_1^2 +x_2^2$, $S=S(x,y)=x_1y_2 - x_2 y_1$, $Q(y_1,y_2)=y_1^2+ y_2^2$, the reminder 
$\widehat H_1 (x,y;\varepsilon) = \varepsilon^{-4}\tilde H_1 (\varepsilon^2 x, \varepsilon y;\varepsilon).
$
In addition,   
the new values of $\omega = \omega(\varepsilon)$, $\alpha=\alpha(\varepsilon)$, $\beta=\beta(\varepsilon)$ are 
\begin{equation}\label{def:omegaalfabeta:L4}
\begin{aligned}
\omega & = \sqrt{\frac{2}{\hat{\gamma}(\nu(\varepsilon))}} \hat \omega (\nu(\varepsilon))= \sqrt{\frac{2}{\hat \gamma(0)}} \hat \omega(0) + \mathcal{O}(\varepsilon^2) = \sqrt{\frac{1944}{531}}   + \mathcal{O}(\varepsilon^2), \\
\alpha & =  \frac{\hat{\alpha}(\nu(\varepsilon))}{4}  = \frac{\hat{\alpha}(0)}{4} + \mathcal{O}(\varepsilon^2)= - \frac{655}{4\cdot 1944} + \mathcal{O}(\varepsilon)<0,\\
\beta &=   \frac{\hat{\beta}(\nu(\varepsilon))}{\sqrt{2 \hat \gamma( \nu (\varepsilon))}} =  \frac{\hat{\beta}(0)}{\sqrt{2\hat{\gamma}(0)}} + \mathcal{O}(\varepsilon^2) =-\frac{515 \sqrt{2}}{4 \cdot 1944} +\mathcal{O}(\varepsilon^2) <0. \\
\end{aligned}
\end{equation}

Hence, we can apply Theorem~\ref{thm:difference_between_manifolds} to $L_4$ in the RPC3BP to obtain Theorem~\ref{thm:main:intro:L4}.


\section{Proof of the main result}
\label{sec:prova_del_teorema_principal}


In this section, we prove Theorem~\ref{thm:difference_between_manifolds}. The strategy of the proof, which follows the seminal ideas by Lazutkin~\cite{Laz84} (see also~\cite{Laz05}), consists of several steps developed in the subsequent sections as follows:
\begin{itemize}
    \item The Hamilton-Jacobi equation~\eqref{HJpolarcoordinates} has meaning whenever $\delta_0,\gamma_0$ are defined and the stable and unstable manifolds are expressed locally as graphs of the gradient of a function. However, since from Lemma~\ref{lem:homoclinic}
    \begin{equation*}
(D\, \gamma_0)^{-\top}(u,\theta) = \begin{pmatrix}
    - \frac{1}{R(u)} \cos \theta & - \frac{1}{r(u)} \sin \theta \\
    - \frac{1}{R(u)} \sin \theta &  \frac{1}{r(u)} \cos \theta
\end{pmatrix},
\end{equation*}
    the Hamilton-Jacobi equation~\eqref{HJpolarcoordinates} is not defined for $u=0$. We first prove the existence of solutions of equation~\eqref{HJpolarcoordinates}$, T^{\bu,\bs}$, in simply connect complex domains avoiding $u=0$ and reaching $\mathcal{O}(\varepsilon)$-neighborhoods of the complex singularities $\pm i \frac{\pi}{2}$ of $\delta_0,\gamma_0$. In this first step, $T^{\bu,\bs}$ are not defined in a common real domain. This is the content of Section~\ref{sec:hamiltonJacobi}. 
    \item In Section~\ref{sec:furtherextension}, we extend the parameterizations obtained in the first step to a common complex domain that contains values of $u$ belonging to a real segment and that are $\mathcal{O}(\varepsilon)$-close to the singularities $\pm i \frac{\pi}{2}$.
    \item After that, we focus on the behavior of the parameterizations of the invariant manifold close to the singularities $\pm i \frac{\pi}{2}$, studying special solutions of the \textit{inner equation}, 
    a parameterless equation, which, eventually, will provide the first-order for the difference $T^{\bu}(u,\theta)- T^{\bs}(u,\theta)$, for $u$ real. 
    This study is performed in Section~\ref{sec:inner}. 
    \item In Section~\ref{sec:matching}, it is proven how well the special solutions of the inner equation approximate the parameterization $T^{\bu,\bs}(u,\theta)$ in the \textit{matching domains} which, with respect to $u$, contain a neighborhood $O(\varepsilon^\gamma)$-close to $\pm i \frac{\pi}{2}$.
    \item Finally, in Section~\ref{thm:distance}, we provide the asymptotic formula  of $T^{\bu}(u,\theta)- T^{\bs}(u,\theta)$ mainly following the strategy in~\cite{B12}.
\end{itemize}

From now on, we denote, for $\sigma >0$,
\begin{equation}\label{def:complex_torus}
\T_\sigma = \{\theta \in \C/ \Z : | \mathrm{Im} \,\theta| \le \sigma \}.
\end{equation}

\subsection{Hamilton-Jacobi equation close to the unperturbed homoclinic}\label{sec:hamiltonJacobi}

We recall that (see~\eqref{HJpolarcoordinates_integrable})
\[
    H_0((D\, \gamma_0)^{-\top}\nabla T_0, \gamma_0) = 0 
\]
with $T_0$ given in~\eqref{def:T0}, describing the unperturbed invariant manifold. Then, we introduce $T_1$ through $T = T_0+T_1$ and we first claim that $T$ is a solution of~\eqref{HJpolarcoordinates} if and only if $T_1$ is a solution of
\begin{equation}
    \label{HJpolarcoordinatesT1}
   \Lout T_1 = \Fout (T_1),
\end{equation}
where
\begin{equation}
    \label{def:operator_L}
   \Lout T_1 = \partial_u T_1 + \frac{\omega}{\varepsilon} \partial_\theta T_1,
\end{equation}
and
\begin{equation}
    \label{def:operator_F}
    \begin{aligned}
   \Fout( T_1)  = &\frac{1}{2R^2}(\partial_u T_1)^2 +  \frac{1}{2r^2}(\partial_\theta T_1)^2 + \alpha \varepsilon^2   (\partial_\theta T_1)^2 - \beta \varepsilon r^2 \partial_\theta T_1 \\ &+ \varepsilon^{-4} \tilde H_1 (\varepsilon^2( \delta_0 +(D\, \gamma_0)^{-\top}\nabla T_1), \varepsilon \gamma_0;\varepsilon).
\end{aligned}
\end{equation}

The fact that $T_1$ satisfies~\eqref{HJpolarcoordinatesT1} follows from a straightforward computation, as shown below. First,
we observe that, since Hamiltonian $H_0$ in~\eqref{def:H0H1} is quadratic in $x$, 
\begin{equation}
\label{eq:derivationHJ_H0}
\begin{aligned}
H_0 ( \delta_0 +(D\, \gamma_0)^{-\top}\nabla T),  \gamma_0)   =  &H_0(\delta_0,\gamma_0)
+ \partial_x H_0 ( \delta_0 ,  \gamma_0)   (D\, \gamma_0)^{-\top}\nabla T \\  &+
\frac{1}{2} D\, T (D\, \gamma_0)^{-1} \partial_x^2 N (\delta_0,\gamma_0) (D\, \gamma_0)^{-\top}\nabla T,
\end{aligned}
\end{equation}
where, we use $T$ as a variable instead of $T_1$.
Next, we recall the definition of $S$, $Q$, and $N$ in~\eqref{def:NQS}. In particular, $Q$ only depends on $y$. Taking into account that
\begin{equation}
\label{eq:Dgamma0_inverse_and_transpose}
(D\, \gamma_0)^{-\top}(u,\theta) = \begin{pmatrix}
    - \frac{1}{R(u)} \cos \theta & - \frac{1}{r(u)} \sin \theta \\
    - \frac{1}{R(u)} \sin \theta &  \frac{1}{r(u)} \cos \theta
\end{pmatrix},
\end{equation}
we obtain
\begin{equation}
\label{eq:Qondelta0gamma0}
Q \circ ( \delta_0 +(D\, \gamma_0)^{-\top}\nabla T),  \gamma_0) = Q\circ (\delta_0,\gamma_0)=r^2.
\end{equation}
Moreover, 
\begin{equation}
\label{eq:Sondelta0gamma0}
\begin{aligned}
S \circ (\  \delta_0 +(D\, \gamma_0)^{-\top}\nabla T),  \gamma_0) & = S \circ (  \delta_0 ,   \gamma_0) + \partial_x S ( \delta_0 ,  \gamma_0)   (D\, \gamma_0)^{-\top}\nabla T \\
& =  -   \partial_\theta T,
\end{aligned}
\end{equation}
where we used that $S \circ (\delta_0 ,   \gamma_0) = 0$ and $\partial_x^2 S = 0$.
As for $N$, we have that 
\[
\begin{aligned}
N \circ (  \delta_0 +(D\, \gamma_0)^{-\top}\nabla T),   \gamma_0)  = & N \circ ( \delta_0 ,  \gamma_0) + \partial_x N (  \delta_0 ,  \gamma_0)   (D\, \gamma_0)^{-\top}\nabla T \\
& 
+ \frac{1}{2} ((D\, \gamma_0)^{-\top}\nabla T)^\top \partial_x^2 N (\delta_0,\gamma_0) (D\, \gamma_0)^{-\top}\nabla T \\
 = & N \circ ( \delta_0 ,  \gamma_0) - 2\partial_u T
 + \frac{1}{R^2} (\partial_u T)^2 + \frac{1}{r^2} (\partial_\theta T)^2.
\end{aligned}
\]
Then, using that $H_0(\delta_0,\gamma_0) = 0$, from~\eqref{eq:derivationHJ_H0}, \eqref{eq:Qondelta0gamma0}, \eqref{eq:Sondelta0gamma0} and the expressions of  $N$ above, we deduce
\[
H_0 ( \delta_0 +(D\, \gamma_0)^{-\top}\nabla T),  \gamma_0) = 
- \frac{\omega}{\varepsilon} \partial_\theta T - \partial_u T
 + \frac{1}{2R^2} (\partial_u T)^2 + \frac{1}{2r^2} (\partial_\theta T)^2.
\]
Finally, from the definition of $H_1$ in~\eqref{def:H0H1}, \eqref{eq:Qondelta0gamma0} and~\eqref{eq:Sondelta0gamma0}, we obtain
\begin{multline*}
H_1(  \delta_0 +(D\, \gamma_0)^{-\top}\nabla T),   \gamma_0;\varepsilon)  \\ =
\alpha \varepsilon^2 (\partial_\theta T)^2  -\beta \varepsilon r^2 \partial_\theta T
+ \varepsilon^{-4} \tilde H_1 (\varepsilon^2( \delta_0 +(D\, \gamma_0)^{-\top}\nabla T), \varepsilon \gamma_0),
\end{multline*}
which proves the claim.

The unperturbed homoclinic $(\delta_0,\gamma_0)$ is, as a function of $u$, meromorphic with poles at $i\pi/2+j\pi$, $j\in \Z$. We will look for solutions $T_1^{u,s}$ of~\eqref{HJpolarcoordinatesT1} defined in domains which are $\Oo(\varepsilon)$-close to $\pm i \pi/2$, the closest singularities of $(\delta_0,\gamma_0)$ to the real line. 
More concretely, given $d \in \left({1 \over 4}, {1 \over 2}\right)$,  $\kappa >0$,  $\slope$, $\hslope \in \left(0, {\pi \over 2}\right)$, and $\varsigma>0$, we introduce the following complex domains:
\begin{equation}
\label{def:domains_outer_s}
\begin{aligned}
D^{\mathrm{out}, \bu}_{\kappa} &= \left\{ u \in \C : d {\pi \over 2} +  \tan \hslope\, \mathrm{Re}(u) \le |\mathrm{Im}(u)| \le {\pi \over 2} - \kappa \varepsilon - \tan \slope \,\mathrm{Re}(u) \right\},\\
D^{\mathrm{out}, \bu}_{\kappa, \varsigma} &= \left\{ u \in D^{\mathrm{out}, \bu}_{\kappa} : \mathrm{Re}(u)  \ge- \varsigma\right\},\\
D^{\mathrm{out}, \bu}_{\kappa, \infty} & = \left\{ u \in D^{\mathrm{out}, \bu}_{\kappa} : \mathrm{Re}(u) < -\varsigma\right\},
\end{aligned}
\end{equation}
(see Figure~\ref{fig:outer_domain}), where the superscript $-\bu-$ stands for unstable, and the corresponding stable ones
\begin{equation}
\label{def:domains_outer_u}
\begin{aligned}
D^{\mathrm{out}, \bs}_{\kappa} &= \left\{ u \in \C : -u \in  D^{\mathrm{out}, \bu}_{\kappa} \right\},\\
D^{\mathrm{out}, \bs}_{\kappa, \varsigma} &= \left\{ u \in \C : -u \in D^{\mathrm{out}, \bu}_{\kappa, \varsigma}\right\},\\
D^{\mathrm{out}, \bs}_{ \kappa, \infty} & = \left\{  u \in \C : -u \in D^{\mathrm{out}, \bu}_{ \kappa, \infty} \right\}.
\end{aligned}
\end{equation}

\begin{figure}[t]
\centering
\begin{overpic}[scale=0.043]{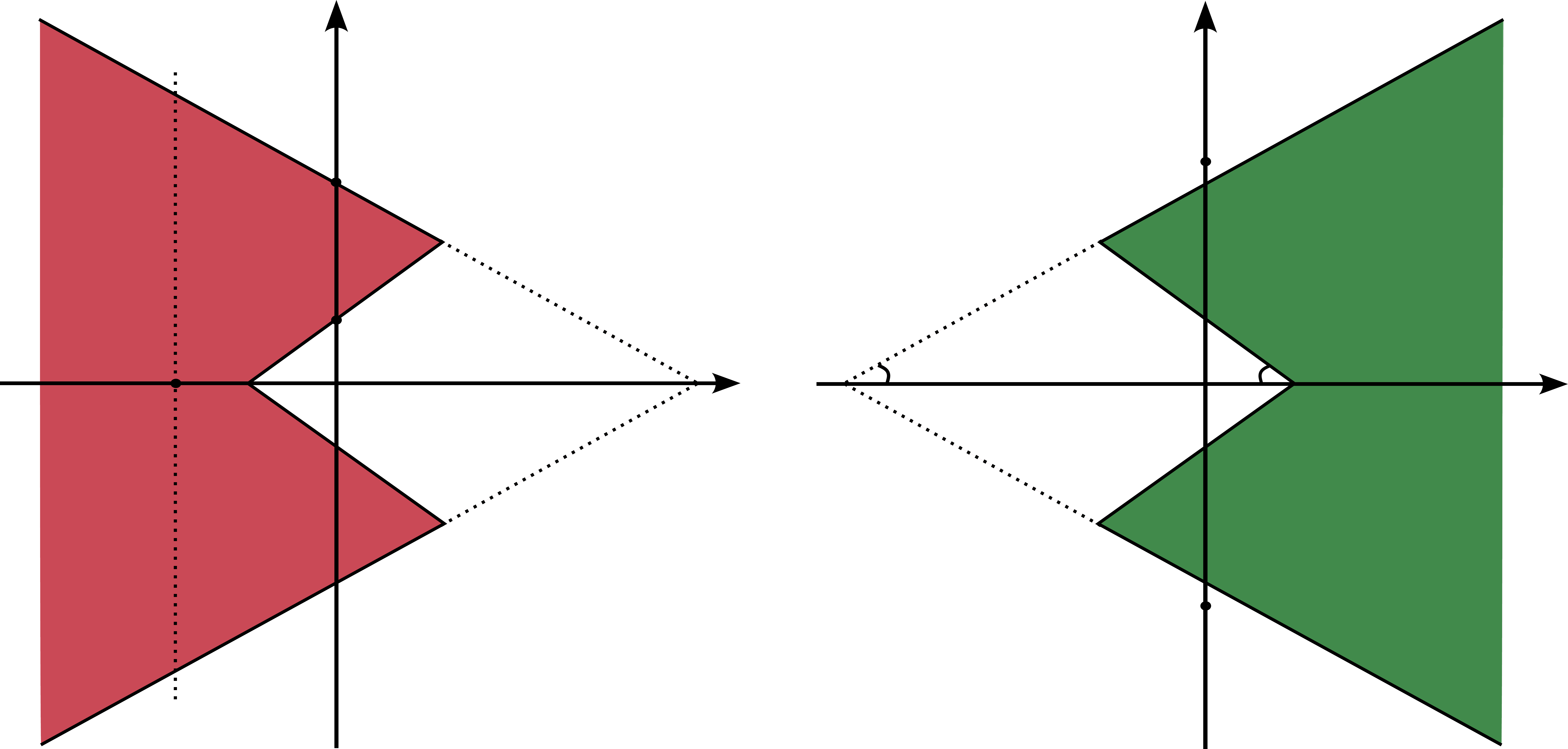}
    \put(23,48){$\mathrm{Im(u)}$}
    \put(78,48){$\mathrm{Im(u)}$}
    \put(99,20){$\mathrm{Re(u)}$}
    \put(59,24){$\slope$}
    \put(73.4,24){$\hslope$}
    \put(73.7,36.6){${\pi \over 2}$}
    \put(71,8.4){$-{\pi \over 2}$}
    \put(23.5,26,8){$ d{\pi \over 2}$}
    \put(24.5,35.3){${\pi \over 2} - \kappa \varepsilon$}
    \put(12,20.3){$-\varsigma$}
    \put(12,30){$D^{\mathrm{out}, \bu}_{\kappa, \varsigma}$}
    \put(2.7,8){$D^{\mathrm{out}, \bu}_{\kappa, \infty}$}
    \put(83,30){$D^{\mathrm{out}, \bs}_{\kappa}$}
\end{overpic}
\caption{{The domains $D^{\mathrm{out}, *}_{\kappa}$, $D^{\mathrm{out}, *}_{\kappa, \varsigma}$, and
$D^{\mathrm{out}, *}_{\kappa, \infty}$ with $*=\bu, \bs$ defined by~\eqref{def:domains_outer_s} and~\eqref{def:domains_outer_u}.}}
\label{fig:outer_domain}
\end{figure}

We remark that we do not write explicitly that these domains depend on $d$, $\slope$ and $\hslope$, because these values will remain fixed along the proof. The only \emph{true} parameter will be $\kappa$.
 
The following result states the existence of solutions for the Hamiltonian-Jacobi equation~\eqref{HJpolarcoordinatesT1} with boundary conditions~\eqref{boundaryHJpolarcoordinates}. Its proof is postponed to Section~\ref{sec:proof:outer}.
\begin{theorem}
    \label{thm:outer} 
     Let $\sigma >0$,  
     $d \in \left({1 \over 4}, {1 \over 2}\right)$,  $\slope$, $\hslope \in \left(0, {\pi \over 2}\right)$, and $\varsigma>0$ be fixed.      
     There exist $\kappa_{\mathrm{out}} > 0$, $\varepsilon_{\mathrm{out}}>0$ and $\Cout_{\mathrm{out}}>0$ such that,
for any  $\kappa \ge \kappa_{\mathrm{out}}$ and $0< \varepsilon <\varepsilon_{\mathrm{out}}$, the Hamilton-Jacobi equation~\eqref{HJpolarcoordinates} admits two analytic solutions $T^* : D^{\mathrm{out}, *}_{\kappa}  {\times \T_\sigma} \to \C$,  $* = \bu,\bs$, satisfying
\begin{align*}
\sup_{(u,\theta) \in D^{\mathrm{out}, *}_{\kappa, \infty}  {\times \T_\sigma}} \left| e^{6|\Re u|}  T_1^*(u,\theta)\right| + \sup_{(u,\theta) \in D^{\mathrm{out}, *}_{\kappa,\varsigma}  {\times \T_\sigma}} \left| \left(u^2+\frac{\pi^2}{4}\right)^5 T_1^*(u,\theta)\right|\le & \Cout_{\mathrm{out}} \varepsilon^2, \\ 
\sup_{(u,\theta) \in D^{\mathrm{out}, *}_{\kappa, \infty}  {\times \T_\sigma}} \left| e^{6|\Re u|}  \partial_u T_1^*(u,\theta)\right| + \sup_{(u,\theta) \in D^{\mathrm{out}, *}_{\kappa,\varsigma}  {\times \T_\sigma}} \left| \left(u^2+\frac{\pi^2}{4}\right)^6 \partial_u T_1^*(u,\theta)\right|\le & \Cout_{\mathrm{out}} \varepsilon^2,\\
\sup_{(u,\theta) \in D^{\mathrm{out}, *}_{\kappa, \infty}  {\times \T_\sigma}} \left| e^{6|\Re u|}  \partial_\theta T_1^*(u,\theta)\right| + \sup_{(u,\theta) \in D^{\mathrm{out}, *}_{\kappa,\varsigma}  {\times \T_\sigma}} \left| \left(u^2+\frac{\pi^2}{4}\right)^6 \partial_\theta T_1^*(u,\theta)\right|\le & \Cout_{\mathrm{out}} \varepsilon^3.
\end{align*}
\end{theorem}



As a consequence of Theorem~\ref{thm:outer}, notice that
\begin{equation}\label{thm:def_Gamma}
\Gamma^* = (\delta_0 + (D\gamma_0)^{-\top}\nabla T_1^*,\gamma_0)^{\top}, \qquad * = \bu,\bs
\end{equation}
are parametrizations of the invariant manifolds, defined in the domains $D^{\mathrm{out}, *}_{\kappa} {\times \T_\sigma}$.

\subsection{Further extension of the invariant manifolds}\label{sec:furtherextension}

Theorem~\ref{thm:outer} ensures the existence of generating functions, $T^{\bu}$ and $T^{\bs}$ for the unstable and stable manifolds defined on $D^{\mathrm{out}, \bu}_{\kappa} \times \T_\sigma $ and $D^{\mathrm{out}, \bs}_{\kappa} \times \T_\sigma$, respectively (see~\eqref{def:domains_outer_s}, \eqref{def:domains_outer_u} and Figure~\ref{fig:outer_domain}). We observe that 
\[
D^{\mathrm{out}, \bu}_{\kappa} \cap D^{\mathrm{out}, \bs}_{\kappa} \cap \R = \emptyset. 
\]
It is not possible to extend $T^{\bu}$ and $T^{\bs}$ through $u=0$ because equation~\eqref{HJpolarcoordinatesT1} is not defined for $u=0$. For this reason, we look for extensions of $T^{\bu}$ in~\eqref{thm:def_Gamma} defined in the following domain (see also Figure~\ref{fig:ext_domain}) 
\begin{equation}\label{def:domains_ext} 
\begin{aligned}
D_{\kappa}^{\mathrm{ext}} = \Big\{ u \in \C : &|\mathrm{Im}(u)| \le {\pi \over 2} - \kappa \varepsilon - \tan \tilde \slope \mathrm{Re}(u),\\
&|\mathrm{Im}(u)| \le {\pi \over 2} - \kappa \varepsilon + \tan \slope \mathrm{Re}(u),\\
&|\mathrm{Im}(u)| \ge d_0{\pi \over 2} - \tan \tilde\vartheta_1 \mathrm{Re}(u) \Big\}
\end{aligned}
\end{equation}    
with $\tilde \slope \in \left(\slope, {\pi \over 2}\right)$, $\tilde\vartheta_1 \in \left(0, {\pi \over 2}\right)$, $d_0 \in \left({1 \over 4}, {1 \over 2}\right)$ and $\slope$ the parameter  defined by~\eqref{def:domains_outer_s}.
\begin{figure}[t]
\centering
\begin{overpic}[scale=0.05]{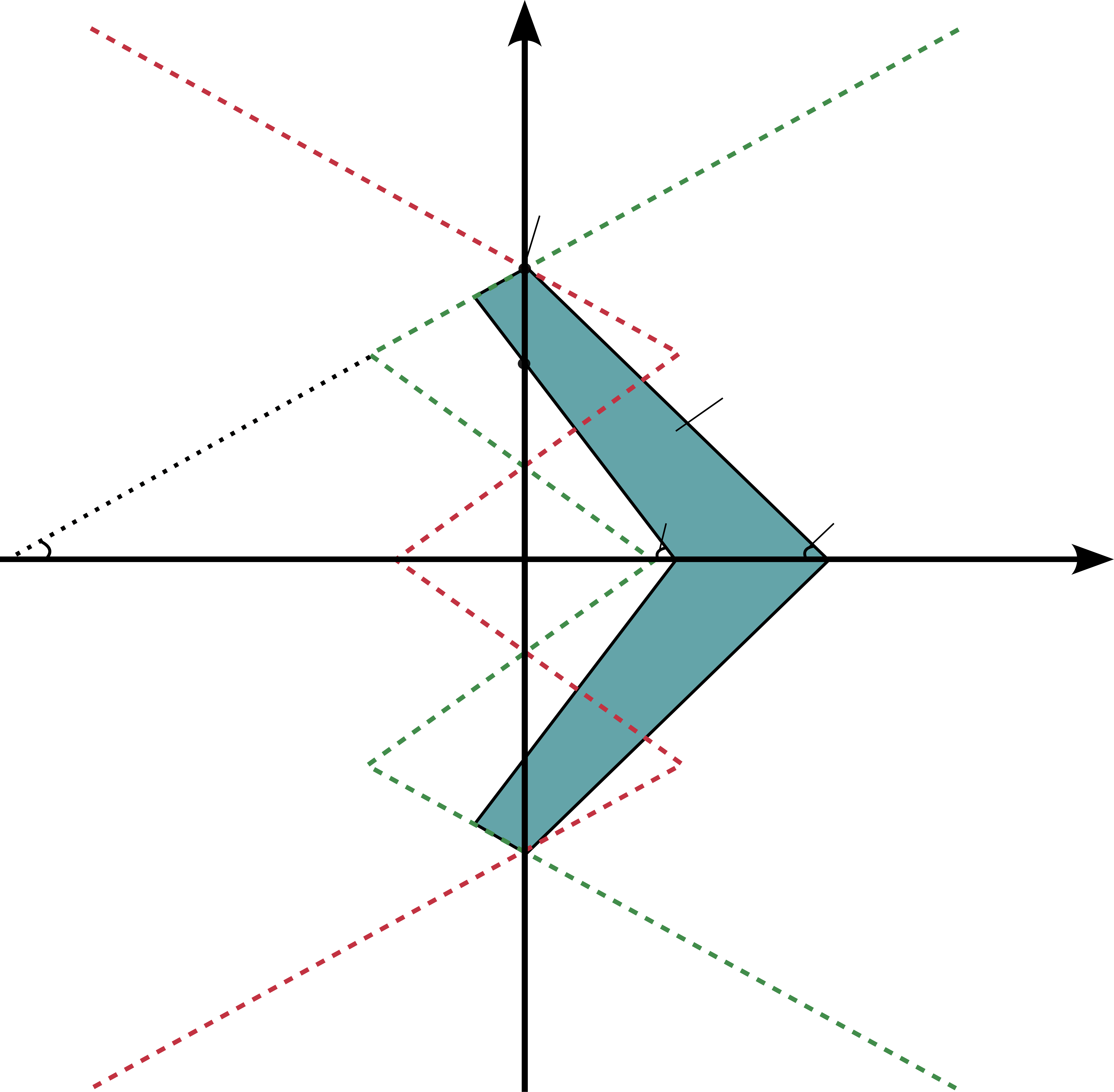}
    \put(48.7,99.5){$\mathrm{Im(u)}$}
    \put(101, 42){$\mathrm{Re(u)}$}
    \put(74.5, 51.5){$\tilde \slope$}
    \put(59.5, 51.5){$\tilde\vartheta_1$}
    \put(65.5, 63){$D_{\kappa}^{\mathrm{ext}}$}
    \put(37, 64){$d_0{\pi \over 2}$}
    \put(48, 82){${\pi \over 2} - \kappa \varepsilon$}
    \put(10, 49.3){$\slope$}
\end{overpic}
\caption{The dashed red-colored domain is $D^{\mathrm{out}, \bu}_{\kappa}$, while the dashed green-colored one is $D^{\mathrm{out}, \bs}_{\kappa}$. The domain in blue represents $D_{\kappa}^{\mathrm{ext}}$.}
\label{fig:ext_domain}
\end{figure}
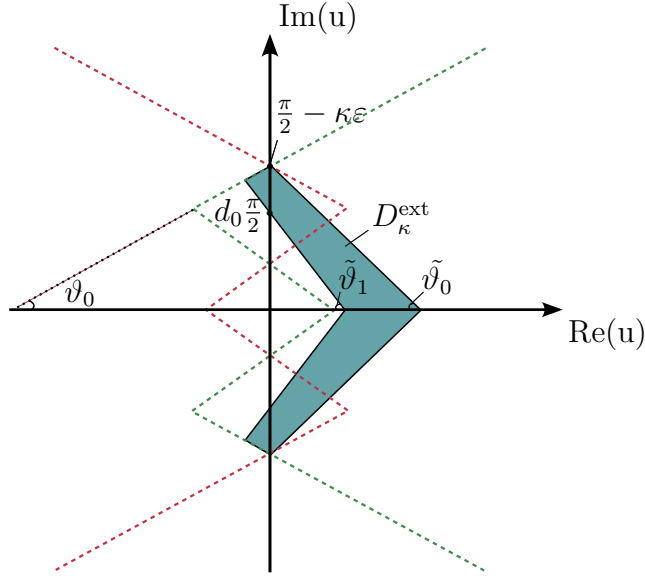

The following result provides the analytic extension to $D_{\kappa}^{\mathrm{ext}} \times \T_{\sigma \over 2}$ of $T^{\bu}$. 
\begin{theorem}\label{thm:ext} 
     Let $\sigma >0$,  
     $d \in \left({1 \over 4}, {1 \over 2}\right)$,  $\slope$, $\hslope \in \left(0, {\pi \over 2}\right)$ and   $\varsigma>0$ be fixed and consider $\kappa_{\mathrm{out}},\varepsilon_{\mathrm{out}}$ provided in Theorem~\ref{thm:outer}. 
     
There exist $d_0\in \left({1 \over 4}, {1 \over 2}\right)$, $\tilde \slope \in \left(\slope, {\pi \over 2}\right)$, $\tilde\vartheta_1 \in \left(0, {\pi \over 2}\right)$,  $0<\varepsilon_{\mathrm{ext}} \leq \varepsilon_{\mathrm{out}}$, $\kappa_{\mathrm{ext}}\geq \kappa_\mathrm{out}$ and a constant $c_{\mathrm{ext}}$ such that, for $\kappa\geq \kappa_{\mathrm{ext}}$ and $0<\varepsilon < \varepsilon_{\mathrm{ext}}$, 
\[
D_{\kappa}^{\mathrm{ext}} \subset D^{\mathrm{out}, \bs}_{\kappa} \cap D^{\mathrm{out}, \bu}_\kappa
\]
and $T_1^{\bu}$ in Theorem~\ref{thm:outer}, can be extended analytically on the domain $D_{\kappa}^{\mathrm{ext}} \times \T_{\sigma \over 2}$. Moreover, 
\begin{equation}\label{est:thm_ext}
\begin{aligned}
&\sup_{(u,\theta) \in D_{\kappa}^{\mathrm{ext}} \times \T_{\sigma \over 2}} \left| \left(u^2+\frac{\pi^2}{4}\right)^5 T_1^*(u,\theta)\right|  \le c_{\mathrm{ext}} {\varepsilon^2}\\
&\sup_{(u,\theta) \in D_{\kappa}^{\mathrm{ext}} \times \T_{\sigma \over 2}} \left| \left(u^2+\frac{\pi^2}{4}\right)^6 \partial_uT_1^*(u,\theta)\right|  \le c_{\mathrm{ext}} {\varepsilon^2}\\
&\sup_{(u,\theta) \in D_{\kappa}^{\mathrm{ext}} \times \T_{\sigma \over 2}} \left| \left(u^2+\frac{\pi^2}{4}\right)^6 \partial_\theta T_1^*(u,\theta)\right|  \le c_{\mathrm{ext}} {\varepsilon^3}
\end{aligned}
\end{equation}
for $* = \bu,\bs$.
\end{theorem}
Since, in order to have uniform bounds, we have to avoid a neighborhood of $u=0$, we are forced to work in 
not simply connected complex domains. This fact leads us to split the proof of Theorem~\ref{thm:ext} into the following four steps, which are proved in Section~\ref{ext:proof}.
\begin{enumerate}
    \item From graph to flow parameterization. 
Let $\Gamma^{\bu}$, defined as in~\eqref{thm:def_Gamma}, the parameterization of the unstable manifold provided by Theorem~\ref{thm:outer}. 
We look for a change of variables 
\begin{equation*}
(u, \theta) = (v + f_1(v, \varphi), \varphi + f_2(v, \varphi))
\end{equation*}
such that  
\begin{equation}
\label{def:ext_tilde_Gamma}
    \Ggf(v, \varphi) = \Gamma^{\bu}(v + f_1(v, \varphi), \varphi + f_2(v, \varphi))
\end{equation}
satisfies that $\zeta(t):=\Ggf\left (v+t,\varphi+\frac{\omega}{\varepsilon} t\right )$ is a solution of $\dot{\zeta} = X_{\mathcal{H}} (\zeta)$, where $X_{\mathcal{H}}$ is the vector field associated to the Hamiltonian $\mathcal{H}$ in~\eqref{def:rescaledHamiltonian}. In other words, we look for an analytic solution $f$ of the equation
\begin{equation}
\label{ext1:inv_eq}
\Lext  \big (\Gamma^{\bu} \circ (\mathrm{id} + f) \big )= X_\mathcal{H}\circ  \big  (\Gamma^\bu \circ (\mathrm{id} +f) \big )
\end{equation}
where $\Lext \Gamma = \partial_v   \Gamma + {\omega \over \varepsilon} \partial_\varphi  \Gamma$.

More specifically, let $\slope,\hat \slope$ be defined by~\eqref{def:domains_outer_s}. We define the domain (see Figure~\ref{fig:ext_domain_step1and2})
\begin{equation}\label{def:domains_ext_step1}
\begin{aligned}
\Dgf_{2\kappa} = \Big\{ u \in \C : &|\mathrm{Im}(u)| \le {\pi \over 2} - 2\kappa \varepsilon - \tan \slope \mathrm{Re}(u),\\
&|\mathrm{Im}(u)| \ge d_2{\pi \over 2} + \tan \hat \slope \mathrm{Re}(u),\\
&|\mathrm{Im}(u)| \le d_3{\pi \over 2} + \tan \hat \slope \mathrm{Re}(u) \Big\}
\end{aligned}
\end{equation}   
where $d < d_2 < d_3 <{1 \over 2}$, and ${\pi \over 2} - 2\kappa \varepsilon >0$. Note that $\Dgf_{2\kappa} \subset D^{\mathrm{out}, \bu}_{\kappa}$. In Section~\ref{sec:ext1:proof} we prove the following proposition. 
\begin{proposition}\label{prop:ext1}
Let $\sigma >0$,  
     $d \in \left({1 \over 4}, {1 \over 2}\right)$,  $\slope$, $\hslope \in \left(0, {\pi \over 2}\right)$ and   $d<d_2<d_3< \frac{1}{2}$ be fixed and consider $\kappa_{\mathrm{out}},\varepsilon_{\mathrm{out}}$ provided in Theorem~\ref{thm:outer}. 
     
    There exist $\kappa_1 \geq \kappa_{\mathrm{out}}$, $0<\varepsilon_1\leq \varepsilon_{\mathrm{out}}$ and constants $c_0, c_1,c >0$ such that, for all $\kappa\geq \kappa_1$ and $0 < \varepsilon <\varepsilon_1$, equation~\eqref{ext1:inv_eq} has an analytic solution $f = (f_1, f_2)^\top: \Dgf_{2\kappa} \times \T_{3 \sigma \over 4} \to \C^2$ satisfying 
    \begin{equation}\label{bound:f:ext1}
        \sup_{(v, \varphi) \in \Dgf_{2\kappa} \times \T_{3 \sigma \over 4}} \left|f_1(v, \varphi)\right| \le c_0 \varepsilon^2, \quad \sup_{(v, \varphi) \in \Dgf_{2\kappa} \times \T_{3 \sigma \over 4}} \left|f_2(v, \varphi)\right| \le c_1\varepsilon.
    \end{equation}
    Moreover, letting
\begin{equation}
\label{def:ext_tilde_Gamma:bis}
    \Ggf(v, \varphi) = \Gamma^{\bu}(v + f_1(v, \varphi), \varphi + f_2(v, \varphi)),
\end{equation}
the following bound holds
\begin{equation}\label{pro:ext1_TildeGamma1Est}
        \sup_{(v, \varphi) \in \Dgf_{2\kappa} \times \T_{3 \sigma \over 4}} |\Ggf(v, \varphi) - \Gamma_0(v, \varphi)|\le c  \varepsilon.
    \end{equation}
    where $\Gamma_0$ is the unperturbed homoclinic defined by Lemma \ref{lem:homoclinic}. 
\end{proposition}
\item Extension of the flow parameterization. Let $D^{\mathrm{fl}}_{3 \kappa}$ be the following domain (see also Figure \ref{fig:ext_domain_step1and2}).
\begin{figure}[t]
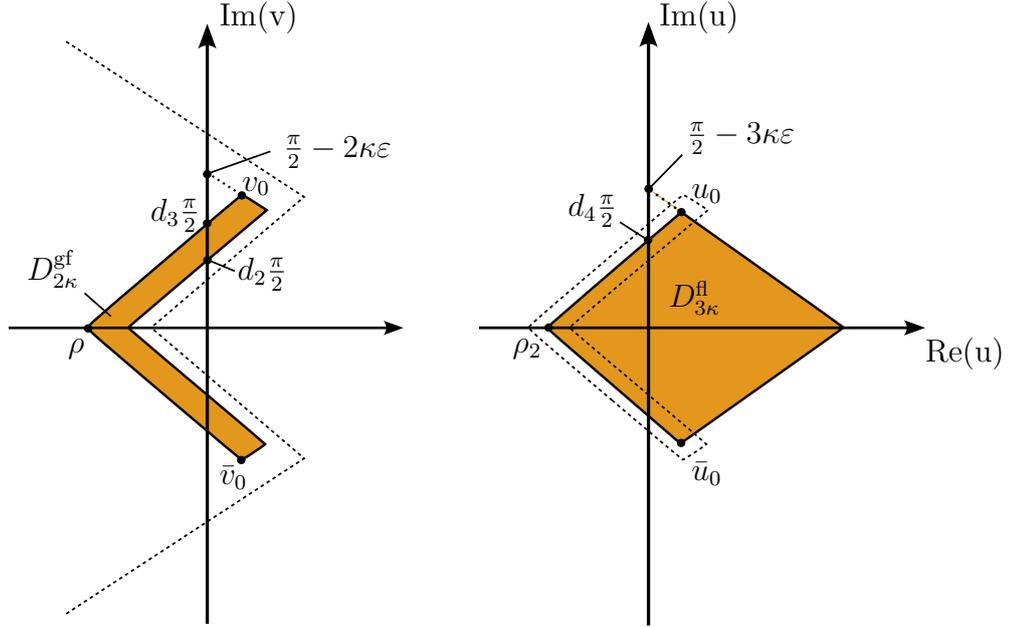

\centering
\begin{overpic}[scale=0.035]{DomainsExt1.pdf}
    \put(23,66){$\mathrm{Im(v)}$}
    \put(71,66){$\mathrm{Im(u)}$}
    \put(100,29){$\mathrm{Re(u)}$}
    \put(30,51.5){${\pi \over 2} - 2 \kappa \varepsilon$}
    \put(15.5,44.5){$d_3{\pi \over 2}$}
    \put(25,38){$d_2{\pi \over 2}$}
    \put(25.5,47.8){$v_0$}
    \put(23,15.5){$\bar v_0$}
    \put(6.5,30){$\rho$}
    \put(2,38){$\Dgf_{2\kappa}$}
    \put(72,35){$D^{\mathrm{fl}}_{3 \kappa}$}
    \put(74,53){${\pi \over 2} - 3 \kappa \varepsilon$}
    \put(61,45){$d_4{\pi \over 2}$}
    \put(74.5,47){$u_0$}
    \put(74.5,16){$\bar u_0$}
     \put(55,30){$\rho_2$}
\end{overpic}
\caption{In the figure on the left, the dashed domain is $D^{\mathrm{out}, \bu}_{\kappa}$ and the colored one is $\Dgf_{2\kappa}$. On the right, the dashed one is $\Dgf_{2\kappa}$, while the orange-colored one is $D^{\mathrm{fl}}_{3 \kappa }$.}
\label{fig:ext_domain_step1and2}
\end{figure}
\begin{equation}\label{def:domains_ext_step2}
\begin{aligned}
D^{\mathrm{fl}}_{3 \kappa} = \Big\{ u \in \C : &|\mathrm{Im}(u)| \le {\pi \over 2} - 3\kappa \varepsilon - \tan \slope \mathrm{Re}(u),\\
&|\mathrm{Im}(u)| \ge d_4{\pi \over 2} + \tan \hat \slope \mathrm{Re}(u)\Big\}
\end{aligned}
\end{equation}  
where $d < d_2 < d_4 < d_3 <{1 \over 2}$ and ${\pi \over 2} - 3 \kappa  \varepsilon >0$. We note that $\Dgf_{2 \kappa} \cap D^{\mathrm{fl}}_{3 \kappa} \ne \emptyset$.  

We extend $\Ggf$ in~\eqref{def:ext_tilde_Gamma:bis} into $D^{\mathrm{fl}}_{3\kappa}\times \T_{3 \sigma \over 4}$ analyzing the invariant equation
\begin{equation}\label{equation:flow}
    \Lfl \Gamma = X_{\mathcal{H}} (\Gamma), \qquad \Lfl \, \Gamma = \partial_u \Gamma + \frac{\omega}{\varepsilon} \partial_\theta \Gamma
\end{equation}
where $X_{\mathcal{H}}$ is the vector field associated to $\mathcal{H}$ (see~\eqref{def:rescaledHamiltonian}). 

\begin{proposition}\label{prop:ext2}
Let $\sigma >0$,  
     $d \in \left({1 \over 4}, {1 \over 2}\right)$,  $\slope$, $\hslope \in \left(0, {\pi \over 2}\right)$ and   $d<d_2<d_4<d_3< \frac{1}{2}$ be fixed. Consider $\kappa_1,\varepsilon_1$ provided in Proposition~\ref{prop:ext1} and $\Ggf$ be defined by~\eqref{def:ext_tilde_Gamma:bis}.

    There exist $0<\varepsilon_2\leq \varepsilon_1$, $\kappa_2 \geq \kappa_1$ and a constant $c_2>0$ such that, for all $0<\varepsilon<\varepsilon_2$ and $\kappa \geq \kappa_2$, $\Ggf$ can be analytically extended to $D^{\mathrm{fl}}_{3 \kappa} \times \T_{3 \sigma \over 4}$. Moreover, it is a solution   of~\eqref{equation:flow} satisfying
    \begin{equation}\label{pro:ext2_TildeGamma1Est}
        \sup_{(u, \theta) \in D^{\mathrm{fl}}_{3 \kappa } \times \T_{3 \sigma \over 4}} |\Ggf(u, \theta) - \Gamma_0(u, \theta)|\le c_2 \varepsilon
    \end{equation}
    where $\Gamma_0$ the unperturbed homoclinic in Lemma \ref{lem:homoclinic}. 
\end{proposition}
This proposition is proved in Section \ref{sec:ext2:proof}.
\item From flow to graph parameterization. Finally, on a suitable domain, we look for a change of coordinates $\mathrm{id} + h$ such that $\Ggf \circ \left(\mathrm{id} + h\right)$ is an invariant graph for the Hamiltonian $\mathcal{H}$. That is, denoting by $\pi_y$ the projection on the $y-$ component, we ask $h$ to satisfy
\begin{equation}\label{cond:ext:3}
\pi_y \left (\Ggf \circ \left(\mathrm{id} + h\right) \right ) = \gamma_0
\end{equation}
where we recall that $\Gamma_0=(\delta_0, \gamma_0)^\top$ (see Lemma~\ref{lem:homoclinic}).

The suitable domains are defined as
\begin{equation}\label{def:domains_ext_final}
\begin{aligned}
\tilde D_{4 \kappa,  d_5 } = \Big\{ u \in \C : &|\mathrm{Im}(u)| \le {\pi \over 2} - 4 \kappa \varepsilon - \tan \slope \mathrm{Re}(u),\\
&|\mathrm{Im}(u)| \le  d_5 {\pi \over 2} + \tan \hat \slope \mathrm{Re}(u),\\
&|\mathrm{Im}(u)| \ge d_5 {\pi \over 2} - \tan \slopee \mathrm{Re}(u) \Big\}
\end{aligned}
\end{equation}   
where the parameters $\slope$ and $\hat \slope$ are introduced by~\eqref{def:domains_outer_s}, while $\slopee \in \left(0, {\pi \over 2}\right)$.  Notice that, when $d_5 \in (d_2, d_4)$, $ \tilde D_{4\kappa,d_5} \cap D^{\mathrm{fl}}_{3\kappa} \neq \emptyset$. 
In Section~\ref{sec:ext3:proof}, we establish the following result. 
\begin{figure}[t]
\centering
\begin{overpic}[scale=0.035]{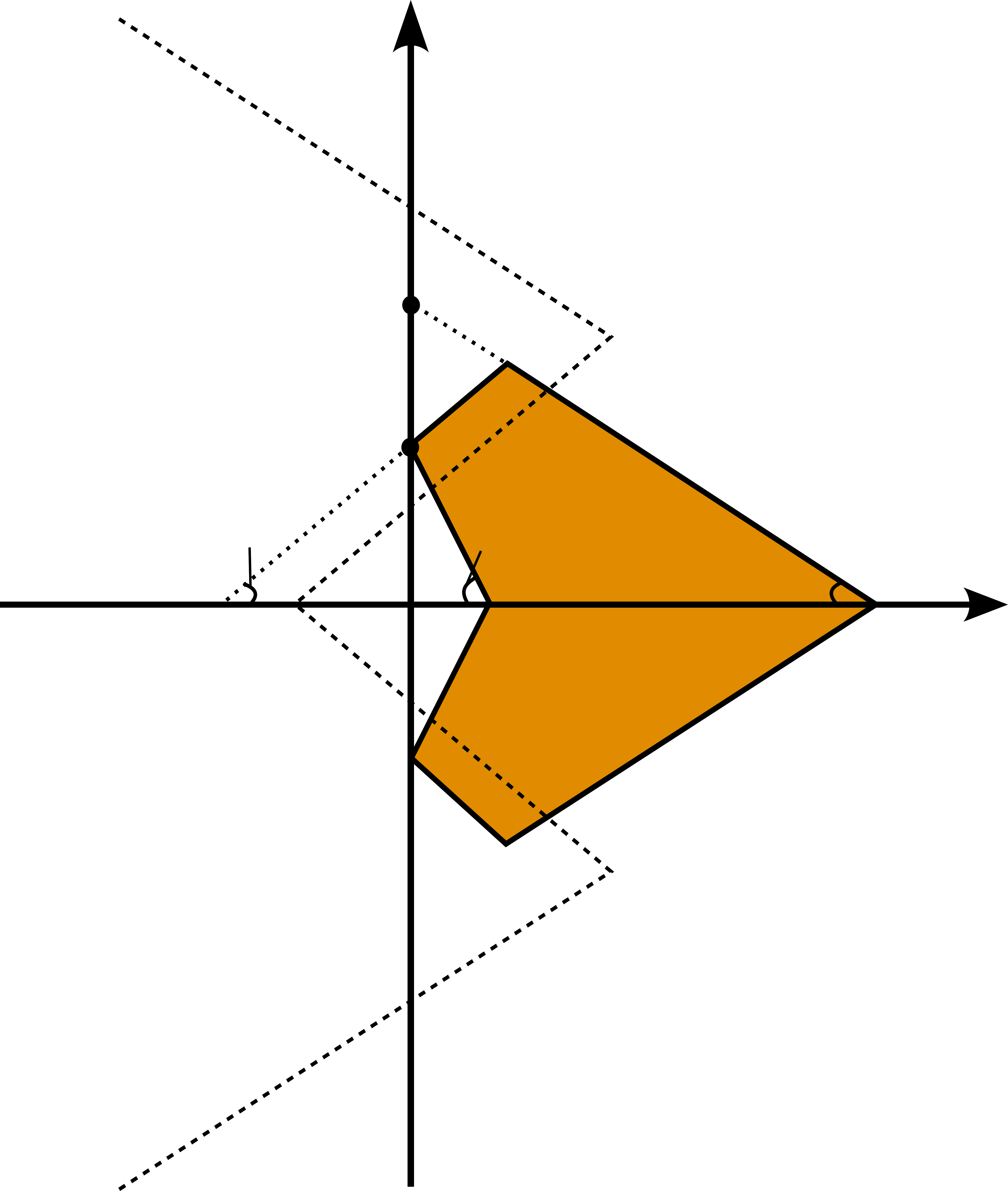}
    \put(35,101){$\mathrm{Im(u)}$}
    \put(85,44){$\mathrm{Re(u)}$}
    \put(61,50.8){$\slope$}
    \put(40.5,54){$\slopee $}
    \put(19,55.5){$\hat\slope $}
    \put(25,65){$d_5 {\pi \over 2} $}
    \put(12,75){${\pi \over 2} - 4 \kappa \varepsilon $}
    \put(43,41){$\tilde D_{4 \kappa, d_5}$}
\end{overpic}
\caption{The dashed domain is $D^{\mathrm{out}, \bu}_{\kappa}$, while the orange-colored one represents $ \tilde D_{4 \kappa,  d_5 }$.}
\label{fig:ext_domain_final}
\end{figure}
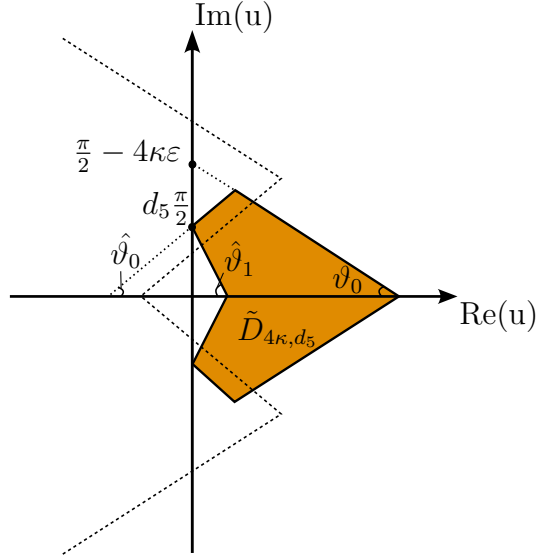
\begin{proposition}\label{prop:ext3}
Let $\sigma >0$,  
     $d \in \left({1 \over 4}, {1 \over 2}\right)$,  $\slope$, $\hslope \in \left(0, {\pi \over 2}\right)$, $\slopee \in \left(0, {\pi \over 2}\right)$ and $d<d_2<d_5<d_4<d_3< \frac{1}{2}$ be fixed. Consider $\kappa_2,\varepsilon_2$ and $\Ggf$ provided in Proposition~\ref{prop:ext2}.

    There exist $0<\varepsilon_3 \leq \varepsilon_2$, $\kappa_3 \geq \kappa_2$ and a constant $c_3>0$ such that for all $\varepsilon <\varepsilon_3$ and $\kappa \geq \kappa_3$, there exists $h : \tilde D_{4 \kappa, d_5} \times \T_{\sigma \over 2} \to \C^2$ satisfying condition~\eqref{cond:ext:3} and
        \begin{equation}\label{eq:ext3_est_g}
        \sup_{(u, \theta)\in \tilde D_{4\kappa, d_5} \times \T_{\sigma \over 2}} \big | h(u, \theta)\big| \le c_3 \varepsilon
    \end{equation}
\end{proposition}
\item We choose $\tilde \slope \in \left(0, {\pi \over 2}\right)$ and $d_0 \in \left({1 \over 4}, {1 \over 2}\right)$ in the definition~\eqref{def:domains_ext} of $D_{\kappa}^{\mathrm{ext}} $ in such a way that 
\begin{equation*}
    D_{\kappa}^{\mathrm{ext}} \subset D^{\mathrm{out}, \bs}_{\kappa} \cup D^{\mathrm{out}, \bs}_{\kappa} \cup \tilde D_{4 \kappa, d_5}
\end{equation*}
where we recall that $\tilde D_{4\kappa, d_5}$ is defined by~\eqref{def:domains_ext_final}. 
Therefore 
\[
\Gamma^{\bu}(u,\theta)=\Ggf \circ (\mathrm{id} + h) (u,\theta), \qquad (u,\theta) \in D_{\kappa}^{\mathrm{ext}} \times \T_{\sigma \over 2}, 
\]
where $\Gamma^{\bu}$ is the parameterization~\eqref{thm:def_Gamma} provided by Theorem~\ref{thm:outer}.


We claim that Proposition \ref{prop:ext3} proves the existence of an analytic extension of $T_1^\bu$ in Theorem~\ref{thm:outer} on the domain $D_{\kappa}^{\mathrm{ext}}$. Indeed, we define $\tilde{T}_1^{\bu}$ through
\[
\nabla \tilde T^{\bu}_1 = (D\gamma_0)^\top  \big (\pi_x \Ggf \circ (\mathrm{id}+h ) \big )- \nabla T_0
\]
where $\pi_x$ stands for the projection on the $x-$component and $T_0$ is the generating function for the unperturbed homoclinic. Since the unstable manifold is Lagrangian, the existence of $\tilde T^{\bu}_1$ is guaranteed. Let $\tilde{T}^\bu = T_0 + \tilde{T}^{\bu}_1$. It is well known that the expression of an invariant manifold as a graph is unique. Then, since $\partial_\theta \tilde{T}^{\bu} (u,\theta)= \partial_\theta T^{\bu}(u,\theta)=\gamma_0(u,\theta)$ for $(u,\theta) \in D^{\mathrm{out},u}_{\kappa} \cap D^{\mathrm{ext}}_{\kappa,d_0}$, $\tilde{T}_1^{\bu}$ is the analytic extension of $T^{\bu}_1$ in Theorem~\ref{thm:outer} into $D^{\mathrm{ext}}_{\kappa}$.

However, based on the previous analysis, we can only guarantee the existence of a constant $c>0$ such that
\begin{equation}\label{firstbound:T1:ext}
\sup_{(u,\theta) \in \big (D^{\mathrm{ext}}_{\kappa} \backslash D^{\mathrm{out},\bu}_{\kappa,\varsigma} \big )\times \T_{\sigma \over 2}} |\nabla  T_1^\bu (u,\theta)| \leq c \varepsilon 
\end{equation}

In Section~\ref{sec:ext4:proof}, we use a uniqueness argument to establish the estimates~\eqref{est:thm_ext}, thereby concluding the proof of Theorem~\ref{thm:ext}.
\end{enumerate}
 
\subsection{The inner equation} \label{sec:inner}
In this section, we aim to analyze the behavior of the parameterization of the invariant manifolds close to the singularities $\pm i \frac{\pi}{2}$, studying special solutions of the \textit{inner equation}, a parameterless equation, which, eventually, will provide the first-order for the difference $T^{\bu}(u,\theta)- T^{\bs}(u,\theta)$, for $u$ real. 

For this purpose, we consider the following change of coordinates
\begin{equation}
\label{uz}
u(z) = \varepsilon z + i{\pi \over 2}.
\end{equation}
It is the usual change of coordinate used in order to provide the inner equation associated with our problem (we refer to~\cite{B06}).  Furthermore,  for $*=\bu, \bs$, we introduce the following scaling
\begin{equation}
\label{phieta12}
K^*(z,\theta)= \varepsilon^3 T^*_1\left (\varepsilon z+ i \frac{\pi}{2}, \theta \right ), \qquad *= \bu, \bs. 
\end{equation}
Since the derivatives of $K^*$ are related to those of $T^*$ by 
\[
\partial_z K^*(z,\theta) = \varepsilon^4 \partial_u T^*_1\left (\varepsilon z+ i \frac{\pi}{2}, \theta \right ), \qquad \partial_\theta  K^*(z,\theta) = \varepsilon^3 \partial_\theta T^*_1\left (\varepsilon z+ i \frac{\pi}{2}, \theta \right )
\]
and $T^*_1$ satisfies $\Lout T^*_1 = \Fout (T^*_1)$, we deduce that $K^*$ is a solution of 
\begin{equation}\label{defLin}  
\Lin K^*  = \varepsilon^4 \Fout ( \varepsilon^{-3} K^*) , \quad \mbox{with} \quad \Lin K = \partial_z K +  \omega\partial_\theta K. 
\end{equation}
It is straightforward to verify that
\begin{equation}\label{hom:poles}
\frac{1}{r(u(z))} = i \varepsilon z + \mathcal{O}(\varepsilon^2 z^2), \qquad 
\frac{1}{R(u(z))} = i \varepsilon^2 z^2 + \mathcal{O}(\varepsilon^3 z^3).
\end{equation}
Moreover, recalling the definitions of $\gamma_0$ and $\delta_0$, we introduce the rescaled functions
\begin{equation}\label{def:gamma:delta:hat:inner} 
\begin{aligned}
\hat{\gamma}_0 (z,\theta) := \varepsilon \gamma_0(u(z),\theta) = & \frac{1} {i z } (1+ \mathcal{O}(\varepsilon z)) \left (\begin{array}{c} \cos \theta \\ \sin \theta \end{array} \right ),
\\
\hat{\delta}_0(z,\theta):= \varepsilon^2 \delta_0(u(z),\theta) =& \frac{1} {i z^2} (1+ \mathcal{O}(\varepsilon z)) \left (\begin{array}{c} \cos \theta \\ \sin \theta \end{array} \right ).
\end{aligned}
\end{equation}
Let $\diagonal_{\alpha,\beta}$ denote the $2\times 2$ diagonal matrix with entries $\alpha$ and $\beta$, and define 
\begin{equation}\label{def:in:Rtheta} 
R_\theta = \left (\begin{array}{cc} \cos \theta & - \sin \theta \\ \sin \theta & \cos \theta \end{array} \right), \hspace{2mm}
P(z,\theta;\varepsilon) = \left (\begin{array}{cc}  -\frac{1}{\varepsilon^2 R(u(z))} + iz^2 & 0 \\ 0 &   \frac{1}{\varepsilon r(u(z))} - iz \end{array} \right ).
\end{equation}
Then, using the expressions above, we obtain that 
\begin{align*}
  D\gamma_0(u(z),\theta) ^{-\top} =&   R_\theta \left (\begin{array}{cc} -\frac{1}{R(u(z))} & 0 \\ 0 & \frac{1}{r(u(z))}
  \end{array}\right )
  =  R_\theta \left [  \left ( \begin{array}{cc} -iz^2 & 0 \\ 0 & iz \end{array} \right ) + P(z,\theta;\varepsilon)\right ]\diagonal_{\varepsilon^2, \varepsilon}
\end{align*}
and as a consequence
\begin{align*}
\varepsilon^2 (D\gamma_0(u(z),\theta))^{-\top}  \nabla T(u(z),\theta)   =&  \varepsilon^2      (D\gamma_0(u(z),\theta))^{-\top}
 \left (\begin{array}{c} \varepsilon^{-4} \partial_z K  \\ \varepsilon^{-3} \partial_\theta K  \end{array} \right ) \\
 =& R_\theta \left[ \left ( \begin{array}{cc} -iz^2 & 0 \\ 0 & iz \end{array} \right )+ P(z,\theta;\varepsilon) \right ]
\nabla K.
\end{align*}
Observing that
\begin{equation}\label{dhatgamma0inner}
(D\hat \gamma_0(z,\theta))^{-\top}= R_\theta \left[ \left ( \begin{array}{cc} -iz^2 & 0 \\ 0 & iz \end{array} \right )+ P(z,\theta;\varepsilon) \right ]
\end{equation}
and recalling the definition~\eqref{def:operator_F} of $\Fout(T)$, we obtain, after the above computations,
\begin{equation}
\label{defFouter_innervariables}
\begin{aligned}
 \varepsilon^4 \Fout( \varepsilon^{-3} K) = & \frac{(\partial_z K)^2}{\varepsilon^{4}R^2(u(z))}   + \frac{(\partial_\theta K)^2}{\varepsilon^{2} r^2(u(z))}  +
 \alpha (\partial_\theta K)^2 - \beta \varepsilon^2 r^2(u(z)) \partial_\theta K \\ & + 
 \tilde{H}_1 ( \hat \delta_0 + (D\hat \gamma_0)^{-\top} \nabla K , \hat \gamma_0; \varepsilon).
\end{aligned}
\end{equation}
Taking $\varepsilon=0$ in the right hand side of~\eqref{defFouter_innervariables}
we obtain 
\begin{equation}\label{def:Fin} 
\begin{aligned}
\Fin(K) = & -(z^2\partial_z K)^2   -  (z\partial_\theta K)^2  +
 \alpha (\partial_\theta K)^2 + \beta  \frac{1}{z^2} \partial_\theta K  \\ & + 
 \tilde{H}_1 ( -iz^{-2} R_\theta e_1 + R_\theta \diagonal_{-iz^2, iz} \nabla K ,  -iz^{-1} R_\theta e_1; 0)
\end{aligned}
\end{equation}
where $e_1=(1,0)^\top$ and  we recall that $\tilde{H}_1(z)= \mathcal{O}(|z|^6)$.

We refer to 
\begin{equation}
\label{InnerEq}
\Lin K = \Fin (K)
\end{equation}
as the inner equation, and we are interested in analytic solutions of the latter satisfying the asymptotic conditions
\begin{equation}
\label{AsymInner}
\lim_{\Re z \to + \infty} K_0^{\bs}(z,\theta) =0, \quad \lim_{\Re z \to - \infty} K_0^{\bu}(z,\theta) =0.
\end{equation}

\begin{figure}[t]
\centering
\begin{overpic}[scale=0.045]{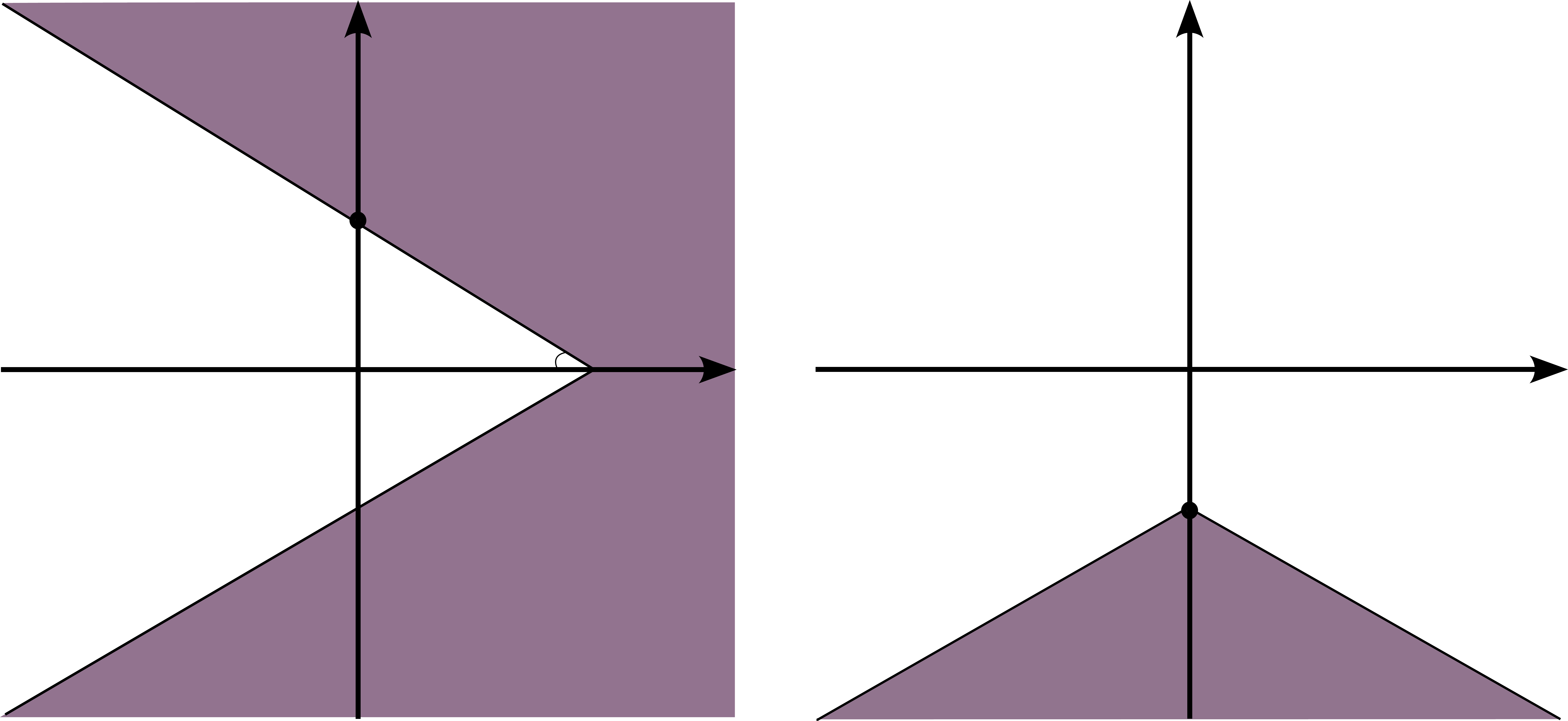}
    \put(23.6,47){$\mathrm{Im(z)}$}
    \put(76.5,47){$\mathrm{Im(z)}$}
    \put(101,20){$\mathrm{Re(z)}$}
    \put(23.5,32){$\kappa$}
    \put(77,13){$-\kappa$}
    \put(30,23.5){$\slope$}
    \put(31,10){$D^{\mathrm{in}, \bs}_{   \kappa}$}
    \put(67,3){$E^{\mathrm{in} }_{   \kappa}$}
\end{overpic}
\caption{The domain on the left is $D^{\mathrm{in}, \bs}_{   \kappa}$, whereas the one on the right is $E^{\mathrm{in} }_{   \kappa}$.}
\label{fig:inner_domain}
\end{figure}
Let us introduce the following domains.  Let $\vartheta_0 \in \left(0, {\pi \over 2}\right)$ and $\kappa>0$ as in definition~\eqref{def:domains_outer_s} of the outer domains $D^{\mathrm{out},*}_{\kappa}$. We define  
\begin{equation}
\label{DomainInner}
D^{\mathrm{in}, \bu}_{\kappa} = \{z \in \C : |\mathrm{Im}z| \ge  \tan \vartheta_0 \mathrm{Re} z +  \kappa\}, \quad D^{\mathrm{in}, \bs}_{   \kappa}  = - D^{\mathrm{in}, \bu}_{ \kappa} 
\end{equation}
{and we denote by}
\begin{equation} \label{DomainDiffInner}
    E^{\mathrm{in} }_{   \kappa} = D^{\mathrm{in},\bu}_{    \kappa} \cap D^{\mathrm{in},\bs}_{   \kappa} \cap \{w\in \mathbb{C}\,:\,\mathrm{Im} w <0\}
\end{equation}
the complex domain where $ K^{\bu}_0-K^{\bs}_0$ is defined. 
{We refer to Figure \ref{fig:inner_domain}.}

\begin{theorem}\label{thm:inner} Fix $\vartheta_0\in \left (0,\frac{\pi}{2}\right )$ and $\sigma>0$. 
There exists $\kappa_{\mathrm{in}}>0$ and $\Cin_{\mathrm{in}}>0$ such that for any $\kappa\geq \kappa_{\mathrm{in}}$, equation~\eqref{InnerEq} has two analytic solutions $K^{*}_0: D^{\mathrm{in},*}_{\kappa} \times \T_\sigma\to \mathbb{C}$, $*=\bu,\bs$, satisfying
\begin{align*}
\sup_{(z,\theta) \in D^{\mathrm{in}, *}_{  \kappa} \times \mathbb{T}_\sigma} |z^5 K^{*}_0(z,\theta)| &\leq \Cin_{\mathrm{in}},\\ \sup_{(z,\theta) \in D^{\mathrm{in}, *}_{  \kappa} \times \mathbb{T}_\sigma} |z^6 \partial_u K^{*}_0(z,\theta)|&\leq \Cin_{\mathrm{in}},\\ \sup_{(z,\theta) \in D^{\mathrm{in}, *}_{  \kappa} \times \mathbb{T}_\sigma} |z^6 \partial_{\theta} K^{*}_0(z,\theta)| &\leq \Cin_{\mathrm{in}} . 
\end{align*}

In addition,  
there exists an analytic function $ \gin :E^{\mathrm{in} }_{ \kappa} \times \T_\sigma\to \mathbb{C}$ 
such that 
\[
\sup_{(z,\theta) \in E^{\mathrm{in}}_{  \kappa} \times \mathbb{T}_\sigma} |z \gin(z,\theta)| \leq \Cin_{\mathrm{in}}
\]
and 
\begin{equation}\label{formula:diff:in}
K^{\bu}_0(z,\theta) - K^{\bs}_0(z,\theta) = \sum_{k<0} \chi^{[k]} e^{ik(\omega z-\theta + \gin(z,\theta))}, 
\end{equation}
with $\chi^{[k]}\in \mathbb{C}$ and bounded uniformly for $k<0$.
\end{theorem}
The proof of this result is deferred to Section~\ref{sec:inner:proof}.

\subsection{The matching errors}\label{sec:matching}

In this section, we verify that the functions $K^{\bu}_0$ and $K^{\bs}_0$ defined by Theorem~\ref{thm:inner}, approximate the functions $T_1^{\bu}$ and $T_1^{\bs}$ in Theorem~\ref{thm:outer} in complex subdomains of $D^{\mathrm{in}, \bu}_{ \kappa}$ and $D^{\mathrm{in}, \bs}_{ \kappa}$, respectively (see Figure~\ref{fig:mhc_domain}). 

We recall that $K^{*}_0$ are analytic functions defined in the inner domains $D_{  \kappa}^{\mathrm{in},*}$, with $*=\bu, \bs$ and $T_1^*$ are defined in the outer domains $D^{\mathrm{out},*}_{ \kappa,\varsigma} \subset D^{\mathrm{out},*}_{\kappa}$. 
One can see that, by rewriting the domains~\eqref{DomainInner} in terms of the $u$ variable using the change of coordinate~\eqref{uz}, then there exists  a suitable $\varsigma>0$ such that, for all $\vartheta_0$, $\hat \vartheta_0 \in  \left(0, {\pi \over 2}\right)$, we have the inclusion $D^{\mathrm{out}, *}_{\kappa,  \varsigma} \subset D^{\mathrm{in}, *}_{\kappa/2} $ for $*=\bu, \bs$. 

Take $\vartheta_0\in \left (0,\frac{\pi}{2}\right )$ and $\kappa>0$.  We fix $\vartheta_1$ and $\vartheta_2$ in such a way that
\begin{equation}\label{choicevarthetamatchin}
0 < \vartheta_1 < \vartheta_0 <\vartheta_2 < {\pi \over 2}
\end{equation}
and a parameter $\gamma \in (0,1)$.  For $j =1,2$, we fix $u_j \in \C$ in order to satisfy
\begin{equation}\label{defu1u2}
\begin{aligned}
& \mathrm{Im} u_j = - \tan \vartheta_j \mathrm{Re} u_j + {\pi \over 2} - \varepsilon \kappa, \quad \mathrm{Re} u_1 <0, \quad  \mathrm{Re} u_2 >0\\
&\left| u_j - i \left({\pi \over 2} - \varepsilon \kappa\right)\right| = \varepsilon^\gamma.
\end{aligned}
\end{equation}
We then introduce the domain (see Figure~\ref{fig:mhc_domain})
\begin{equation}
\label{Dmhcu}
\begin{aligned}
D^{\mathrm{mch}, \bu}_{\kappa} =& \Big\{ u \in \C : \mathrm{Im} u \le - \tan \vartheta_1 \mathrm{Re} u + {\pi \over 2} - \varepsilon \kappa, \\
& \mathrm{Im} u \le - \tan \vartheta_2 \mathrm{Re} u + {\pi \over 2} - \varepsilon \kappa \\
& \mathrm{Im} u \ge \mathrm{Im}u_1 - \tan \left({\vartheta_1 + \vartheta_2 \over 2}\right) (\mathrm{Re}u - \mathrm{Re} u_1) \Big\}. 
\end{aligned}
\end{equation}
\begin{figure}[t]
\centering
\begin{overpic}[scale=0.055]{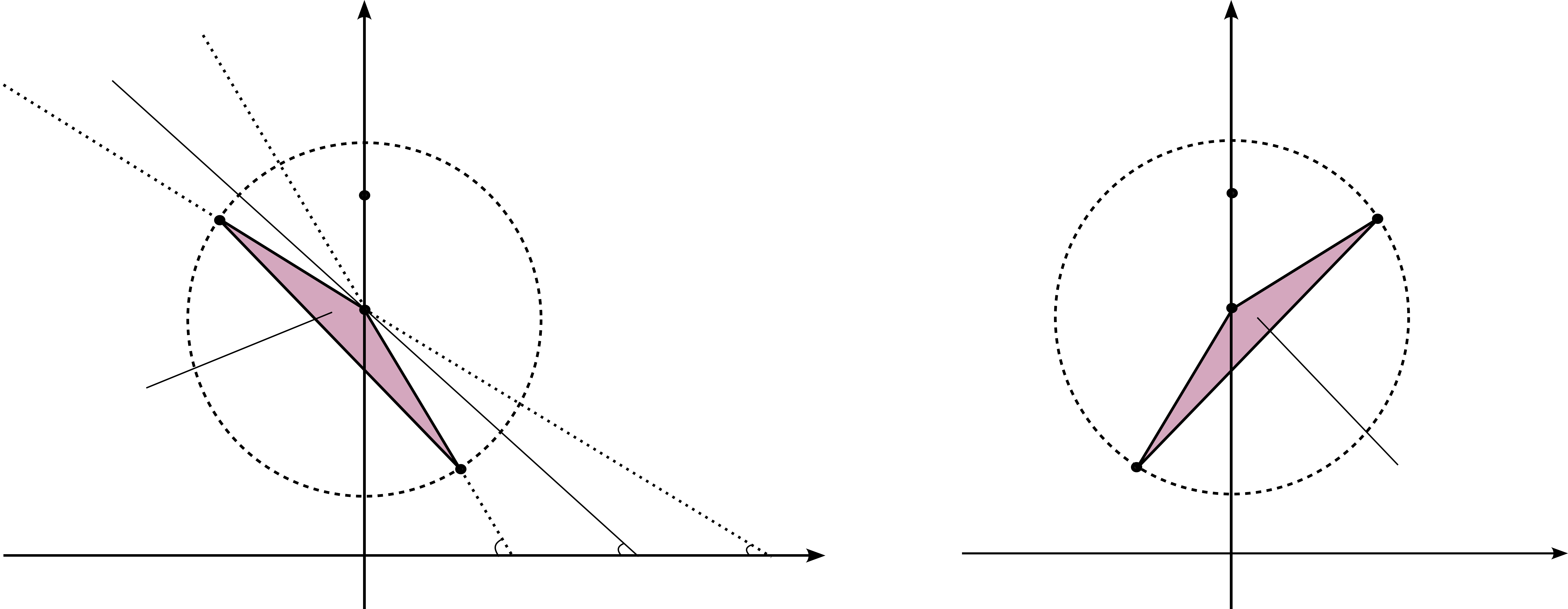}
    \put(23.6,40){$\mathrm{Im(u)}$}
    \put(78.5,40){$\mathrm{Im(u)}$}
    \put(99,0){$\mathrm{Re(u)}$}
    \put(10,23.7){$u_1$}
    \put(25,9){$u_2$}
    \put(28,4.3){$\vartheta_2$}
    \put(35,4.3){$\vartheta_0$}
    \put(42,4.3){$\vartheta_1$}
    \put(24,25.5){${\pi \over 2}$}
    \put(79.5,25.7){${\pi \over 2}$}
    \put(24,20){${\pi \over 2} - \kappa \varepsilon$}
    \put(68.5,19){${\pi \over 2} - \kappa \varepsilon$}
    \put(88,25.5){$\bar u_1$}
    \put(70,7){$\bar u_2$}
    \put(3,11){$D^{\mathrm{mch}, \bu}_{\kappa}$}
    \put(89,6){$D^{\mathrm{mch}, \bs}_{\kappa}$}
\end{overpic}
\caption{The domains $D^{\mathrm{mch}, *}_{\kappa}$ with $*=\bu, \bs$.}
\label{fig:mhc_domain}
\end{figure}

We point out that $D^{\mathrm{mch}, u}_{\kappa}$ is a triangular domain, having its vertices at $u_1$, $u_2$ and $i{\pi \over 2} - \varepsilon \kappa$.  Moreover, we define
\begin{equation}
\label{Dmhcs}
D^{\mathrm{mch}, \bs}_{\kappa}  = \{ u \in \C : -\bar u \in D^{\mathrm{mch}, \bu}_{\kappa} \}.
\end{equation}
With this choice of $\vartheta_1,\vartheta_2$ with respect to $\vartheta_0$ (see~\eqref{choicevarthetamatchin}), it is straightforward to verify that, for $\bar \kappa = {\kappa \over 2}$ and for a suitable $\hat \vartheta_0$, one has
\begin{equation}\label{inclusion_domains}
D^{\mathrm{mch}, *}_{\kappa} \subset D^{\mathrm{out}, *}_{\kappa, \varsigma} \subset D^{\mathrm{in}, *}_{\bar \kappa}  \qquad 
\mbox{ with $*=\bu,\bs$}.
\end{equation}

%
%
\begin{theorem}\label{thm:matching}
Fix $\vartheta_0\in \left (0,\frac{\pi}{2} \right )$, $\vartheta_1,\vartheta_2$ as in~\eqref
{choicevarthetamatchin}, $\gamma\in (0,1)$ and $\sigma>0$. 
Then there exists $\kappa_{\mathrm{mch}}>0$, $\varepsilon_{\mathrm{mch}}>0$ and $\Cmch_{\mathrm{mch}}$ such that for $\kappa \geq \kappa_{\mathrm{mch}}$ and $0<\varepsilon <\varepsilon_{\mathrm{mch}}$,
\[
\sup_{(u,\theta) \in {D}^{\mathrm{mch},*}_\kappa \times \mathbb{T}_\sigma} \left |\left (u-\frac{i\pi}{2} \right )^4 \left (T_1^*(u,\theta) - \frac{1}{\varepsilon^3} K^*_0\left (\frac{u- i\frac{\pi}{2}}{\varepsilon}  ,\theta \right ) \right  ) \right | \leq \Cmch_{\mathrm{mch}} \varepsilon^{2-\gamma}, 
\]
with $*=\bu,\bs$, $T_1^*$ defined by Theorem~\ref{thm:outer} and $K^*_0$ are the solutions of the inner equation provided by Theorem~\ref{thm:inner}. 
\end{theorem}
The proof of this result is deferred to Section~\ref{sec:proof:matching}.

\subsection{Measuring the distance between the invariant manifolds} \label{thm:distance}

This section is devoted to completing the proof of Theorem~\ref{thm:difference_between_manifolds}.
We first fix the constants that define the domains {$\T_\sigma$}, $D^{\mathrm{out},*}_{\kappa}$, $D^{\mathrm{ext}}_{\kappa}$, $D^{\mathrm{in},*}_{\kappa}$, and $D^{\mathrm{mch},*}_{\kappa}$, with $* =\bu,\bs$, except for $\kappa$. 
Throughout, we set
$$
\kappa = s\,|\log \varepsilon|
$$
with $s$ to be determined later.
Notice that, for a given $s>0$ and taking $\varepsilon$ small enough, $\kappa \geq \max\{\kappa_{\mathrm{out}}, \kappa_{\mathrm{ext}}, \kappa_{\mathrm{in}}, \kappa_{\mathrm{mch}}\}$. Therefore Theorems~\ref{thm:outer}, \ref{thm:ext}, \ref{thm:inner} and \ref{thm:matching} hold true with this choice of $\kappa$. 

To shorten the notation, we introduce
\begin{equation}\label{def:E:final}
\E = D_{s |\log \varepsilon|}^{\mathrm{ext}},
\end{equation}
%
for a given $s>0$, the domain where $T^{\bu}$ and $T^{\bs}$ are defined (we refer to~\eqref{def:domains_ext} for the definition of $D_{\kappa}^{\mathrm{ext}}$). Observe that $\E$ contains a real segment in the $u$ variable and an open set which is $s \varepsilon |\log \varepsilon| $- close to $\pm i \frac{\pi}{2}$. To simplify the notation and avoid a flood of constants, in the rest of this section, we write $u \lesssim v$ to indicate that there exists a constant $C > 0$ such that  $u \le C v$.


For $(u,\theta)\in \E \times \mathbb{T}_\sigma$, we define the difference
$$
\Delta (u,\theta) = T_1^{\bu}(u,\theta) - T_1^{\bs}(u,\theta) 
$$
and we notice that, by Theorems~\ref{thm:outer} and~\ref{thm:ext}, $T_1^{\bu}$ and $ T_1^{\bs}$ are solutions of the same equation~\eqref{HJpolarcoordinatesT1}, namely they satisfy $\Lout T^{\bu,\bs} = \Fout(T^{\bu,\bs})$ with $\Fout$ defined in~\eqref{def:operator_F}. Therefore 
$\Delta$ is a solution of the linear homogeneous equation
\begin{equation} \label{eq:diff:out}
\Lout \Delta = A \nabla \Delta  
\end{equation}
where, $\nabla \Delta = \big (\partial_u \Delta, \partial_\theta \Delta)^{\top}$ and, denoting $T_\lambda = T_1^{\bs} + \lambda (T_1^{\bu}- T_1^{\bs})$, $A$ is the row vector defined by
\begin{equation}\label{def:A:diff}
    \begin{aligned}
A(u,\theta) =&  (\mathbf{a}(u,\theta),\mathbf{b}(u,\theta)) \\ =&\left (\frac{1}{2R^2} \big (\partial_u T_1^{\bu} + \partial_u T_1^{\bs} \big ) , \left [\frac{1}{2r^2} + \alpha \varepsilon^2 \right ]\big (\partial_\theta T_1^{\bu} + \partial_\theta 
T_1^{\bs} \big ) -\beta \varepsilon r^2 
 \right ) \\&
+{\varepsilon^{-2}} \int_{0}^1 {\partial_x} \tilde{H}_1 \left(\varepsilon^2\left(\delta_0+ \big (D\gamma_0 \big )^{-\top} \nabla T_\lambda\right), \varepsilon \gamma_0; \varepsilon\right) \big (D\gamma_0 \big )^{-\top}\, d\lambda.
    \end{aligned}
\end{equation}
{Thanks to the latter, one can rewrite equation~\eqref{eq:diff:out} as
\begin{equation} \label{eq:diff:out:ab}
    \left(1 + \mathbf{a}\right)\partial_u \Delta + \left({\omega \over \varepsilon} + \mathbf{b}\right) \partial_\theta \Delta = 0.
\end{equation}

The following result characterizes $\Delta$ as a one-variable periodic function by means of a change of variable. 
\begin{proposition}\label{prop:C}
Let $s>0$ and consider $E$ defined by~\eqref{def:E:final}. Then there exist $\varepsilon_{4}, c_{4}>0$ and $\mathcal{C}:\E \times \T_\sigma \to \C$, a {real} analytic function, such that, for $0<\varepsilon < \varepsilon_{4}$, 
$$
\sup_{(u,\theta) \in \E \times \T_\sigma} \left | \left (u^2 + \frac{\pi^2}{4}\right )\mathcal{C}(u,\theta) \right | \leq c_{4}   {\varepsilon^2}
$$
and 
$$
\Delta (u,\theta) = \sum_{k\in \mathbb{Z}} \Upsilon^{[k]} (\varepsilon) e^{{ik\left ( {\omega \over \varepsilon}\left(u + \mathcal{C}(u, \theta)\right) -\theta\right)}}
$$
with $\Upsilon^{[k]} (\varepsilon) \in \mathbb{C}$ satisfying that $\Upsilon^{[-k]} (\varepsilon)= \overline{\Upsilon^{[k]}} (\varepsilon)$.

As a consequence {$(w,\theta) = \Xi (u,\theta):=(u+  \mathcal{C}(u,\theta),\theta)$} defines a change of variables onto $\E$ and $\Xi(\E)$. 
\end{proposition}
The proof of this proposition is delayed to Section~\ref{sec:proof:difference}. 
%


We follow the strategy proposed in~\cite{B12} by computing a first order of $\Upsilon^{[k]}(\varepsilon)$ by means of the coefficients $\chi^{[k]}$ related with the difference $K^{\bu}_0-K^{\bs}_0$ of the solutions of the inner equation (see Theorem~\ref{thm:inner}). To this end, we consider
\begin{equation}\label{def:Delta0+}
\begin{aligned}
\Delta_0^+(u,\theta) & = \sum_{k<0} \Upsilon^{[k]}_0(\varepsilon) e^{ik\left ( \frac{\omega}{\varepsilon} (u+\mathcal{C}(u,\theta))  -\theta\right )}\\ & :=\sum_{k<0} \frac{1}{\varepsilon^3 } \chi^{[k]} e^{-\frac{|k|\omega \pi}{2\varepsilon}} e^{ik\left (\frac{\omega}{\varepsilon}( u+ \mathcal{C}(u,\theta)) -\theta \right )}
, \\
\Delta_0^{-}(u,\theta) & = \sum_{k>0} \Upsilon^{[k]}_0(\varepsilon) e^{ik\left ( \frac{\omega}{\varepsilon}( u+ \mathcal{C}(u,\theta)) -\theta \right )}
\\& :=\sum_{k>0} \frac{1}{\varepsilon^3 } \overline{\chi^{[-k]}} e^{-\frac{|k|\omega \pi}{2\varepsilon}} e^{ik\left (\frac{\omega}{\varepsilon}( u+ \mathcal{C}(u,\theta)) -\theta \right )}
\end{aligned}
\end{equation}
and the function
\begin{equation}\label{def:Delta0}
\Delta_0 = \Delta_0^+ + \Delta_0^-.
\end{equation}
%

To justify the definition of $\Upsilon_0^{[k]}$, we recall the inner change of variables in~\eqref{phieta12}, so that $T_1^* (u,\theta) = \varepsilon^{-3} K^* \left (\varepsilon^{-1} \left (u-i\frac{\pi}{2} \right ), \theta\right )$. 
With this choice of $\Delta_0^-$, 
\[
\Delta_0(\bar{u},\theta) = \overline{\Delta_0(u,\theta)}, \qquad \mbox{for $\theta \in \T$.}
\]
We control $\widetilde{\Delta}:=\Delta - \Delta_0 $ by bounding $| \Upsilon^{[k]} (\varepsilon) - \Upsilon^{[k]}_0 (\varepsilon)|$. The key observation is as follows. On the one hand, considering 
$$
\Theta(U,\theta) = \sum_{k\neq 0} \big (\Upsilon^{[k]}(\varepsilon)- \Upsilon^{[k]}_0(\varepsilon)\big ) e^{ik\left (\frac{\omega}{\varepsilon} U - \theta\right )} + \Upsilon^{[0]}(\varepsilon)
$$
and the change of variables $(U,\theta) = \Xi (u,\theta):=(u+ \mathcal{C}(u,\theta),\theta)$ given in Proposition~\ref{prop:C}, we have that 
\begin{equation}\label{def:rel:Theta:Xi:Delta}
\Theta (U,\theta) = \widetilde{\Delta}  \circ \Xi^{-1} (U,\theta).
\end{equation}

On the other hand, since $\Theta$ is $2\pi-$periodic with respect to $\theta$, we have the relation
\begin{equation}\label{rel:Theta:Upsilon}
\Upsilon^{[k]}(\varepsilon)- \Upsilon^{[k]}_0(\varepsilon)=  \Theta^{[-k]} (U) e^{-ik \omega \frac{U}{\varepsilon}},
\end{equation}
where $\Theta^{[k]}(U)$ is the $k-$Fourier coefficient of $\Theta$. Then,  since $\Upsilon^{[k]}(\varepsilon)$ is independent on $U$, we can evaluated~\eqref{rel:Theta:Upsilon} at  
$$
U_* = u_* + \mathcal{C}(u_*,0), \qquad \text{with}\; u_*=i\frac{\pi}{2} -is\varepsilon |\log \varepsilon| \in \E.
$$ 
{This choice is seen to be optimal, as it leads to the most accurate estimates.}
Using Proposition~\ref{prop:C} to bound $|\mathcal{C}(u_*,0)|$, we obtain, for $k< 0$
$$
|\Upsilon^{[k]}(\varepsilon)- \Upsilon^{[k]}_0(\varepsilon) |\lesssim \big |\Theta^{[-k]} (U_*)\big | e^{- \frac{|k|\omega \pi}{2\varepsilon} + |k| \omega s |\log \varepsilon| } 
e^{ \frac{c_{4}|k| \omega}{s |\log \varepsilon|}}.
$$

As a result, one can obtain an exponentially small bound for $\big |\Upsilon^{[k]} (\varepsilon)- \Upsilon^{[k]}_0 (\varepsilon) \big |$ by bounding $\big |\Theta^{[-k]} (U_* ) \big |$ when $k<0$. The relation~\eqref{def:rel:Theta:Xi:Delta} clearly implies that 
\begin{equation}\label{bound:Thetak}
\big |\Theta^{[-k]} (U_* )\big | \lesssim \sup_{\theta \in \T} \big |\Theta (U_*  ,\theta) \big | =
\sup_{\theta \in \T} \big | \widetilde{\Delta} (\Xi^{-1} (U_*, \theta)) \big | .
\end{equation}
Since $\Xi$ satisfies $|\Xi (u,\theta) - (u,\theta)|\lesssim  \varepsilon s^{-1}|\log \varepsilon|^{-1}$, one easily deduces that
\begin{equation}\label{defEmach}
\big |\Theta^{[-k]} (U_* )\big |  \lesssim \sup_{(u,\theta) \in \E^{\mathrm{mch}} \times \T} \big |\widetilde{\Delta} (u,\theta)\big |, \qquad \text{with}\; \E^{\mathrm{mch}} = \E\cap D^{\mathrm{mch}, \bu}_{s|\log \varepsilon|} \cap D^{\mathrm{mch},\bs}_{s|\log \varepsilon|}.
\end{equation}
We recall that $\E, D^{\mathrm{mch},\bu}_{s|\log \varepsilon|},D^{\mathrm{mch},\bs}_{s|\log \varepsilon|}$ were defined in~\eqref{def:E:final}, \eqref{Dmhcu} and~\eqref{Dmhcs} respectively. 

Now, we use Theorems~\ref{thm:inner} and~\ref{thm:matching} to bound $|\widetilde{\Delta}(u,\theta)|$ for $u\in \E^{\mathrm{mch}}$ and $\theta \in \T$. More concretely, we prove the following result:
\begin{lemma} \label{lem:bound:tildeDelta}
Let $s>0$ be such that $\omega s\leq 1-\gamma$. There exist $\varepsilon_5,c_5$ such that, for $0<\varepsilon < \varepsilon_5$
$$
\sup_{(u,\theta) \in \E^{\mathrm{mch}} \times \T} \big |\widetilde{\Delta} (u,\theta)\big | \leq c_5\frac{\varepsilon^{ {\omega}s}}{ \varepsilon^3 |\log \varepsilon|},\qquad \widetilde{\Delta} = \Delta- \Delta_0.
$$
\end{lemma}
The proof of this lemma is postponed to the end of this section. 
Putting together the results in~\eqref{rel:Theta:Upsilon}, \eqref{bound:Thetak} and Lemma~\ref{lem:bound:tildeDelta}, we obtain
$$
|\Upsilon^{[k]}(\varepsilon)- \Upsilon^{[k]}_0(\varepsilon) | \lesssim \frac{\varepsilon^{\omega s}}{\varepsilon^3 |\log \varepsilon|}e^{- \omega \frac{|k|\pi}{2\varepsilon} + |k| \omega s |\log \varepsilon| } 
e^{ \frac{c_{4}|k| \omega }{s |\log \varepsilon|}}, \qquad k<0.
$$

Notice that, since $\Upsilon^{[-k]}(\varepsilon)- \Upsilon^{[-k]}_0(\varepsilon) = \overline{\Upsilon^{[k]}(\varepsilon)}- \overline{\Upsilon^{[k]}_0(\varepsilon)}$ the same bound holds for $k>0$. Then, since $c_{4} \omega s^{-1}|\log \varepsilon|^{-1} \leq 1$ if $\varepsilon$ is small enough,
we deduce 
$$
|\Upsilon^{[k]}(\varepsilon)- \Upsilon^{[k]}_0(\varepsilon) | \lesssim \begin{cases}
 {\displaystyle{\frac{1}{\varepsilon^3 |\log \varepsilon| } e^{-\frac{\pi \omega}{2\varepsilon}}}}   & k=-1,1 \\ \\
  {\displaystyle{\frac{1}{\varepsilon^{3-\omega s (1- |k|)} |\log \varepsilon| }
 e^{-\frac{|k|\omega \pi}{2\varepsilon}}} e^{ {c_{4}|k| \omega \over s|\log \varepsilon|}}}
 & |k|\geq 2
\end{cases}
$$
that implies, if $\varepsilon$ is small enough
$$
|\Upsilon^{[k]}(\varepsilon)- \Upsilon^{[k]}_0(\varepsilon) | \lesssim \begin{cases}
 {\displaystyle{\frac{1}{\varepsilon^3 |\log \varepsilon| } e^{-\frac{\pi \omega}{2\varepsilon}}  }} & k=-1,1 \\ \\
  {\displaystyle{\frac{1}{\varepsilon^{3} |\log \varepsilon| }
 e^{-\frac{|k|\omega \pi}{3\varepsilon}}}}
 & |k|\geq  2.
 \end{cases}
$$
Therefore, for $u\in \E\cap \mathbb{R}$ and $\theta\in \T$,
\begin{align*}
\big |\widetilde{\Delta} (u,\theta) - \Upsilon^{[0]}(\varepsilon)\big | & \lesssim \frac{1}{\varepsilon^{3} |\log \varepsilon| } \left [e^{-\frac{\pi \omega}{2\varepsilon}}  + \sum_{k\in \mathbb{Z}, |k|\geq 2}  e^{-\frac{|k| \omega \pi}{3\varepsilon}}\right ] \lesssim \frac{1}{\varepsilon^{3} |\log \varepsilon| }  e^{-\frac{\pi \omega}{2\varepsilon}} \\
\big |\partial_u \widetilde{\Delta} (u,\theta) \big | & \lesssim \frac{1}{\varepsilon^{4} |\log \varepsilon| }  e^{-\frac{\pi \omega}{2\varepsilon}} \\ 
\big |\partial_\theta \widetilde{\Delta} (u,\theta) \big | & \lesssim \frac{1}{\varepsilon^{3} |\log \varepsilon| }  e^{-\frac{\pi\omega }{2\varepsilon}}
\end{align*}
and Theorem~\ref{thm:difference_between_manifolds} follows trivially 
by definition~\eqref{def:Delta0} of $\Delta_0$.

\subsubsection{Proof of Lemma~\ref{lem:bound:tildeDelta}}
Using the notation $*=\bu,\bs$, we introduce 
\begin{equation*}
\begin{aligned}
\widetilde{\Delta}_1^{*} (u,\theta)= &T_1^{*} (u,\theta) - \frac{1}{\varepsilon^3} K_0^* \left (
\frac{u-i\frac{\pi}{2}}{\varepsilon},  \theta \right ) 
\\ 
\widetilde{\Delta}_2 (u,\theta) = &\frac{1}{\varepsilon^3} \left [ 
K_0^{\bu} \left (
\frac{u-i\frac{\pi}{2}}{\varepsilon},  \theta \right ) - K_0^{\bs} \left (
\frac{u-i\frac{\pi}{2}}{\varepsilon},  \theta \right ) \right ]  - \Delta_0^+(u,\theta)
\end{aligned}
\end{equation*}
and we decompose $\widetilde{\Delta}$ as
\begin{equation}\label{proof:def_tilde_delta}
\widetilde{\Delta} = \widetilde{\Delta}^{\bu}_1 - \widetilde{\Delta}^{\bs}_1 + \widetilde{\Delta}_2 - \Delta_0^-,
\end{equation}
we refer to~\eqref{def:Delta0+} for the definition of $\Delta_0^+$ and $\Delta_0^-$.
By Theorem~\ref{thm:matching}, we have that 
\begin{equation}\label{proof:stime_tilde_delta1}
\sup_{(u,\theta) \in E^{\mathrm{mch}} \times \T} |\widetilde{\Delta}_1^*(u,\theta)| \lesssim \frac{\varepsilon^{1-\gamma}}{\varepsilon^3 s^4|\log \varepsilon|^4}
\end{equation}
with $\E^{\mathrm{mch}}$ defined in~\eqref{defEmach}. 
We observe that, by Theorem~\ref{thm:inner} and definition~\eqref{def:Delta0+} of $\Delta_0^+$,
\begin{align*}
\widetilde{\Delta}_2(u,\theta) &= \frac{1}{\varepsilon^3}\sum_{k<0} \chi^{[k]} e^{-\frac{|k|\omega \pi}{2\varepsilon}}
\left [e^{ik \left (\frac{\omega}{\varepsilon}u - \theta + \textbf{g}\left(\varepsilon^{-1} \left (u-i\frac{\pi}{2}\right ), \theta\right )\right )} - 
e^{ik \left ( \frac{\omega}{\varepsilon}(u+\mathcal{C}(u,\theta))-\theta\right )}
\right ]\\ &= 
\frac{1}{\varepsilon^3}\sum_{k<0} \chi^{[k]} e^{-\frac{|k|\omega \pi}{2\varepsilon} }
e^{ik \left (\frac{\omega}{\varepsilon}u - \theta \right )}\left [ e^{ik \textbf{g}\left (\varepsilon^{-1} \left (u-i\frac{\pi}{2}\right ), \theta\right )} - 
e^{ik \left ( \frac{\omega}{\varepsilon}\mathcal{C}(u,\theta)\right )}
\right ]
\end{align*}

From Theorem~\ref{thm:inner} and Proposition~\ref{prop:C}, there exists a constant $c>0$, such that
$$
\frac{1}{\omega}\left |\textbf{g}\left (\varepsilon^{-1} \left (u-i\frac{\pi}{2}\right ), \theta\right ) \right |, \, \frac{1}{\varepsilon}|\mathcal{C}(u,\theta)| \leq c \frac{1}{s|\log \varepsilon|}, \qquad (u,\theta) \in \E^{\mathrm{mch}}\times \T.
$$
Then,  
\begin{align*} 
\big |\widetilde{\Delta}_2(u,\theta)\big | &\lesssim \frac{1}{\varepsilon^3 s |\log\varepsilon|}
\sum_{k<0} \big | \chi^{[k]} \big | e^{-\frac{|k|\omega \pi}{2\varepsilon} }  \big |e^{ik \omega \frac{u}{\varepsilon}} \big |
e^{\frac{|k| \omega c}{s|\log \varepsilon|} } \\ 
&\lesssim \frac{1}{ \varepsilon^3 s |\log\varepsilon|} \sum_{k<0} \big | \chi^{[k]} \big | \varepsilon^{|k| \omega s}e^{\frac{|k| \omega c}{s|\log \varepsilon|} },
\end{align*}
where in the last inequality we have evaluated at $u= i \frac{\pi}{2} - i \varepsilon s |\log \varepsilon|$. 

Therefore, since by Theorem~\ref{thm:inner}, $\chi^{[k]}$ are bounded uniformly for $k<0$ and $\varepsilon^{\omega s} e^{s^{-1} \omega c |\log \varepsilon|^{-1}} \leq 2 \varepsilon^{\omega s} \leq \frac{1}{2}$ if $\varepsilon$ small enough, we conclude that 
\begin{equation}\label{proof:stime_tilde_delta2}
\sup_{(u,\theta) \in E^{\mathrm{mch}} \times \T} \big |\widetilde{\Delta}_2(u,\theta)\big | \lesssim \frac{1}{\varepsilon^{3-\omega s}s |\log \varepsilon| }.
\end{equation}
We bound $\widetilde{\Delta}_0^-$ straightforwardly as
\begin{equation}\label{proof:stime_delta0-}
\sup_{(u,\theta) \in E^{\mathrm{mch}} \times \T} \big | {\Delta}_0^-(u,\theta)\big |\lesssim e^{-\frac{\pi \omega}{2\varepsilon}}.
\end{equation}
Therefore, since $\omega s\leq 1- \gamma$, combining~\eqref{proof:def_tilde_delta},~\eqref{proof:stime_tilde_delta1},~\eqref{proof:stime_tilde_delta2} and~\eqref{proof:stime_delta0-} we obtain that
$$
\sup_{(u,\theta) \in E^{\mathrm{mch}} \times \T} \big |\widetilde{\Delta}(u,\theta)\big | \lesssim \frac{\varepsilon^{\omega s}}{\varepsilon^3 |\log \varepsilon|}.
$$
This concludes the proof of Lemma~\ref{lem:bound:tildeDelta}.
  
\section{The outer scale}\label{sec:proof:outer}

The strategy of the proof is the following. We look for two solutions of~\eqref{HJpolarcoordinates} close to the unperturbed homoclinic given by Lemma~\ref{lem:homoclinic}, which in turn is given by the function $T_0$ in~\eqref{def:T0}. These solutions are obtained by solving a suitable fixed-point equation. The key point is that these solutions are defined in some complex domains ($D^{\mathrm{out}, \bu}_{\kappa} \times \T_\sigma$ and $D^{\mathrm{out}, \bs}_{\kappa}\times \T_\sigma$ see~\eqref{def:domains_outer_u} and~\eqref{def:domains_outer_s}) in the variables $(u,\theta)$ and, within these domains, we have optimal bounds.

For the rest of this section, we fix $\sigma >0$, $d \in \left({1 \over 4}, {1 \over 2}\right)$,  $\slope$, $\hslope \in \left(0, {\pi \over 2}\right)$, and $\varsigma>0$. These are the parameters we used in the definition of the domains $\T_\sigma$, $D^{\mathrm{out}, *}_{\kappa}$, $D^{\mathrm{out}, *}_{\kappa, \varsigma}$ and $D^{\mathrm{out}, *}_{\kappa, \infty}$ for $*=\bu, \bs$ (see~\eqref{def:complex_torus},~\eqref{def:domains_outer_s} and~\eqref{def:domains_outer_u}). Given $\varrho>0$, let $B_\varrho$ be the ball of radius $\varrho$ centered at the origin.
Furthermore, we denote by $\varepsilon_0$ and $\kappa_0$ some threshold values. We will take them to be suitably small and large, respectively, throughout this section.
 
For convenience, we introduce the following notation to avoid a proliferation of constants. We will write $u \lesssim v$ if there exists a constant $C>0$ independent of $\varepsilon_0$ and $\kappa_0$ such that $u \le Cv$. 

The rest of this section is divided into three parts. First, we introduce a series of Banach spaces satisfying some properties that we will use to solve the Hamilton-Jacobi equation~\eqref{HJpolarcoordinatesT1}. This is the content of Section \ref{sec:domains_and_Banach_spaces}. In Section \ref{sec:Inverse_of_Lout}, we look for a right inverse of the operator $\Lout$ (see~\eqref{def:operator_L}) and we use it to rewrite equation~\eqref{HJpolarcoordinatesT1} as a fixed-point equation. A solution of the latter is established in Section \ref{sec:Fixed_point_outer}.

\subsection{Functional set up}
\label{sec:domains_and_Banach_spaces}

We introduce several spaces of analytic functions defined in the domains $D^{\mathrm{out}, \bu}_{\kappa}$ and $D^{\mathrm{out}, \bs}_{\kappa}$, introduced in~\eqref{def:domains_outer_s} and~\eqref{def:domains_outer_u}.

First, given a function $g:D^{\mathrm{out}, *}_{\kappa}\to \C$, we define the norm
\begin{equation*}
\label{norm1}
|g|^*_{n,m} = \sup_{u \in D^{\mathrm{out}, *}_{\kappa, \varsigma}} \left| g(u) \cosh^n u\right| + \sup_{u \in D^{\mathrm{out}, *}_{\kappa,  \infty}} \left| g(u) \cosh^m u\right|,
\end{equation*}
and, given a function $\phi : D^{\mathrm{out}, *}_{\kappa} \times \T_\sigma \to \C$ with Fourier series $ \sum_{\ell \in\Z} \phi^{[\ell]}(u) e^{i\ell \theta}$,  
\begin{align}
\label{norm2}
|\phi|^*_{n,m, \sigma} &= \sum_{\ell \in \Z} |\phi^{[\ell]}|^*_{n,m} e^{|\ell|\sigma},\\
\label{norm3}
\lfloor \phi \rfloor_{n, m, \sigma}^* &= |\phi|^*_{n,m, \sigma} + |\partial_u \phi|^*_{n+1,m, \sigma} + {1 \over \varepsilon} |\partial_\theta \phi|^*_{n+1,m, \sigma},
\end{align}
where $*= \bu,\bs$,  and $n$, $m\in \R$. 

We will also use the same notation for vector-valued functions
 $\phi : D^{\mathrm{out}, *}_{\kappa} \times \T_\sigma \to \C^a$ and matrices $M= \{M_{ij}\}_{1 \le i \le a, 1 \le j \le b } : D^{\mathrm{out}, *}_{\kappa} \times \T_\sigma \to \mathcal{M}_{a\times b}(\C)$, where $\mathcal{M}_{a\times b}(\C)$ denotes the space of $a\times b$ matrices with complex coefficients. In these cases, we set
\[
|\phi|^*_{n,m, \sigma} = \max_{1\le i \le a} |\phi_i|^*_{n,m, \sigma}, \qquad |M|^*_{n,m, \sigma} = \max_{1\le i \le a, \, 1\le j \le b} |M_{ij}|^*_{n,m, \sigma}.
\]
We define analogously the norm $\lfloor \cdot \rfloor_{n, m, \sigma}^*$.

Using these norms, we consider the Banach spaces
\begin{eqnarray*}
\mathcal{X}_{n,m}^* &=&  \left\{ g : D^{\mathrm{out}, *}_{\kappa} \to \C : \mbox{$g$ is analytic and $|g|^*_{n,m} <\infty$ }\right\},\\
\mathcal{X}_{n,m, \sigma}^* &= &\left\{ \phi : D^{\mathrm{out}, *}_{\kappa} \times \T_\sigma \to \C : \mbox{$\phi$ is analytic and $|\phi|^*_{n,m, \sigma} <\infty$ }\right\},\\
\mathcal{\tilde X}_{n,m, \sigma}^* &= &\left\{ \phi : D^{\mathrm{out}, *}_{\kappa} \times \T_\sigma \to \C : \mbox{$\phi$ is analytic and $\lfloor\phi\rfloor^*_{n,m, \sigma} <\infty$ }\right\},
\end{eqnarray*}
with $*= \bu,\bs$.  We use the same notation for vector-valued functions and matrices. A vector-valued function or a matrix is said to belong to the above spaces whenever each of its components does.

In the following, we state several lemmas that summarize some properties of the previous norms 
\begin{lemma}\label{LemmaPropNormOut1}
We fix $n,m\in \mathbb{N}$. Then, there exists $\varepsilon_0>0$ and a constant $C>0$ such that for all 
$\phi \in \mathcal{X}_{n,m, \sigma}^*$, $*= \bu,\bs$, $0<\varepsilon \le \varepsilon_0$ and $\kappa\geq 1$,
\begin{enumerate}
\item For all $n_+ > n$, and $\phi \in \mathcal{X}_{n_+,m, \sigma}^*$,
\begin{equation*}
|\phi|^*_{n_+, m, \sigma} \le C |\phi|^*_{n, m, \sigma}.
\end{equation*}
\item For all $n_- < n$, and $ \phi \in \mathcal{X}_{n_-,m, \sigma}^*$,
\begin{equation*}
|\phi|^*_{n_-, m, \sigma} \le {C \over (\varepsilon \kappa)^{n-n_-}} |\phi|^*_{n, m, \sigma}.
\end{equation*}
\item For all $m_- < m$, and $ \phi \in \mathcal{X}_{n,m_-, \sigma}^*$,
\begin{equation*}
|\phi|^*_{n, m_-, \sigma} \le C |\phi|^*_{n, m, \sigma}.
\end{equation*}
\end{enumerate}
Let $\phi_1 \in \mathcal{X}_{n_1,m_1, \sigma}^*$ and $\phi_2 \in \mathcal{X}_{n_2,m_2, \sigma}^*$ for $*= \bu,\bs$. Then, for all $m \le m_1 + m_2$, the product $\phi_1 \phi_2 \in \mathcal{X}_{n_1 + n_2,m, \sigma}^*$ and 
\begin{equation*}
|\phi_1 \phi_2|^*_{n_1 + n_2, m, \sigma} \le C |\phi_1|^*_{n_1, m_1, \sigma} |\phi_2|^*_{n_2, m_2, \sigma},
\end{equation*}
for $*= \bu,\bs$. This property holds true also if $\phi_1$ and $\phi_2$ are matrices whose elements belong to $\mathcal{X}_{n_1,m_1, \sigma}^*$ and $\mathcal{X}_{n_2,m_2, \sigma}^*$, respectively.
\end{lemma}
\begin{proof}
We refer to~\cite{Cast15} for the proof.
\end{proof}

\begin{lemma}\label{LemmaPropNormOut2}
Let $\sigma_0 > \sigma$. Then there exists a constant $C>0$ such that for all $\phi \in \mathcal{X}^*_{n, m, \sigma_0}$, with $*=\bu, \bs$,
\begin{equation*}
|\phi|^*_{n,m,\sigma} \le C\left(\sup_{(u, \theta) \in D^{\mathrm{out}, *}_{\kappa, \varsigma}\times \T_{\sigma_0}} \left| \phi(u, \theta) \cosh^n u\right|   +   \sup_{(u, \theta) \in D^{\mathrm{out}, *}_{\kappa, \infty} \times \T_{\sigma_0}} \left| \phi(u, \theta) \cosh^m u\right|      \right).
\end{equation*}
\end{lemma}
\begin{proof}
We refer to~\cite{Cast15} for the proof.
\end{proof}
We will widely use the claims in Lemmas~\ref{LemmaPropNormOut1} and~\ref{LemmaPropNormOut2} without further mention. 

\subsection{Inverting the operator $\Lout$}\label{sec:Inverse_of_Lout}

In order to transform equation~\eqref{HJpolarcoordinatesT1} into a fixed point equation, it is necessary to choose a right inverse of the operator $\Lout$ in~\eqref{def:operator_L},
acting in the spaces introduced in Section~\ref{sec:domains_and_Banach_spaces}.
This is rather standard.
Given $g^*: D^{\mathrm{out}, *}_{\kappa}  \times \T_\sigma \to \C$,  with $*=\bu, \bs$,
we consider its Fourier serie $g^*(u,\theta) = \sum_{\ell\in \Z} g^{*,[\ell]}(u) e^{i\ell \theta}$.
Then, a formal solution of the equation 
\begin{equation*}
\Lout(\phi^*) = g^*
\end{equation*}
 is given by $\G^*(g^*)$, where
\begin{equation}
\label{def:Gouter}
\begin{aligned}
\G^{\bs,[\ell]} (g) (u) & = - \int_u^{+\infty} e^{i{\omega \over \varepsilon}\ell(\tau - u)} g^{\bs, [\ell]}(\tau)d\tau,  \\
\G^{\bu,[\ell]} (g) (u) & = - \int_u^{-\infty} e^{i{\omega \over \varepsilon}\ell(\tau - u)} g^{\bu, [\ell]}(\tau)d\tau.
\end{aligned}
\end{equation}

\begin{lemma}\label{PropG}
Let $n$, $m \in \Z$ with $n \ge 0$ and $m \ge 0$ and $g \in \mathcal{X}^*_{n,m, \sigma}$ with $*=\bu, \bs$. There exists $0 < \varepsilon_0 <1$ and a constant $C >0$ independent of $g$ such that for all $\ell \in \Z$, $\kappa \ge 1$ and $0 < \varepsilon \le \varepsilon_0$
\begin{enumerate}
\item If $n > 1$,  then $|\G^{*, [\ell]}(g)|^*_{n-1, m} \le C |g^{[\ell]}|^*_{n, m}$.
\item If $\ell \ne 0$ and $n \ge 0$, then $|\G^{*,[\ell]}(g)|^*_{n, m} \le C\displaystyle{\varepsilon \over |l|} |g^{[\ell]}|^*_{n, m}$.
\item If $n > 1$, then $|\G^*(g)|^*_{n-1, m, \sigma} \le C |g|^*_{n, m, \sigma}$.
Moreover, if $g^{[0]}(u) =0$ for all $u \in D^{\mathrm{out}, *}_{\kappa}$, then for all $n > 0$ one has that $|\G^*(g)|^*_{n, m, \sigma} \le C\varepsilon |g|^*_{n, m, \sigma}$.
\item If $g \in \mathcal{X}_{n,m,\sigma}^*$,  and $n > 1$, then $\mathcal{G}(g) \in \mathcal{\tilde X}^*_{n-1, m,\sigma}$ and
\begin{equation*}
\lfloor \G^*(g)\rfloor_{n-1, m, \sigma}^* \le C |g|_{n,m,\sigma}^*.
\end{equation*}
\end{enumerate}
\end{lemma}
\begin{proof}
We refer to~\cite{B06} for the proof. 
\end{proof}

\subsection{Hamilton-Jacobi equation as a fixed point}\label{sec:Fixed_point_outer}

Using the operator $\G^*$ in~\eqref{def:Gouter}, we rewrite the Hamilton-Jacobi equation~\eqref{HJpolarcoordinatesT1} as
\begin{equation}
    \label{eq:HJ_fixed_point}
T_1 = \G^* \circ \Fout (T_1).
\end{equation}
with $*=\bu, \bs$. In this section, we verify that the right-hand side of the latter is a contraction in a ball of radius $\Oo(\varepsilon^2)$ of $\mathcal{\tilde X}_{5,6,\sigma}^*$,  giving rise to two solutions, $T_1^*$, with $* = \bu, \bs$. We split the proof of this claim into several technical lemmas.
\begin{lemma}
\label{lem:F_at_0} There exist $\kappa_0 \ge 1$, $0 < \varepsilon_0<1$ and a constant $c_0>0$ such that,
for any  $\kappa \ge \kappa_0$ and $0<\varepsilon\le\varepsilon_0$, 
\begin{equation*}
\left|\G \circ \Fout(0) \right|^*_{5,6,\sigma} \le c_0 \varepsilon^2
\end{equation*}
with $* = \bu, \bs$.
\end{lemma}
\begin{proof}
By Lemma~\ref{lem:homoclinic}, each component of $\gamma_0$ and $\delta_0$ satisfy
\[
| \varepsilon \gamma_0|_{1,1,\sigma}^* \lesssim \varepsilon, \qquad | \varepsilon^2 \delta_0 |_{2,2,\sigma}^* \lesssim \varepsilon^2.
\]
Then, since, in view of~\eqref{def:operator_F},
\[
\Fout(0) = \varepsilon^{-4} \tilde H_1 (\varepsilon^2\delta_0, \varepsilon \gamma_0; \varepsilon),
\]
using the fact that $\tilde H_1(x,y; \varepsilon) = \Oo_6(x,y)$ uniformly in $\varepsilon$ (see Proposition \ref{prop:main_rescaling}), and Lemma \ref{PropG}, the claim follows immediately.
\end{proof}

\begin{lemma}
\label{lem:d_gamma_0_inverse_transpose_times_gradient_T}
For any $\kappa >0$, if $T \in \mathcal{\tilde X}_{5,6,\sigma}^*$, with $*= \bu$, $\bs$,
\[
| (D\, \gamma_0)^{-\top}\nabla T|_{4,5,\sigma}^*  \lesssim (1+\kappa^{-1}) \lfloor T \rfloor_{5,6,\sigma}^*.
\]
\end{lemma}

\begin{proof}
By Lemma~\ref{lem:homoclinic},
\[
\left| \frac{1}{R} \cos \theta  \right|_{-2,-1,\sigma}^*, \; \left| \frac{1}{r} \sin \theta  \right|_{-1,-1,\sigma}^* \lesssim 1.
\]
Hence, in view of~\eqref{eq:Dgamma0_inverse_and_transpose}, the first component of 
$(D\, \gamma_0)^{-\top}\nabla T$ satisfies
\[
\begin{aligned}
\left| \frac{1}{R} \cos \theta \, \partial_u T + \frac{1}{r} \sin \theta \, \partial_\theta T\right|_{4,5,\sigma}^* & \lesssim
\left| \partial_u T \right|_{6,6,\sigma}^* + \left| \partial_\theta T\right|_{5,6,\sigma}^*   \le \lfloor T \rfloor_{5,6,\sigma}^* + \frac{1}{\kappa\varepsilon}\left| \partial_\theta T\right|_{6,6,\sigma}^* \\
& \le (1+\kappa^{-1}) \lfloor T \rfloor_{5,6,\sigma}^*.
\end{aligned}
\]
The second component is bounded analogously.
\end{proof}

\begin{lemma}\label{lem:Lipschitz_constant_of_Fout}
Given $\Cout_1 >0$ and $*=\bu, \bs$, there exist $\kappa_0\ge 1, 0 <\varepsilon_0 <1$ and a constant $\Cout_2>0$ such that, for any  $\kappa \ge \kappa_0$ and $0<\varepsilon \le \varepsilon_0$, the functional $\Fout: B_{\Cout_1 \varepsilon^2}\subset \mathcal{\tilde X}_{5,6,\sigma}^* \to \mathcal{X}_{6,6,\sigma}^*$ is Lipschitz with
$\mathrm{lip}\, \Fout \le {\Cout}_2 / \kappa^2$.
\end{lemma}

\begin{proof}
We remark that, if $T\in \mathcal{\tilde X}_{n,m,\sigma}^*$, $* =\bu$, $\bs$, then, from the definition of the norm in~\eqref{norm3}
\[
|\nabla T|_{n+1,m,\sigma} \le (1+ \varepsilon) \lfloor T \rfloor_{n,m,\sigma}.
\]

By the definition of $\Fout$ in~\eqref{def:operator_F}, we write
\[
\Fout = \F_1+\F_2+\F_3+\F_4+\F_5,
\]
where
\begin{align}
\label{def:F1}
 \F_1(T) & = \varepsilon^{-4} \tilde H_1 (\varepsilon^2( \delta_0 +(D\, \gamma_0)^{-\top}\nabla T), \varepsilon \gamma_0; \varepsilon),\\
\label{def:F2}
\F_2(T) & = \frac{1}{2r^2}(\partial_\theta T)^2, \\
\label{def:F3}
\F_3(T) & = \alpha \varepsilon^2   (\partial_\theta T)^2, \\
\label{def:F4}
\F_4(T) & = - \beta \varepsilon r^2 \partial_\theta T, \\
\label{def:F5}
\F_5(T) & = \frac{1}{2R^2}(\partial_u T)^2.
\end{align}
We bound each term separately.

We start with $\F_1$. We first remark that, by Lemma~\ref{lem:d_gamma_0_inverse_transpose_times_gradient_T}, if $T\in B_{\Cout_1 \varepsilon^2} \subset \mathcal{\tilde X}_{5,6,\sigma}^*$, with $* = \bu,\bs$,
\[
| (D\, \gamma_0)^{-\top} \nabla T |_{2,2,\sigma}^* \le \frac{1}{\varepsilon^2 \kappa^2} | (D\, \gamma_0)^{-\top} \nabla T |_{4,5,\sigma}^* \lesssim \frac{1}{\kappa^2}.
\]
Hence, $\F_1(T)$ is well defined if $\kappa_0$ is large enough. Moreover, if $T\in B_{\Cout_1 \varepsilon^2} \subset \mathcal{\tilde X}_{5,6,\sigma}^*$, since $\partial_x \tilde H_1(x,y; \varepsilon) = \Oo_5(x,y)$ uniformly in $\varepsilon$,
\[
|\partial_x \tilde H_1(
\varepsilon^2 \delta_0 +\varepsilon^2(D\, \gamma_0)^{-\top}\nabla  T), \varepsilon \gamma_0; \varepsilon)|_{2,1,\sigma}^* \lesssim |(\varepsilon \gamma_0)^5|_{2,1,\sigma}^* \lesssim \frac{1}{(\varepsilon \kappa)^3} |(\varepsilon \gamma_0)^5|_{5,5,\sigma}^* \lesssim \frac{\varepsilon^2}{\kappa^3}.
\]
Then, for any $T, \tilde T \in B_{\Cout_1 \varepsilon^2} \subset \mathcal{\tilde X}_{5,6,\sigma}^*$, by Lemma~\ref{lem:d_gamma_0_inverse_transpose_times_gradient_T}
\begin{multline*}
|\F_1(T) - \F_1(\tilde T)|_{6,6,\sigma}^* \\
\begin{aligned}
= & \varepsilon^{-2} \Big|\int_0^1 \partial_x \tilde H_1(
\varepsilon^2 \delta_0 +\varepsilon^2(D\, \gamma_0)^{-\top}(\nabla \tilde T+ \tau(\nabla T- \nabla \tilde T)), \varepsilon \gamma_0; \varepsilon)\,d\tau \\
& \times (D\, \gamma_0)^{-\top}(\nabla T- \nabla \tilde T)\Big|_{6,6,\sigma}^* \\
& \lesssim \frac{1}{\kappa^3} |(D\, \gamma_0)^{-\top}(\nabla T- \nabla \tilde T)|_{4,5,\sigma}^* \\
& \lesssim \frac{1}{\kappa^3} \lfloor T-\tilde T\rfloor_{5,6,\sigma}^*.
\end{aligned}
\end{multline*}
This proves the claim for $\F_1$.

We deal with $\F_2$. For any $T, \tilde T \in B_{\Cout_1 \varepsilon^2} \subset \mathcal{\tilde X}_{5,6,\sigma}^*$,
\begin{align*}
|\F_2(T) - \F_2(\tilde T)|_{6,6,\sigma}^* = & \left|\frac{1}{r^2} (\partial_\theta T+\partial_\theta \tilde T) (\partial_\theta T-\partial_\theta\tilde T)  \right|_{6,6,\sigma}^* \\
& \le \varepsilon \left|\frac{1}{r^2} (\partial_\theta T+\partial_\theta \tilde T)
\right|_{0,0,\sigma}^* \lfloor T-\tilde T\rfloor_{5,6,\sigma}^* \\
& \lesssim \varepsilon \left|\frac{1}{r^2}\right|_{-2,-2,\sigma}^* \left|\partial_\theta T+\partial_\theta \tilde T
\right|_{2,2,\sigma}^* \lfloor T-\tilde T\rfloor_{5,6,\sigma}^* \\
& \lesssim \frac{1}{\varepsilon^3 \kappa^4} \left|\partial_\theta T+\partial_\theta \tilde T
\right|_{6,6,\sigma}^* \lfloor T-\tilde T\rfloor_{5,6,\sigma}^* \\
& \lesssim \frac{1}{\kappa^4}  \lfloor T-\tilde T\rfloor_{5,6,\sigma}^*,
\end{align*}
which proves the claim for $\F_2$.

In the case of $\F_3$, for any $T, \tilde T \in B_{\Cout_1 \varepsilon^2} \subset \mathcal{\tilde X}_{5,6,\sigma}^*$,
\begin{align*}
|\F_3(T) - \F_3(\tilde T)|_{6,6,\sigma}^* = & |\alpha|\varepsilon^2 | (\partial_\theta T+\partial_\theta \tilde T) (\partial_\theta T-\partial_\theta\tilde T)  |_{6,6,\sigma}^* \\
& \lesssim  \varepsilon^3 | \partial_\theta T+\partial_\theta \tilde T|_{0,0,\sigma}^* \lfloor T-\tilde T\rfloor_{5,6,\sigma}^* \\
& \lesssim \frac{1}{\varepsilon^3 \kappa^6}  | \partial_\theta T+\partial_\theta \tilde T|_{6,6,\sigma}^* \lfloor T-\tilde T\rfloor_{5,6,\sigma}^* \\
& \lesssim \frac{1}{\kappa^6}  \lfloor T-\tilde T\rfloor_{5,6,\sigma}^*,
\end{align*}
which proves the claim for $\F_3$.

In the case of $\F_4$, for any $T, \tilde T \in B_{\Cout_1 \varepsilon^2} \subset \mathcal{\tilde X}_{5,6,\sigma}^*$,
\begin{align*}
|\F_4(T) - \F_4(\tilde T)|_{6,6,\sigma}^* = & |\beta|\varepsilon | r^2 (\partial_\theta T-\partial_\theta\tilde T)  |_{6,6,\sigma}^* \\
& \lesssim  \varepsilon^2 |r^2|_{0,0,\sigma}^* \lfloor T-\tilde T\rfloor_{5,6,\sigma}^* \\
& \lesssim \frac{1}{\kappa^2} |r^{2}|_{2,2,\sigma}^* \lfloor T-\tilde T\rfloor_{5,6,\sigma}^* \\
& \lesssim \frac{1}{\kappa^2}  \lfloor T-\tilde T\rfloor_{5,6,\sigma}^*,
\end{align*}
which proves the claim for $\F_4$.

Finally, in the case of $\F_5$, for any $T, \tilde T \in B_{\Cout_1 \varepsilon^2} \subset \mathcal{\tilde X}_{5,6,\sigma}^*$,
\begin{align*}
|\F_5(T) - \F_5(\tilde T)|_{6,6,\sigma}^* = & \left|\frac{1}{R^2} (\partial_u T+\partial_u \tilde T) (\partial_u T-\partial_u \tilde T)  \right|_{6,6,\sigma}^* \\
& \le  \left|\frac{1}{R^2} (\partial_u T+\partial_u \tilde T)
\right|_{0,0,\sigma}^* \lfloor T-\tilde T\rfloor_{5,6,\sigma}^* \\
& \lesssim  \left|\frac{1}{R^2}\right|_{-4,-2,\sigma}^* \left|\partial_u T+\partial_u \tilde T
\right|_{4,2,\sigma}^* \lfloor T-\tilde T\rfloor_{5,6,\sigma}^* \\
& \lesssim \frac{1}{\varepsilon^2\kappa^2}\left|\partial_u T+\partial_u \tilde T
\right|_{6,6,\sigma}^* \lfloor T-\tilde T\rfloor_{5,6,\sigma}^* \\
& \lesssim \frac{1}{\kappa^2}  \lfloor T-\tilde T\rfloor_{5,6,\sigma}^*.
\end{align*}
This finishes the proof of the lemma, taking $\kappa_0$ large enough.
\end{proof}

Let $\Cout_0$ be the constant introduced in Lemma \ref{lem:F_at_0}. The following proposition summarizes the existence and properties of the solutions of the Hamilton-Jacobi equation~\eqref{HJpolarcoordinatesT1}.

\begin{proposition}
\label{prop:manifolds_in_the_outer_domain}
There exist $\kappa_0 \ge 1$, $0 <\varepsilon_0<1$ and a constant $\Cout_3>0$ such that,
for any  $\kappa \ge \kappa_0$ and $0<\varepsilon\le\varepsilon_0$, $\G^* \circ \Fout: B_{2\Cout_0 \varepsilon^2}\subset \mathcal{\tilde X}_{5,6,\sigma}^* \to B_{2\Cout_{0} \varepsilon^2}$ is a contraction with Lipschitz constant
$\mathrm{lip}\, \G^* \circ \Fout \le \Cout_3 / \kappa^2$, for $* = \bu$, $\bs$. Hence,
the Hamilton-Jacobi equation~\eqref{HJpolarcoordinatesT1} admits two solutions,
$T_1^* \in \mathcal{\tilde X}_{5,6,\sigma}^*$, $* = \bu$, $\bs$, such that
\[
\lfloor T_1^* \rfloor_{5,6,\sigma}^* \le 2\Cout_0{\varepsilon^2}.
\]
\end{proposition}

\begin{proof}
It is an immediate consequence of Lemmas~\ref{PropG},~\ref{lem:F_at_0} and~\ref{lem:Lipschitz_constant_of_Fout}.
\end{proof}

\section{Extension of the invariant manifolds} \label{ext:proof} 

Along this section, we will use the notation and domains introduced in Section~\ref{sec:furtherextension} without explicit mention. We denote by $\diagonal_{\alpha, \beta}$ the diagonal $2\times2$ matrix having $\alpha$ and $\beta$ as entries. Moreover, for all $\theta \in \T$, we recall definition~\eqref{def:in:Rtheta} of $R_\theta$ 
\begin{equation}\label{def:R_theta}
    R_\theta =\begin{pmatrix} \cos \theta & -\sin \theta \\ \sin \theta & \cos \theta \end{pmatrix}.
\end{equation}
For a given $\varrho>0$, $B_\varrho$ stands for the ball of radius $\varrho$ centered at the origin. 

We assume that all the parameters in the statement of the Propositions~\eqref{prop:ext1}, \eqref{prop:ext2}, and $\eqref{prop:ext3}$ are fixed. Furthermore, we denote by $\varepsilon_0$ and $\kappa_0$ two threshold values, which will be assumed to be suitably small and large, respectively, throughout this section. Similarly to Section \ref{sec:proof:outer}, $u \lesssim v$ denotes $u \le C v$ for a constant $C>0$.

\subsection{From graph to flow parametrization. }\label{sec:ext1:proof}  
In this section, we prove Proposition~\ref{prop:ext1}. 
We observe that equation~\eqref{ext1:inv_eq} has four components, and the unstable invariant manifold $\Gamma^\bu$ is Lagrangian. Hence, thanks to the symplectic character of the vector field $X_\mathcal{H}$, if two of these equations are satisfied, then the other two are satisfied as well. We choose to analyze the equations corresponding to the third and fourth components. 

We first claim that $f$ satisfies equation~\eqref{ext1:inv_eq} if and only if
\begin{equation}\label{ext1:inv_eq2}
    \Lext f = \Fext(f)
\end{equation}
where, 
\begin{equation}\label{defLF_extension}
\begin{aligned}
    \Lext f &= \partial_v f + {\omega \over \varepsilon} \partial_\varphi f, \\ \Fext(f) &= -\left(\diagonal_{{1 \over \dot r^2},{1 \over r^2}} \nabla T_1^{\bu} + \diagonal_{{1 \over \dot r},{1 \over r}}R_{-\theta} \partial_x H_1 \circ \Gamma^{\bu}\right) \circ \left(\mathrm{id} + f\right)
    \end{aligned}
\end{equation}
with $H_1$ defined in~\eqref{def:H0H1}, $T_1^{\bu}$ given by Theorem~\ref{thm:outer} and $\Gamma^\bu$ defined in~\eqref{thm:def_Gamma}. From now on we dropped down the apex $\bu$ in $T_1^{\bu}, \Gamma^\bu$. Indeed, first, a straightforward computation shows that equation~\eqref{ext1:inv_eq} is equivalent to 
\begin{equation*}
    \left(R_\theta \diagonal_{\dot r, r}\right) \circ \left(\mathrm{id} + f\right) \left(\begin{pmatrix}
        1 \\ {\omega \over \varepsilon} \end{pmatrix} + \Lext f\right) = - \partial_x \mathcal H \circ \Gamma \circ \left(\mathrm{id} + f\right) 
\end{equation*}
with $\mathcal{H}$ the Hamiltonian defined by~\eqref{def:rescaledHamiltonian}. Multiplying both sides of the latter by the inverse of the matrix $R_\theta \diagonal_{\dot r, r}$ one has that 
\begin{equation}\label{ext1:proof_inv_eq}
\Lext f = - \begin{pmatrix} 1 \\ {\omega \over \varepsilon} \end{pmatrix} - \left(\diagonal_{{1 \over \dot r},{1 \over r}}R_{-\theta} \partial_x \mathcal H \circ \Gamma\right)\circ \left(\mathrm{id} + f\right).
\end{equation}
We can rewrite $\partial_x \mathcal H \circ \Gamma$ as
\begin{equation*}
    \partial_x \mathcal H \circ \Gamma = \partial_x H_0 \circ \Gamma_0 + \partial_x H_0 \circ \Gamma - \partial_x H_0 \circ \Gamma_0 +\partial_x H_1 \circ \Gamma
\end{equation*}
where we refer to Lemma \ref{lem:homoclinic} for the definition of $\Gamma_0 = (\delta_0,\gamma_0)^{\top}$. Using definition~\eqref{def:H0H1} of $H_0,H_1$, and Lemma~\ref{lem:homoclinic}, one can easily see that 
\begin{equation*}
\begin{aligned}
   &- \diagonal_{{1 \over \dot r},{1 \over r}}R_{-\theta} \partial_x H_0 \circ \Gamma_0 = \begin{pmatrix} 1 \\ {\omega \over \varepsilon} \end{pmatrix}, \\
   &\partial_x H_0 \circ \Gamma - \partial_x H_0 \circ \Gamma_0 +\partial_x H_1 \circ \Gamma = (D\, \gamma_0)^{-\top} \nabla T_1 + \partial_x H_1 \circ \Gamma.
   \end{aligned}
\end{equation*}
Observing that $(D\, \gamma_0)^{-\top}   = R_\theta \diagonal_{{1 \over \dot r},{1 \over r}}$, and replacing the latter into~\eqref{ext1:proof_inv_eq}, we have the claim. 

We find a solution of equation~\eqref{ext1:inv_eq2} by the fixed point theorem. We first rewrite equation~\eqref{ext1:inv_eq2} as a fixed point equation using a suitable right inverse $\mathcal{G}$ of the operator $\Lext$ (see Section \ref{sec:ext_prel}). Then, we analyze the functional $\Fext$ and we prove that $\mathcal{G} \circ \Fext$ is a contraction defined on a suitable closed subset of a special Banach space (see Section \ref{sec:ext_fixed_point}).

\subsubsection{Preliminaries and technical lemma} \label{sec:ext_prel}
We consider the following Banach space
\begin{equation}\label{def:ext1_BS}
\begin{aligned}
    \mathcal{Y}^{\mathrm{gf}} &= \Big\{f : \Dgf_{2\kappa} \times \T_{3\sigma \over 4} \to \C : \mbox{$f$ is analytic and}\\
    & \hspace{48mm} |f|_\infty = \sup_{(v, \varphi) \in \Dgf_{2\kappa} \times \T_{3\sigma \over 4}} |f(v, \varphi)| < \infty\Big\}.
\end{aligned}
\end{equation}
We will use the same notation for vector-valued functions and matrices. That is, in the case of functions $f=(f_1,...,f_a)^\top:\Dgf_{2\kappa} \times \T_{3\sigma \over 4} \to \C^a$, we set $|f|_\infty = \max_{1 \le i \le a}|f_i|_\infty$. On the other hand, when we consider matrices $M=\{M_{ij}\}_{1 \le i \le i, 1 \le j \le b }:\Dgf_{2\kappa} \times \T_{3\sigma \over 4} \to \mathcal{M}_{a \times b}(\C)$, we det $|M|_\infty = \max_{1 \le i \le a, 1 \le j \le b}|M_{ij}|_\infty$. We recall that $\mathcal{M}_{a \times b}(\C)$ denotes the space of $a \times b$ matrices with complex coefficients. 

Given $g : \Dgf_{2\kappa} \times \T_{3\sigma \over 4} \to \C$, we look for solutions of the following equation
\begin{equation*}
    \Lext(\phi) = g.
\end{equation*}
Writting $g = \sum_{\ell \in \Z} g^{[\ell]}(v) e^{i\ell\theta}$, with $g^{[\ell]}(v)$ the Fourier coefficient of $g$, a formal solution $\mathcal{G}(g) = \sum_{\ell \in \Z} \mathcal{G}^{[\ell]}(g)(v) e^{i\ell\theta}$ of $\Lext(\phi)=g$ is
\begin{equation}\label{def:right_inverse_ext1}
    \begin{aligned}
       &\mathcal{G}^{[\ell]}(g)(v) = \int_{v_0}^v e^{{\omega \over \varepsilon}i \ell (\tau-v)}g^{[\ell]}(\tau)d\tau \hspace{10mm} \mbox{if $\ell >0$}\\ 
       &\mathcal{G}^{[0]}(g)(v) = \int^{v}_\rho g^{[0]}(\tau)d\tau \hspace{24mm} \mbox{if $\ell =0$}\\ 
       &\mathcal{G}^{[\ell]}(g)(v) = \int_{\bar v_0}^v e^{{\omega \over \varepsilon}i \ell (\tau-v)}g^{[\ell]}(\tau)d\tau \hspace{10mm} \mbox{if $\ell <0$}
    \end{aligned}
\end{equation}
where $v_0$ and $\rho$ are the topmost and leftmost points of the domain $\Dgf_{2\kappa}$ (see Figure~\ref{fig:ext_domain_step1and2}). 
\begin{lemma}\label{lem:ext1_right_inv}
    For all $g \in  \mathcal{Y}^{\mathrm{gf}}$, there exists a positive constant $C$ such that 
    \begin{equation*}
        |\mathcal{G}(g)|_\infty \le C|g|_{\infty}. 
    \end{equation*}
\end{lemma}
\begin{proof}
    The proof of this lemma is a direct consequence of the definition of the operator $\mathcal{G}$ and the fact that the domain $\Dgf_{2\kappa}$ only contains points at distance of order $1$ from the singularities $\pm i{\pi \over 2}$.
\end{proof}

\subsubsection{The fixed point equation} \label{sec:ext_fixed_point}
Using the operator defined in~\eqref{def:right_inverse_ext1}, equation~\eqref{ext1:inv_eq2} can be reformulated as the fixed point equation 
\begin{equation}\label{ext1:inv_eq3}
f = \mathcal{G} \circ \Fext(f).
\end{equation}
To prove Proposition \ref{prop:ext1}, we show that the right-hand side of the above equation is a contraction in a suitable closed subset of $\mathcal{Y}^{\mathrm{gf}} \times \mathcal{Y}^{\mathrm{gf}}$ (see~\eqref{def:ext1_BS}). 
For this purpose, for all $f = (f_1, f_2)^\top \in \mathcal{Y}^{\mathrm{gf}} \times \mathcal{Y}^{\mathrm{gf}}$, we introduce the following norm
\begin{equation}\label{ext1: norm_infty_2}
 \|f\|_\infty = \|(f_1,f_2)\|_\infty = |f_1|_\infty + \varepsilon |f_2|_\infty.
\end{equation}
The proof is structured through several lemmas.
\begin{lemma}\label{lem:ext1_first_iter}
    There exist $\kappa_0\ge 1$ and $0<\varepsilon_0<1$ and a constant $c_0>0$ such that for any $\kappa \ge \kappa_0$ and $0 <\varepsilon \le \varepsilon_0$
    \begin{equation*}
        \left\|\mathcal{G}\circ \Fext(0)\right\|_\infty \le c_0 \varepsilon^2
    \end{equation*}
    where $\Fext$ is defined in~\eqref{defLF_extension}.
\end{lemma}
\begin{proof}
Along this proof, we skip the apex $-\mathrm{gf}-$. We first observe that 
    \begin{equation}
    \begin{aligned}\label{eq:ext1_F01}
        &\mathcal{F}_1(0) = - {1 \over \dot r^2} \partial_u T_1 - {\cos \theta \over \dot r} \partial_{x_1} H_1 \circ \Gamma - {\sin \theta \over \dot r} \partial_{x_2} H_1 \circ \Gamma\\
        &\mathcal{F}_2(0) = - {1 \over r^2} \partial_\theta T_1 - {\cos \theta \over r} \partial_{x_2} H_1 \circ \Gamma + {\sin \theta \over r} \partial_{x_1} H_1 \circ \Gamma.
    \end{aligned}
    \end{equation}
    Using Theorem \ref{thm:difference_between_manifolds} and Proposition \ref{prop:main_rescaling} 
    \begin{equation}\label{eq:ext1_F02}
    \begin{aligned}
        {\cos \theta \over \dot r} \partial_{x_1} H_1 \circ \Gamma + {\sin \theta \over \dot r} \partial_{x_2} H_1 \circ \Gamma &= {1 \over \dot r}\left(\cos \theta \partial_{x_1} \hat H_1 \circ \Gamma + \sin \theta  \partial_{x_2} \hat H_1 \circ \Gamma\right)\\
        {\cos \theta \over r} \partial_{x_2} H_1 \circ \Gamma - {\sin \theta \over r} \partial_{x_1} H_1 \circ \Gamma &= - \varepsilon \beta r^2 + \varepsilon^2 2 \alpha \partial_\theta T_1 \\
        &+ {1 \over r}\left(\cos \theta \partial_{x_2} \hat H_1 \circ \Gamma - \sin \theta  \partial_{x_1} \hat H_1 \circ \Gamma\right).
    \end{aligned}
    \end{equation}
By Proposition \ref{prop:main_rescaling}, we know that $\partial_x \hat H_1 (x,y;\varepsilon) = \varepsilon^{-2} \partial_x \tilde H_1(\varepsilon^2 x, \varepsilon y;\varepsilon)$ and $\partial_x \tilde H_1(x, y;\varepsilon) = \mathcal{O}_5(x,y)$ uniformly in $\varepsilon$, which implies
\begin{equation*}
    \left|\partial_{x} \hat H_1 \circ \Gamma\right|_\infty \lesssim \varepsilon^3.
\end{equation*}
Replacing~\eqref{eq:ext1_F02} into~\eqref{eq:ext1_F01}, thanks to the latter, Theorem \ref{thm:outer} and the fact that the domain $\Dgf_{2\kappa}$ only contains points at distance of order $1$ from the singularities $\pm i{\pi \over 2}$, we have that 
\begin{align*}
    &\left|\mathcal{F}_1(0)\right|_\infty \lesssim \left|\partial_u T_1\right|_\infty + \left|\partial_x \hat H_1 \circ \Gamma\right|_\infty \lesssim \varepsilon^2\\
     &\left|\mathcal{F}_2(0)\right|_\infty \lesssim \varepsilon |r|_\infty^2 + \varepsilon^2\left|\partial_\theta T_1\right|_\infty + \left|\partial_x \hat H_1 \circ \Gamma\right|_\infty \lesssim \varepsilon.
\end{align*}
Using these estimates, definition~\eqref{ext1: norm_infty_2}, and Lemma \ref{lem:ext1_right_inv}, the proof is complete.
\end{proof}

Let $\rho$ be a positive parameter. We denote by $B_\varrho^\times$ the closed ball in $\mathcal{Y}^{\mathrm{gf}} \times \mathcal{Y}^{\mathrm{gf}}$ of radius $\varrho$ centered at the origin with respect to the norm $\| \cdot \|_\infty$. 

\begin{lemma}\label{lem:ext1_contract} Given $c_1>0$, there exist $\kappa_0\ge 1$, $0 <\varepsilon_0<1$ and a constant $c_2>0$ such that for any $\kappa \ge \kappa_0$ and $0 < \varepsilon \le \varepsilon_0$, the functional $\Fext: B_{c_1\varepsilon^2}^\times\subset \mathcal{Y}^{\mathrm{gf}}\times \mathcal{Y}^{\mathrm{gf}} \to  \mathcal{Y}^{\mathrm{gf}}\times \mathcal{Y}^{\mathrm{gf}}$ is Lipschitz  with
$\mathrm{lip}\, \Fext \le c_2 / \kappa$.
\end{lemma}
\begin{proof} We skip the apex $-\mathrm{gf}-$ along the proof. 
We first notice that, if $(v,\varphi) \in \Dgf_{2\kappa} \times \T_{3\sigma \over 4}$ and $f=(f_1,f_2)^\top \in B_{c_1\varepsilon^2} \times B_{c_1\varepsilon}$, then 
$(v,\varphi) + f(v,\varphi) \in {D}_{\kappa}^{\mathrm{out},\bu} \times \T_\sigma$, if $\varepsilon_0$ is small enough and $\kappa_0$ large enough, therefore the operator $\mathcal{F}$ is well defined on $B_{c_1\varepsilon^2} \times B_{c_1\varepsilon}$. 

We introduce 
\[
F^1(u,\theta)= -\diagonal_{{1 \over \dot r^2},{1 \over r^2}} \nabla T_1 (u,\theta), \qquad 
F^2(u,\theta)= -\diagonal_{{1 \over \dot r},{1 \over r}}R_{-\theta} \partial_x H_1 \circ \Gamma (u,\theta).
\]
By the definition of $\mathcal{F}$ in~\eqref{defLF_extension}, we can write 
    \begin{equation*}
        \mathcal{F} = F^1 \circ \left(\mathrm{id} + f\right)+ F^2\circ \left(\mathrm{id} + f\right).
    \end{equation*}

For any $f = (f^1, f^2)^\top \in B^\times_{c_1\varepsilon^2} \subset \mathcal{Y}^{\mathrm{gf}}\times \mathcal{Y}^{\mathrm{gf}}$, we denote by $f_\tau = f^2 + \tau(f^1-f^2)$. In addition we set $z=(v,\varphi)$.
By the mean value theorem
\begin{equation}\label{mean_ext_1}
\mathcal{F}(f^1)(z) - \mathcal{F}(f^2)(z) =  \sum_{i=1}^2 \int_{0}^1 D F^i (z+f_\tau(z)) (f^1(z)-f^2(z)) \, d\tau .
\end{equation}

Let $0 <c<1$ be a constant such that 
\begin{align*}
\widehat{D}_{\kappa} := \left \{ u\in D^{\mathrm{out},\bu}_{\kappa}\,:\, \exists v\in \Dgf_{2\kappa}, \; |u-v| \leq c \kappa \varepsilon\right \} \subset D_{\kappa}^{\mathrm{out},\bu}.
\end{align*}

We notice that if $u\in \widehat{D}_\kappa$, $\left | u^2 + \frac{\pi^2}{4}\right |^{-1} \lesssim 1$. Then, by Theorem~\ref{thm:outer}, taking if necessary $\varepsilon_0$ small enough, 
\begin{equation}\label{boundF1_ext}
\sup_{(u,\theta) \in \widehat{D} _\kappa \times \T_{\sigma}} \left | F^1 (u,\theta)\right |  \lesssim \varepsilon^2.
\end{equation}

In addition, from Theorem \ref{thm:difference_between_manifolds} and Proposition \ref{prop:main_rescaling}, 
\begin{align*}
F^2(u,\theta) = - \begin{pmatrix}    \displaystyle{   {\cos \theta \over \dot r} \partial_{x_1} \widehat H_1 \circ \Gamma + {\sin \theta \over \dot r}  \partial_{x_2} \widehat H_1 \circ \Gamma }\\ \\ \displaystyle{
     - \varepsilon \beta r^2 + \varepsilon^2 2 \alpha \partial_\theta T_1 + {\cos \theta \over r} \partial_{x_2} \widehat H_1 \circ \Gamma - {\sin \theta\over r}  \partial_{x_1} \widehat H_1 \circ \Gamma }\end{pmatrix} 
\end{align*}
where we recall that by Proposition \ref{prop:main_rescaling}, $\widehat H_1(x,y;\varepsilon) = \varepsilon^{-4} \tilde H_1(\varepsilon^2x, \varepsilon y;\varepsilon)$, and $\tilde H_1(x, y;\varepsilon) = \mathcal{O}_6(x,y)$ uniformly in $\varepsilon$. This implies that 
\begin{equation*}
  \sup_{(u,\theta) \in \widehat{D} _\kappa \times \T_{\sigma}}   \left|\partial_{x} \widehat H_1 \circ \Gamma (u,\theta)\right|  \lesssim \varepsilon^3 
\end{equation*}
and therefore
\begin{equation}\label{boundF2_ext}
\sup_{(u,\theta) \in \widehat{D} _\kappa \times \T_{\sigma}} \left | F^2 (u,\theta)\right |  \lesssim \varepsilon .
\end{equation} 
Taking $\kappa_0$ and $\varepsilon_0$ large and small enough, respectively, if $z=(v,\varphi) \in \Dgf_{2\kappa} \times \T_{\frac{3\sigma}{4}}$ then
\[
\Omega:=\left \{w \in \mathbb{C} \times \T_\sigma \,:\, |z+f_{\tau}(z) -w  |\leq \frac{c}{2} \kappa \varepsilon \right \} \subset \widehat{D}_\kappa \times \T_\sigma
\]
and therefore, using Cauchy estimates in $\Omega$ and bounds~\eqref{boundF1_ext} and~\eqref{boundF2_ext}
\[
\sup_{z\in \Dgf_{2\kappa} \times \T_{3 \sigma \over 4} } \big | DF^1 (z+f_\tau(z)) \big|  \lesssim \frac{\varepsilon}{\kappa} , \qquad \sup_{z\in \Dgf_{2\kappa}\times \T_{3\sigma \over 4}} \big | DF^2(z+f_\tau(z)) \big | \lesssim \frac{1}{\kappa} 
\]
Using expression~\eqref{mean_ext_1} for $\mathcal{F}(f^1) - \mathcal{F}(f^2)$, and~\eqref{ext1: norm_infty_2}, the proof of the lemma is finished. 
\end{proof}

 Let $c_0$ be the constant introduced in Lemma \ref{lem:ext1_first_iter}. The following proposition establishes the existence and describes the properties of the solutions of equation~\eqref{ext1:inv_eq3}.
\begin{proposition}\label{prop_ext1_final}
    There exist $\kappa_0\ge 1$, $0 <\varepsilon_0< 1$ and a constant $c_3>0$ such that, for any $\kappa \ge \kappa_0$ and $0 <\varepsilon \le \varepsilon_0$, $\mathcal{G} \circ \Fext : B_{2 c_0\varepsilon^2}^\times \subset \mathcal{Y}^{\mathrm{gf}} \times \mathcal{Y}^{\mathrm{gf}} \to B_{2 c_0\varepsilon^2}^\times$ is a contraction with Lipschitz constant $\mathrm{lip}\, \G \circ \Fext \le c_3 / \kappa$. Hence, the equation~\eqref{ext1:inv_eq3} admits a solution $f=(f_1, f_2)^\top \in  \mathcal{Y}^{\mathrm{gf}}\times \mathcal{Y}^{\mathrm{gf}}$ satisfying 
    \begin{equation*}
        \left\|f \right\|_\infty \le 2c_0\varepsilon^2.
    \end{equation*}
    Moreover, 
    \begin{equation}\label{prop:ext1_sub}
        \sup_{(v, \varphi) \in \Dgf_{2\kappa} \times \T_{3 \sigma \over 4}} |\Ggf(v, \varphi) - \Gamma_0(v, \varphi)|\le c  \varepsilon.
    \end{equation}
    where $\Ggf$ is the flow parameterization of the unstable invariant manifold defined by~\eqref{def:ext_tilde_Gamma:bis} and  $\Gamma_0$ is the unperturbed homoclinic defined by Lemma \ref{lem:homoclinic}. 
\end{proposition}
\begin{proof}
    The proof of this proposition is a direct consequence of the previous lemmas. Indeed, combining Lemma \ref{lem:ext1_right_inv} and Lemma \ref{lem:ext1_contract} for $\kappa_0$ and $\varepsilon_0$ large and small enough, respectively, one can prove that the operator $\mathcal{G}\circ \Fext$ is Lipschitz with $\mathrm{lip}\, \mathcal{G}\circ \Fext \leq Cc_2 \kappa^{-1}$ which is smaller than $1$ for $\kappa_0$ large enough. Furthermore, for all $f \in B_{2 c_0\varepsilon^2}^\times$, using Lemmas~\ref{lem:ext1_right_inv}, \ref{lem:ext1_first_iter} and~\ref{lem:ext1_contract}, we observe that 
    \begin{align*}
 \left \| \mathcal{G} \circ \Fext (f) \right \|_\infty & \leq \left \| \mathcal{G} \circ \Fext (0) \right \|_\infty + \left \| \mathcal{G} \circ \Fext (f) - \mathcal{G}\circ \Fext (0)\right \|_\infty \\ & \leq c_0  \varepsilon^2 + \frac{C c_2}{\kappa} \|f\|_\infty \leq  2c_0 \varepsilon^2 
 \end{align*}
 for $\kappa_0$ large enough. Thus
  $\mathcal{G} \circ \Fext : B_{2c_0 \varepsilon^2}^\times  \to B_{2c_0 \varepsilon^2}^\times $ is a contraction.  This concludes the proof of the first part of this proposition.

It remains to verify~\eqref{prop:ext1_sub}. Using the Taylor theorem, we observe that 
\begin{align*}
\Ggf(v, \phi) =& \Gamma^{\bu}\circ \left(\mathrm{id} + f\right)(v, \phi)\\
        =&\Gamma_0 (v, \phi) + \int_0^1D\Gamma_0 \circ (\mathrm{id} + \tau f)(v, \phi) d\tau \cdot f(v, \phi)\\
         \label{def:ext2_proof_TildeGamma}
        &+ \left(\Gamma^{\bu} -  \Gamma_0\right)\circ (\mathrm{id} + f)(v, \phi)
\end{align*}
for all $(v, \varphi) \in \Dgf_{2\kappa} \times \T_{3 \sigma \over 4}$.  
Hence, combining Theorem~\ref{thm:outer} with the estimate $\|f\|_\infty \lesssim \varepsilon^2$, we conclude the proof of~\eqref{prop:ext1_sub} and thus the proposition.
\end{proof}

The proof of Proposition~\ref{prop:ext1} is a direct consequence of Proposition \ref{prop_ext1_final}.

\subsection{Extension of the flow parameterization}\label{sec:ext2:proof}

The aim of this section is to prove Proposition~\ref{prop:ext2}, that is, we look for $\widehat{\Gamma}$, a solution of equation~\eqref{equation:flow}, such that $\widehat{\Gamma}(u,\theta) = \Ggf(u,\theta)$
when $u \in \Dgf_{2\kappa} \cap D^{\mathrm{fl}}_{3\kappa}$ (see~\eqref{def:ext_tilde_Gamma:bis}, \eqref{def:domains_ext_step1}, \eqref{def:domains_ext_step2} and Figure \ref{fig:ext_domain_step1and2} for the definition of $\Ggf, \Dgf_{2\kappa}$ and $D^{\mathrm{fl}}_{3\kappa}$ respectively). 
%
%


Let $\widehat \Gamma_1  = \widehat\Gamma  - \Gamma_0$. 
A straightforward computation ensures that equation~\eqref{equation:flow} can be rewritten as  
\begin{equation}\label{ext2:inv_eq}
    \widehat{\mathcal{L}} \, \Gamma  = \Ffl (\Gamma )
\end{equation}
where
\begin{align} \label{def:Lhat:Fhat}
    \widehat{\mathcal{L}} \,  \Gamma  &= \Lfl \,  \Gamma   - \left(DX_{H_0} \circ \Gamma_0\right) \Gamma  ,\\
    \Ffl(\Gamma  ) &= X_{H_0} \circ \big ( \Gamma_0 +  \Gamma \big )  - X_{H_0} \circ \Gamma_0 - \left(DX_{H_0} \circ \Gamma_0\right)  \Gamma + X_{H_1} \circ \big (\Gamma_0+ \Gamma \big )  . \notag
\end{align}

For brevity, we write $\Gamma$ in place of $\widehat{\Gamma}_1$. We also recall that $\Lfl = \partial_u + \frac{\omega}{\varepsilon} \partial_\theta$ (see~\eqref{equation:flow}) and $X_{H_0}, X_{H_1}$ denote the vector fields associated to the Hamiltonian $H_0, H_1$ defined in Proposition \ref{prop:main_rescaling}.  

The proof of Proposition~\ref{prop:ext2} relies on a fixed point argument. First, in Section~\ref{sec:ext2_prel}, we look a right inverse $\Ghat$ of the operator $\Lhat$. We need to define $\Ghat$ carefully in order to fix some initial conditions that guarantee that the solution of~\eqref{ext2:inv_eq} coincides with $\Ggf - \Gamma_0$ in a suitable open subset. In the second part, we verify that the operator $ \Ghat \circ \Ffl$ is a contraction defined on a suitable closed subset of a specific Banach space (see Section \ref{sec:ext2_fixed_point}).

\subsubsection{Preliminaries and technical lemmas}\label{sec:ext2_prel}

We define the Banach space
\begin{align}\label{def:ext2_BS}
    \mathcal{Y}^{\mathrm{fl}} &= \Big\{f : D^{\mathrm{fl}}_{3 \kappa} \times \T_{3\sigma \over 4} \to \C : \mbox{$f$ is analytic and} \nonumber\\
    &\hspace{44mm} |f|_\infty = \sup_{(u, \theta) \in D^{\mathrm{fl}}_{3 \kappa} \times \T_{3\sigma \over 4}} |f(u, \theta)| < \infty\Big\}.
\end{align}
We will use the same notation for vector-valued functions and matrices. We say that a vector-valued function or a matrix belongs to the space $\mathcal{Y}^{\mathrm{fl}}$ if it is the case for each component.  As in Section \ref{sec:ext_prel}, the norm $|\cdot|_\infty$ of a vector-valued function or a matrix is defined as the maximum of the norms $|\cdot|_\infty$ of its components.

First, we prove the following technical result.
\begin{proposition}\label{prop:ext2_RI1}
    There exists a matrix $M \in \mathcal{M}_4(\C)$ with $M\in \mathcal{Y}^{\mathrm{fl}}$ such that $\Lhat\, M=0$. Moreover, there exists a constant $C>0$ independent of $\varepsilon$ such that 
    \begin{equation}\label{prop:est_M}
        |M|_\infty \le C, \qquad \mbox{and} \qquad |M^{-1}|_\infty \le C.
    \end{equation}
\end{proposition}
\begin{proof}
We first compute $DX_{H_0}(\Gamma_0(u,\theta))$. Letting  
\[
f(u,\theta)= -1 + 2r^2(u) (3\cos^2 \theta +  \sin^2 \theta), \qquad 
g(u,\theta)=4 r^2(u) \cos \theta \sin \theta
\]
and
\[
h(u,\theta) = -1 + 2 r^2(u)(\cos^2 \theta +  3\sin^2 \theta),
\]
a straightforward computation proves that 
\[
DX_{H_0}(\Gamma_0(u,\theta))=  \begin{pmatrix}
    \frac{\omega}{\varepsilon} J & A(u,\theta) \\
    -\mathrm{id} & \frac{\omega}{\varepsilon} J
\end{pmatrix}=
\begin{pmatrix}
    0 & - \frac{\omega}{\varepsilon} & f(u,\theta) & g(u,\theta) \\
    \frac{\omega}{\varepsilon} & 0 & g(u,\theta) & h(u,\theta) \\
    -1 & 0 & 0 &-\frac{\omega}{\varepsilon} \\ 
    0 & -1 & \frac{\omega}{\varepsilon} & 0
\end{pmatrix}.
\]
We note that $\partial_u \Gamma_0$ is a solution of $\Lhat \, \Gamma =0$ (see~\eqref{def:Lhat:Fhat}) and we recall that $\dot{r}(u)=-R(u)$. We mimmic the shape of 
\[
\partial_u \Gamma_0(u,\theta)= \big (\dot{R}(u) \cos \theta, \dot{R}(u) \sin \theta, - R(u) \cos \theta, -R(u) \sin \theta)^\top, 
\]
to find another solution of $\Lhat \, \Gamma =0$. Indeed, let
\begin{equation}\label{def:zeta:1}
\begin{aligned}
\Gamma_1 (u,\theta ) & = (\xi (u,\theta), \eta(u,\theta)) \\ & = \big (-\dot{\varrho}(u) \cos \theta, -\dot{\varrho}(u) \sin \theta, \varrho (u) \cos \theta , \varrho(u) \sin \theta)^\top.
\end{aligned}
\end{equation}
It is straightforward to check that 
\[
\Lfl\eta = - \xi + \frac{\omega}{\varepsilon} J \eta   
\]
independently of the choice of $\varrho$. 

Then imposing that 
\[
\Lfl\xi = \frac{\omega}{\varepsilon} J \xi + A(u,\theta) \eta
\]
one obtains that $\varrho$ has to satisfy
\begin{equation}\label{eq:variational:polars}
\ddot{\varrho}  = (1- 6r^2(u)) \varrho.
\end{equation}
Since trivially $\dot{r}$ satisfy~\eqref{eq:variational:polars}, see for instance Remark~\ref{rmk:homoclinic} (in fact, we recover $\partial_u \Gamma_0$ considering $\varrho=\dot{r}$), 
using the standard reduction of order method, a new solution of~\eqref{eq:variational:polars} is
\begin{equation}\label{def:varrho:1}
\varrho(u)= \dot{r}(u) \int_{\rho_2}^u \frac{1}{\dot{r}^2(s)} ds  
\end{equation}
with $\rho_2\in \mathbb{R} \cap D^{\mathrm{fl}}_{3 \kappa}$ the leftmost point of $D^{\mathrm{fl}}_{3 \kappa}$. Notice that $\dot{r}(u)=-\sinh u (\cosh u)^{-2}$ only vanishes on $D^{\mathrm{fl}}_{3 \kappa}$ when $u=0$, therefore, in the definition of $\varrho$, we take an integration path avoiding $u=0$. Since $u=0$ is a simple zero of $\dot{r}(u)$, namely, $\dot{r}(u)=-u + \mathcal{O}(u^2)$, the definition of $\varrho(u)$ does not depend on the integration path (there is no residue) and we conclude that it is analytic in $D^{\mathrm{fl}}_{3 \kappa}$.  
Notice that $\lim_{u\to 0} \varrho(u)= 1$ and 
\begin{equation}\label{det:1}
\left | \begin{array}{cc} \dot{r}(u) & \varrho(u) \\ \ddot{r}(u) & \dot{\varrho}(u) \end{array} \right | = \left | \begin{array}{cc} -R(u) & \varrho(u) \\ -\dot{R}(u) & \dot{\varrho}(u) \end{array} \right | = 1. 
\end{equation}
We conclude that $\Gamma_1$ defined as~\eqref{def:zeta:1} with $\varrho$ in~\eqref{def:varrho:1} is an analytic solution of $\Lhat \, \Gamma =0$. 

We proceed analogously taking into account that 
\[
\partial_\theta \Gamma_0(u,\theta)= \big (\dot{r} \sin \theta, -\dot{r} \cos \theta, - r \sin \theta, r \cos \theta \big )^\top
\]
is also a solution of $\Lhat \, \Gamma =0$. Indeed, it is easily checked that the function
\begin{equation} \label{def:zeta:2}
\Gamma_2 (u,\theta ) =   \big (- \dot{\chi}(u) \sin \theta, \dot{\chi}(u) \cos \theta, \chi (u) \sin \theta , -\chi(u) \cos \theta \big )^\top
\end{equation}
is a solution of $\Lhat \, \Gamma =0$ if and only if $\chi$ satisfies 
\[
\ddot{\chi} = (1- 2 r^2(u)) \chi. 
\]
From Remark~\ref{rmk:homoclinic}, $\chi (u)= r(u)$ is a solution and then another solution is given by 
\begin{equation}\label{def:chi:var}
\chi(u)=r(u) \int_{\rho_2}^u \frac{1}{r^2(s)} ds
\end{equation}
which is clearly analytic provided $r(u) \neq 0$ for $u\in D^{\mathrm{fl}}_{3 \kappa}$ (recall that $\rho_2$ is the leftmost point of this region). As before (see~\eqref{det:1})
\begin{equation}\label{det:2}
\left | \begin{array}{cc} {r}(u) & \chi(u) \\ \dot{r}(u) & \dot{\chi}(u) \end{array} \right | =1. 
\end{equation}

We define
\begin{equation*}
   M(u,\theta)= \begin{pmatrix} \partial_u \Gamma_0(u,\theta), \partial_\theta \Gamma_0(u,\theta), \Gamma_1(u,\theta), \Gamma_2(u,\theta) \end{pmatrix}
\end{equation*}
with $\Gamma_1, \Gamma_2$ defined in~\eqref{def:zeta:1} and~\eqref{def:zeta:2} with 
$\varrho,\chi$ in~\eqref{def:varrho:1} and~\eqref{def:chi:var}. It is clear that  
$|M |_\infty \lesssim 1$ because for $u \in D^{\mathrm{fl}}_{3 \kappa}$, $\left |u^2 + \frac{\pi^2}{4}\right| \gtrsim 1$. Recalling that $R= -\dot{r}$ and using~\eqref{det:1} and~\eqref{det:2}, 
\begin{align*}
    \mathrm{det} \big (M(u,\theta) \big ) = & e^{i\theta} e^{-i\theta} e^{i\theta} e^{-i\theta} \left |\begin{array}{cccc}
\dot{R}  &  \dot{r}  & - \dot{\varrho}  & -\dot{\chi} \\
\dot{R} & - \dot{r} & - \dot{\varrho} & \dot{\chi} \\
-R & -r & \varrho & \chi \\
-R & r & \varrho & -\chi
 \end{array}
    \right | \\ = & 4 (\dot{R} \varrho - R \dot{\varrho} ) ( \dot{r} \chi - \dot{\chi} r)  =-4.
\end{align*}
Therefore $M $ is a fundamental matrix of $\Lhat \,\Gamma = 0$ and $|M|_\infty, |M^{-1}|_\infty \lesssim 1$.  
\end{proof}

Let $\mathcal{G}$ be the right inverse of $\Lfl$, $\Lfl \circ \mathcal{G}(g)=g$, defined as in~\eqref{def:right_inverse_ext1} but replacing $v_0$ and $\rho$ by $u_0$ and $\rho_2$ that are the topmost and leftmost points of  $D^{\mathrm{fl}}_{3 \kappa}$.  
\begin{lemma}\label{lem:ext2_right_inv}
    There exists a positive constant $C$ such that for all $g \in  \mathcal{Y}^{\mathrm{fl}}$, 
    \begin{equation*}
        |\mathcal{G}(g)|_\infty \le C|g|_{\infty}. 
    \end{equation*}
\end{lemma}
\begin{proof}
    As in the proof of Lemma \ref{lem:ext1_right_inv}, the proof of this result is straightforward from the fact that the domain $D^{\mathrm{fl}}_{3 \kappa}$ only contains points away from $\pm i \frac{\pi}{2}$.
\end{proof}
 It is immediate to verify that, by Proposition \ref{prop:ext2_RI1} and Lemma~\ref{lem:ext2_right_inv},  
 \begin{equation*}
     \Ghat_0 (\Gamma) = M \mathcal{G}\left(M^{-1}\Gamma\right) 
 \end{equation*}
is a right inverse of the operator $\Lhat$. However, we can not guarantee that a solution of $\Gamma = \Ghat_0 \circ \Ffl \Gamma$ is the analytic continuation of $\Ggf - \Gamma_0$. To overcome this difficulty, we observe that if $\phi $ satisfies that $\Lfl \phi =0$, then   
$
\Lfl (M \phi) = 0
$
and therefore $M \phi + \Ghat_0 (g)$ is also a right inverse of $\Lhat$. We set
$\phi (u,\theta) = \sum_{\ell \in \mathbb{Z}} \phi^{[\ell]}(u) e^{i\ell \theta}$ with 
\begin{align}\label{def:ext2_g}
    \phi^{[\ell]}(u) &= e^{i{\omega \over \varepsilon}\ell(u_0 - u)}\left(M^{-1}\left(\Ggf - \Gamma_0\right)\right)^{[\ell]}(u_0) \quad \mbox{if $\ell >0$,}\nonumber\\
    \phi^{[0]}(u) &= \left(M^{-1}\left(\Ggf - \Gamma_0\right)\right)^{[0]}(\rho_2), \\
    \phi^{[\ell]}(u) &= e^{i{\omega \over \varepsilon}\ell(\bar u_0 - u)}\left(M^{-1}\left(\Ggf - \Gamma_0\right)\right)^{[\ell]}(\bar u_0) \quad \mbox{if $\ell <0$}\nonumber
\end{align}
where $u_0$ and $\rho_2$ that are the topmost and leftmost points of the domain $D^{\mathrm{fl}}_{3 \kappa}$ (see Figure~\ref{fig:ext_domain_step1and2}. It is important to emphasize that 
\begin{equation}\label{u0rho2flow}
u_0,\rho_2, \bar{u}_0 \in  \Dgf_{2\kappa}. 
\end{equation}
We define the following right inverse of $\Lhat$:
\begin{equation}\label{def:ext2_RightInverseHatL}
    \Ghat (g) = M \phi + \Ghat_0 (g).
\end{equation} 
\begin{corollary}\label{cor:ext2}
There exists a constant $C$ such that for $g \in  \mathcal{Y}^{\mathrm{fl}}$,
\[ |M \phi|_\infty \le C \varepsilon, \qquad 
|\Ghat_0(g)|_\infty \le C |g|_{\infty}  \]  
\end{corollary}
\begin{proof}
    The proof of this result follows from Lemma~\ref{lem:ext2_right_inv}, Proposition~\ref{prop:ext2_RI1}, claim~\eqref{u0rho2flow} and the fact that, by Proposition~\ref{prop:ext1}, $|\Ggf (u,\theta)- \Gamma_0(u,\theta) | \lesssim \varepsilon$ for $u\in \tilde D^{\mathrm{out},\bu}_{2\kappa}$. Indeed, recalling that $u_0$ and $\rho_2$ are the topmost and leftmost points of the domain $D^{\mathrm{fl}}_{3 \kappa}$ we obtain that 
   \begin{equation*}
       |\phi|_\infty \lesssim \varepsilon \left(1 + \sum_{\ell >0}e^{-{\omega \over \varepsilon}\ell \left(\mathrm{Im}(u_0) - \mathrm{Im}(u)\right)} + \sum_{\ell <0}e^{-{\omega \over \varepsilon}\ell \left(\mathrm{Im}(\bar u_0) - \mathrm{Im}(u)\right)}\right) \lesssim \varepsilon.
   \end{equation*}
\end{proof}
\subsubsection{The fixed point equation} \label{sec:ext2_fixed_point}

We rewrite equation~\eqref{ext2:inv_eq} as the fixed point equation
\begin{equation}\label{ext2:inv_eq_FixedPoint}
  \Gamma   = \Ghat \circ \Ffl( \Gamma )
\end{equation}
with $\Ghat$ defined by~\eqref{def:ext2_RightInverseHatL}

\begin{lemma}\label{lem:ext2_first_iter}
    There exist $\kappa_0\ge 1$ and $0 <\varepsilon_0<1$ and a constant $c_0>0$ such that for any $\kappa \ge \kappa_0$ and $0 < \varepsilon \le \varepsilon_0$
    \begin{equation*}
        \left|\Ghat  \circ \Ffl(0)\right|_{\infty} \le c_0\varepsilon.
    \end{equation*}
\end{lemma}
\begin{proof}
    We note that, by Corollary~\ref{cor:ext2} it suffices to check that $|\Ffl(0)|_\infty \lesssim \varepsilon$. By definition~\eqref{def:Lhat:Fhat} of $\Ffl$ and using definition~\eqref{def:H0H1} for $H_1$ in Proposition~\ref{prop:main_rescaling}, 
    \[  
        \Ffl(0)(u, \theta) = X_{H_1} \circ \Gamma_0(u, \theta) =\begin{pmatrix} - \varepsilon \beta r^2R\sin \theta + \partial_{y_1} \hat H_1\circ \Gamma_0 \\ \varepsilon \beta r^2R\cos \theta + \partial_{y_2} \hat H_1\circ \Gamma_0\\
       -\varepsilon \beta r^3 \sin \theta - \partial_{x_1}\hat H_1 \circ \Gamma_0 \\\varepsilon \beta r^3 \cos \theta - \partial_{x_2}\hat H_1 \circ \Gamma_0\end{pmatrix}      
    \]
    where $\widehat{H}_1(x,y;\varepsilon) = \varepsilon^{-4} \tilde{H}_1 (\varepsilon^2 x, \varepsilon y; \varepsilon)$ and $\tilde{H}_1 (x,y;\varepsilon)=\mathcal{O}_6(x,y)$ uniformly in $\varepsilon$. 
  Therefore, a straightforward computation proves that $\left|\Ffl(0)\right|_\infty \lesssim\varepsilon$. 
\end{proof}

\begin{lemma}\label{lem:ext2_contract} Given $c_1>0$, there exist $\kappa_0\ge 1$ and $0 <\varepsilon_0<1$ and a constant $c_2>0$ such that for any $\kappa \ge \kappa_0$ and $0 < \varepsilon \le \varepsilon_0$, the functional $\Ffl : B_{c_1\varepsilon} \subset  \mathcal{Y}^{\mathrm{fl}} \to   \mathcal{Y}^{\mathrm{fl}} $ is Lipschitz\footnote{Notice that, here $g\in B_{c_1\varepsilon} \subset \mathcal{Y}^{\mathrm{fl}}$ is a vector-value function of $4$ components (see~\eqref{def:ext2_BS}).} with $\mathrm{lip}\, \Ffl \le c_2 \varepsilon$.
\end{lemma}

\begin{proof}
Let $\Gamma^1, \Gamma^2 \in B_{c_1\varepsilon}$. 
    Thanks to several Taylor extensions, we obtain that 
    \begin{align*}
         \left|\Ffl\left(\Gamma^{1}\right) - \Ffl\left(\Gamma^{2}\right)\right|_\infty \lesssim &\int_0^1 \int_0^1\left|D^2X_{H_0} \circ \left(\Gamma_0 + s\left(\Gamma^{2} +\tau\left(\Gamma^{1} - \Gamma^{2}\right)\right)\right)\right|_\infty ds \\
         &\times \left|\Gamma^{2} +\tau\left(\Gamma^{1} - \Gamma^{2}\right)\right|_\infty d\tau \left|\Gamma^{1} - \Gamma^{2}\right|_\infty\\
         &+ \int_0^1 \left|DX_{H_1} \circ \left(\Gamma_0 +\Gamma^{2} + \tau \left(\Gamma^{1} - \Gamma^{2}\right)\right)\right|_\infty d\tau\left|\Gamma^{1} - \Gamma^{2}\right|_\infty,
    \end{align*}
    where $\times$ stands for the usual multiplication and $D^2X_{H_0}$ is the second covariant derivative of the Hamiltonian vector field $X_{H_0}$. We observe that for all $\Gamma^{i}_1 \in B_{c_1\varepsilon} \subset \mathcal{Y}^{\mathrm{fl}}$ with $i=1,2$, we have that 
    \begin{equation}\label{proof:ext2_contr1}
        \left|\Gamma^{2} +\tau\left(\Gamma^{1} - \Gamma^{2}\right)\right|_\infty \le c_1\varepsilon
    \end{equation}
    and thanks to Proposition \ref{prop:main_rescaling}, one easily verifies that
    \begin{equation}
    \begin{aligned}\label{proof:ext2_contr2}
        &\left|D^2X_{H_0} \circ \left(\Gamma_0 + s\left(\Gamma^{2} +\tau\left(\Gamma^{1} - \Gamma^{2}\right)\right)\right)\right|_\infty \lesssim 1\\
        &\left|DX_{H_1} \circ \left(\Gamma_0 +\Gamma^{2} + \tau \left(\Gamma^{1} - \Gamma^{2}\right)\right)\right|_\infty \lesssim\varepsilon
        \end{aligned}
    \end{equation}
    for all $\tau$, $s \in [0,1]$. 
    Using the estimates~\eqref{proof:ext2_contr1} and~\eqref{proof:ext2_contr2}, the proof of the Lemma is finished.
\end{proof}

Let $c_0$ be the constant introduced in Lemma \ref{lem:ext2_first_iter}. The following proposition summarizes the existence and properties of the solution of equation~\eqref{ext2:inv_eq_FixedPoint}.
\begin{proposition}
    There exists $\kappa_0\ge 1$, $0 <\varepsilon_0<1$ and a constant $c_3>0$ such that, for any $\kappa \ge \kappa_0$ and $0 < \varepsilon \le\varepsilon_0$, the functional $\Ghat \circ \Ffl : B_{2c_0\varepsilon}\subset \mathcal{Y}^{\mathrm{fl}}  \to B_{2c_0\varepsilon}$ is a contraction with $\mathrm{lip}\, \Ghat \circ\Ffl \le c_3 \varepsilon$. Hence, equation~\eqref{ext2:inv_eq_FixedPoint} admits a solution such that
    \begin{equation*}
        |\Gamma|_\infty \le 2 c_0 \varepsilon.
    \end{equation*}
\end{proposition}
\begin{proof}
    It is an immediate consequence of Corollary \ref{cor:ext2} and Lemmas \ref{lem:ext2_first_iter} and \ref{lem:ext2_contract}.
\end{proof}

\subsection{From flow to graph parametrization}\label{sec:ext3:proof}
In this section we prove Proposition~\ref{prop:ext3}, namely we look for $h : \tilde D_{4 \kappa, d_5} \times \T_{\sigma \over 2} \to \C^2$ satisfying~\eqref{eq:ext3_est_g}, in such a way that 
\begin{equation}\label{eq:ext3_invEq1}
    \pi_y\left(\Ggf\circ \left(\mathrm{id} + h\right)\right)(u, \theta) = \gamma_0(u, \theta), \qquad 
    (u, \theta) \in \tilde D_{4 \kappa, d_5} \times \T_{\sigma \over 2},
\end{equation}
(see~\eqref{cond:ext:3}) where we recall that $\pi_y$ stands for the projection on the $y$ components and $\Ggf$ is the analytic extension of the parameterization of the unstable manifold provided by Proposition~\ref{prop:ext2}. 

First, we need to introduce and recall some notations. We denote by $\Ggf_1  = \Ggf - \Gamma_0$, by $\gamma_0 = \pi_y\Gamma_0$ (see Lemma~\ref{lem:homoclinic}) and by $\Ggf_{1,y} = \pi_y\Ggf_1$. We recall that $\diagonal_{\alpha, \beta}$ stands for the diagonal $2\times2$ matrix having $\alpha$ and $\beta$ as entries and $R_\theta$ is defined by~\eqref{def:R_theta}. For a given symmetric bilinear form $A$, $A\cdot (v,v)$ stands for the vector $v$ given twice as an argument to $A$.

We claim that $h$ is a solution of~\eqref{eq:ext3_invEq1} if and only if
\begin{equation}\label{eq:ext3_FixedPointEq}
   h = \mathcal{F}^{\mathrm{gr}}(h)
\end{equation}
with
\begin{equation*}
   \mathcal{F}^{\mathrm{gr}}(h) = -\diagonal_{{1 \over \dot r},{1 \over r}}R_\theta \left(\int_0^1(1 -\tau) D^2 \gamma_{0} \circ \left(\mathrm{id} + \tau h\right)d\tau \cdot (h, h) + \Ggf_{1,y}\circ \left(\mathrm{id}+h\right)\right).
\end{equation*}
In order to prove the claim, we observe that 
\begin{align*}
     \pi_y\left(\Ggf\circ \left(\mathrm{id} + h\right)\right) &= \gamma_{0}\circ \left(\mathrm{id} + h\right) + \Ggf_{1,y}  \circ \left(\mathrm{id} + h\right)\\
     \gamma_{0}\circ \left(\mathrm{id} + h\right) &=  \gamma_{0} + D \gamma_{0}h + \int_0^1(1 -\tau)D^2 \gamma_{0} \circ \left(\mathrm{id} + \tau h\right)d\tau \cdot (h, h)
\end{align*}
where the latter is obtained after a Taylor extension. Hence, we rewrite~\eqref{eq:ext3_invEq1} as
\begin{equation*}
     D \gamma_{0} \,h =  - \int_0^1(1 -\tau)D^2 \gamma_{0} \circ \left(\mathrm{id} + \tau h\right)d\tau \cdot (h, h) - \Ggf_{1,y}\circ \left(\mathrm{id} + h\right).
\end{equation*}
    A straightforward computation shows that $D \gamma_{0,y} = R_{-\theta}\diagonal_{\dot r, r}$. Thus, by multiplying both sides of the latter  by $\diagonal_{{1 \over \dot r},{1 \over r}}R_\theta$ we obtain~\eqref{eq:ext3_FixedPointEq}.
%

\subsubsection{The fixed point equation}
The aim of this section is to find a solution to equation~\eqref{eq:ext3_FixedPointEq}. To this end, we define the following Banach space
\begin{align}\label{def:ext3_BS}
    \mathcal{Y}^{\mathrm{gr2}} &= \Big\{f : \tilde D_{4\kappa, d_5} \times \T_{\sigma \over 2} \to \C : \mbox{$f$ is analytic and} \nonumber\\
    &\hspace{44mm} |f|_\infty = \sup_{(u, \theta) \in \tilde D_{4 \kappa , d_5} \times \T_{\sigma \over 2}} |f(u, \theta)| < \infty\Big\}.
\end{align}
We will use the same notation for vector-valued functions and matrices. As usual, the proof is divided into two steps, stated in Lemmas~\ref{lem:ext3_first_iter} and~\ref{lem:ext3_contract}. 

\begin{lemma}\label{lem:ext3_first_iter}
    There exist $\kappa_0\ge 1$ and $0<\varepsilon_0<1$ and a constant $c_0>0$ such that for any $\kappa \ge \kappa_0$ and $0 < \varepsilon \le \varepsilon_0$
    \begin{equation*}
        \left|\mathcal{F}^{\mathrm{gr}}(0)\right|_\infty \le c_0\varepsilon.
    \end{equation*}
\end{lemma}
\begin{proof}
    We observe that 
    $\mathcal{F}^{\mathrm{gr}}(0) = - \diagonal_{{1 \over \dot r},{1 \over r}}R_\theta \Ggf_{1, y}$. 
    On the one hand,  
    \begin{equation}\label{eq:ext3_StimeF01}
        \left|\diagonal_{{1 \over \dot r},{1 \over r}}R_\theta\right|_\infty \lesssim 1.
    \end{equation}
    On the other hand, since $\tilde D_{4\kappa , d_5} \subset D^{\mathrm{fl}}_{2\kappa}$, it follows from~\eqref{pro:ext2_TildeGamma1Est} in Proposition~\ref{prop:ext2} that
    \begin{equation}\label{eq:ext3_StimeF02}
        \left|\Ggf_{1, y} \right|_\infty\lesssim\varepsilon.
    \end{equation}
    Using~\eqref{eq:ext3_StimeF01} and~\eqref{eq:ext3_StimeF02}, one can conclude the proof of this lemma. 
\end{proof}
 
\begin{lemma}\label{lem:ext3_contract}
Given $c_1>0$, there exist $\kappa_0\ge 1$, $0 < \varepsilon_0<1$ and a constant $c_2>0$ such that for any $\kappa \ge \kappa_0$ and $0 < \varepsilon \le \varepsilon_0$ the functional $\mathcal{F}^{\mathrm{gr}} : B_{c_1\varepsilon} \subset \mathcal{Y}^{\mathrm{gr2}}\times \mathcal{Y}^{\mathrm{gr2}} \to  \mathcal{Y}^{\mathrm{gr2}}\times \mathcal{Y}^{\mathrm{gr2}}$ is Lipschitz with
$\mathrm{lip}\, \mathcal{F}^{\mathrm{gr}} \le c_2 / \kappa$.
\end{lemma}
\begin{proof}
    For any $h^1,h^2 \in B_{c_1\varepsilon} \subset \mathcal{Y}^{\mathrm{gr2}}\times \mathcal{Y}^{\mathrm{gr2}}$, we denote by $h_\tau = h^2 + \tau \left(h^1 - h^2\right)$. We have that 
    \begin{align}\label{eq:ext3_Contr_DiffF}
        &\mathcal{F}^{\mathrm{gr}}\left(h^1\right) - \mathcal{F}^{\mathrm{gr}}\left(h^2\right) = - \diagonal_{{1 \over \dot r},{1 \over r}}R_\theta \Bigg(\int_0^1 \left(1 -\tau \right) D^2 \gamma_{0}\circ \left(\mathrm{id} + \tau h^1\right)d\tau \cdot \left(h^1, h^1-h^2\right)\nonumber\\
        &\hspace{10mm}+ \int_0^1 \left(1 -\tau \right) D^2 \gamma_{0}\circ \left(\mathrm{id} + \tau h^1\right)d\tau \cdot \left(h^1-h^2, h^2\right)\nonumber \\
        &\hspace{10mm}+ \int_0^1 \left(1 -\tau \right)\int_0^1 D^3 \gamma_{0}\circ \left(\mathrm{id} + \tau h^2 + s\tau \left(h^1-h^2\right)\right)ds \left(h^1-h^2\right)\tau d\tau\cdot \left(h^2, h^2\right)\nonumber \\
        &\hspace{10mm}+ \int_0^1 D \Ggf_{1,y} \circ \left(\mathrm{id} +  h^2 + \tau \left(h^1 - h^2\right)\right)d\tau \left(h^1-h^2\right)\Bigg).
    \end{align}
    By Lemma~\ref{lem:homoclinic}, $\gamma_{0}$ is analytic on a neighborhood of $\tilde D_{4\kappa, d_5} \times \T_{\sigma \over 2}$. This implies that 
    \begin{equation}\label{eq:ext3_StimaGamma0}
        \max_{i=1,2,3}\left|D^i\gamma_{0}\right|_\infty \lesssim 1.
    \end{equation}
    In addition, using the same arguments as the ones in the proof of Lemma \ref{lem:ext1_contract}, Cauchy estimates and~\eqref{eq:ext3_StimeF02}, we prove that
    \[
    \left |D \Ggf_{1,y}\right |_{\infty} \lesssim {1 \over \varepsilon \kappa}  \left |\Ggf_{1,y}\right |_{\infty}\lesssim \frac{1}{\kappa}.
    \]
    Thanks to the latter,~\eqref{eq:ext3_StimeF01} and~\eqref{eq:ext3_Contr_DiffF}, we obtain that 
    \begin{align*}
        \left|\mathcal{F}^{\mathrm{gr}}\left(h^1\right) - \mathcal{F}^{\mathrm{gr}}\left(h^2\right)\right|_\infty &\lesssim  \left|\diagonal_{{1 \over \dot r},{1 \over r}}R_\theta\right|_\infty \Bigg( \left|D^2 \gamma_{0}\right|_\infty \left|h^1\right|_\infty + \left|D^2 \gamma_{0}\right|_\infty \left|h^2\right|_\infty\\
        &+  \left|D^3 \gamma_{0}\right|_\infty \left(\left|h^2\right|_\infty\right)^2 + \left|D\Ggf_{1,y}\right|_\infty \Bigg)\left|h^1- h^2\right|_\infty\\
        &\lesssim {1 \over \kappa} \left|h^1- h^2\right|_\infty.
    \end{align*} 
  This concludes the proof of this lemma.
\end{proof}

Let $c_0$ and $c_2$ be the constants introduced in Lemmas \ref{lem:ext3_first_iter} and \ref{lem:ext3_contract}, respectively. The proposition below ensures the existence of the solution of equation~\eqref{ext2:inv_eq_FixedPoint} and describes its properties. 
\begin{proposition}
    There exist $\kappa_0\ge 1$ and $0 <\varepsilon_0<1$ such that, for any $\kappa \ge \kappa_0$ and $0<\varepsilon\le\varepsilon_0$, the functional $\mathcal{F}^{\mathrm{gr}} : B_{2c_0\varepsilon} \subset \mathcal{Y}^{\mathrm{gr2}}\times \mathcal{Y}^{\mathrm{gr2}} \to  B_{2c_0\varepsilon}$ is a contraction with Lipschitz constant $\mathrm{lip}\, \mathcal{F}^{\mathrm{gr}} \le c_2 / \kappa$. Hence, equation~\eqref{eq:ext3_FixedPointEq} admits a solution such that 
    \begin{equation*}
        |h|_\infty \le 2 c_0 \varepsilon.
    \end{equation*}
\end{proposition}
\begin{proof}
    The proof follows immediately from Lemmas \ref{lem:ext3_first_iter} and \ref{lem:ext3_contract}.
\end{proof}


\subsection{Refinement of the analytic extension of $T^{\bu}$}\label{sec:ext4:proof}

In this section, we conclude the proof of Theorem~\ref{thm:ext} by proving the estimates~\eqref{est:thm_ext}. 
We have already proven, see~\eqref{firstbound:T1:ext}, that 
\[
\sup_{(u,\theta) \in \big (D^{\mathrm{ext}}_{\kappa} \backslash D^{\mathrm{out},\bu}_{\kappa, \varsigma} \big ) \times \T_{\sigma \over 2}} |\nabla T_1^\bu (u,\theta)| \leq c \varepsilon.
\]
In addition, $T_0+T_1^\bu$ is a solution of~\eqref{HJpolarcoordinates} and therefore, as proved in Section~\ref{sec:hamiltonJacobi}, $T_1$ is a solution of
$\Lout T_1 = \Fout (T_1)$ where $\Lout T = \partial_u T + \frac{\omega}{\varepsilon} \partial_\theta T$ is defined in~\eqref{def:operator_L} and $\Fout$ (see~\eqref{def:operator_F}) is given by
\begin{equation}
    \label{def:operator_F:ext}
    \begin{aligned}
   \Fout( T)  = &\frac{1}{2R^2}(\partial_u T)^2 +  \frac{1}{2r^2}(\partial_\theta T)^2 + \alpha \varepsilon^2   (\partial_\theta T)^2 - \beta \varepsilon r^2 \partial_\theta T \\ &+ \varepsilon^{-4} \tilde H_1 (\varepsilon^2( \delta_0 +(D\, \gamma_0)^{-\top}\nabla T), \varepsilon \gamma_0; \varepsilon).
\end{aligned}
\end{equation}
We take advantage of the fact that $T_1^{\bu}$ is a known function, so that it is a solution of the non-homogeneous partial linear equation 
\[
\Lout T =  \Fout (T_1^\bu)
\]
where we are using $T$ as a variable.
We observe that $g(u,\theta) = \Fout (T_1^{\bu})(u,\theta ) = \sum_{\ell\in \mathbb{Z}} g^{[\ell]} (u) e^{i\ell \theta}$ with 
\begin{equation}\label{boundsFout:ext}
|g^{[\ell]}(u)|\lesssim \varepsilon^2,  \qquad \mbox{for all $u \in D^{\mathrm{ext}}_{\kappa} \backslash D^{\mathrm{out},\bu}_{\kappa,\varsigma}$}. 
\end{equation}
In the latter, we used that $\tilde{H}_1(x,y;\varepsilon)$ is of order $\mathcal{O}_6(x,y)$ uniformly in $\varepsilon$.

Let $u_1 \in D^{\mathrm{ext}}_{\kappa} \backslash D^{\mathrm{out},\bu}_{\kappa,\varsigma} $ (see Figure~\eqref{fig:ext_domain}) be the topmost point. By Theorem~\ref{thm:outer} (notice that $u_1$ is away from $\pm i\frac{\pi}{2}$), 
\begin{equation}\label{initialcondition:ext:ref}
|T_1^{\bu}(u_1)| \lesssim \varepsilon^2, \qquad |\partial_u T_1^{\bu}(u_1)| \lesssim \varepsilon^2, \qquad |\partial_\theta T_1^{\bu}(u_1)| \lesssim \varepsilon^3.
\end{equation}
This implies that $T_1^{\bu}(u, \theta) = \sum_{k\in \mathbb{Z}} (T_1^{\bu})^{[\ell]} (u) e^{i\ell \theta}$ with 
\begin{align*}
    (T_1^{\bu})^{[\ell]} (u) & = e^{-\frac{\omega}{\varepsilon} i\ell (u-u_1)} (T_1^{\bu})^{[\ell]} (u_1) + 
    \int_{u_1}^u e^{\frac{\omega}{\varepsilon} i\ell (s-u)} g^{[\ell]}(s)ds \quad \mbox{if $\ell >0$} \\
    (T_1^{\bu})^{[0]} (u) & = (T_1^{\bu})^{[0]}(u_1) + \int_{u_1}^u g^{[0]}(s) ds \\ 
    (T_1^{\bu})^{[\ell]} (u) & = e^{-\frac{\omega}{\varepsilon} i\ell (u-\bar{u}_1)} (T_1^{\bu})^{[\ell]} (\bar{u}_1) + 
    \int_{\bar{u}_1}^u e^{\frac{\omega}{\varepsilon} i\ell (s-u)} g^{[\ell]}(s)ds \quad \mbox{if $\ell <0$}.
\end{align*}
for all $(u,\theta) \in \big (D^{\mathrm{ext}}_{\kappa} \backslash D^{\mathrm{out},\bu}_{\kappa,\varsigma} \big ) \times \T_{\sigma \over 2}$. Hence, using bounds~\eqref{boundsFout:ext} and~\eqref{initialcondition:ext:ref}, it is straightforward to check that 
$|T_1^{\bu}(u,\theta)| \lesssim \varepsilon^2$ for all $(u,\theta) \in \big (D^{\mathrm{ext}}_{\kappa} \backslash D^{\mathrm{out},\bu}_{\kappa,\varsigma} \big ) \times \T_{\sigma \over 2}$. 

Finally, we note that $\partial_u T_1^\bu$ is a solution of  $\Lout T =  \partial_u \Fout (T_1^\bu)$. Therefore
\begin{align*}
    (\partial_u T_1^{\bu})^{[\ell]} (u) & = e^{-\frac{\omega}{\varepsilon} i\ell (u-u_1)} (\partial_{u} T_1^{\bu})^{[\ell]} (u_1) + 
    \int_{u_1}^u e^{\frac{\omega}{\varepsilon} i\ell (s-u)} \partial_s g^{[\ell]}(s)ds \quad \mbox{if $\ell >0$} \\
    (\partial_u T_1^{\bu})^{[0]} (u) & = (\partial_u T_1^{\bu})^{[0]}(u_1) + \int_{u_1}^u \partial_s g^{[0]}(s) ds \\ 
    (\partial_u T_1^{\bu})^{[\ell]} (u) & = e^{-\frac{\omega}{\varepsilon} i\ell (u-\bar{u}_1)} (\partial_u T_1^{\bu})^{[\ell]} (\bar{u}_1) + 
    \int_{\bar{u}_1}^u e^{\frac{\omega}{\varepsilon} i\ell (s-u)} \partial_s g^{[\ell]}(s)ds \quad \mbox{if $\ell <0$}
\end{align*}
for all $(u,\theta) \in \big (D^{\mathrm{ext}}_{\kappa} \backslash D^{\mathrm{out},\bu}_{\kappa,\varsigma} \big ) \times \T_{\sigma \over 2}$.
Using that $ \ell\, \Im (s-u)<0$, doing parts in the integrals defining $(\partial_u T_1^{\bu})^{[\ell]}$ and using again bounds~\eqref{boundsFout:ext} and~\eqref{initialcondition:ext:ref} it is elementary to prove that $|(\partial_u T_1^{\bu})^{[\ell]} (u)| \lesssim \varepsilon^2$.
As a consequence 
\[
|\partial_u T_1^\bu(u,\theta) |\lesssim \varepsilon^2, \qquad |\partial_\theta T_1^\bu(u,\theta) |\lesssim \varepsilon^3, 
\]
for all $(u,\theta) \in \big (D^{\mathrm{ext}}_{\kappa} \backslash D^{\mathrm{out},\bu}_{\kappa,\varsigma} \big ) \times \T_{\sigma \over 2}$. This concludes the proof of Theorem~\ref{thm:ext}.

\section{The inner scale}\label{sec:inner:proof}

In this section, we look for two solutions $K_0^*$, with $*=\bu, \bs$, of the following inner equation 
\begin{equation}\label{InnerEq_proof}  
\Lin K = \Fin(K),
\end{equation}
where $\Lin$ and $\Fin$ are defined by~\eqref{defLin} and~\eqref{def:Fin} respectively. We find these solutions by rewriting the latter as a fixed-point equation by inverting the linear operator $\Lin$ on suitable Banach spaces. Afterwards, we analyze the difference $K^{\bu}_0-K^{\bs}_0$.

For the rest of this section, we fix $\sigma>0$ and $\vartheta_0 \in \left (0,\frac{\pi}{2}\right )$. We used these parameters in the definitions of the domains $\T_\sigma$, $D^{\mathrm{in}, \bu}_{\kappa}$ and $D^{\mathrm{in}, \bs}_{\kappa}$ (see~\eqref{def:complex_torus} and~\eqref{DomainInner}).

We recall that $B_\varrho$ is the ball of radius $\varrho$ centered at the origin, that $\kappa_0$ is a threshold parameter that will be assumed sufficiently large and that $u \lesssim v$ means that$u \le C v$ for some $C>0$.

The rest of this section is divided into four parts. First, in Section \ref{proof:Thm_inner_BS}, we define some Banach spaces we used in the analysis of the inner equation~\eqref{InnerEq_proof}. Secondly, in Section \ref{sec:in:Banach}, we find a suitable right-inverse $\G$ of the linear operator $\Lin$ and we analyze its properties. Thirdly, in Section \ref{sec:in:contraction}, we write the equation~\eqref{InnerEq_proof} as a fixed point equation and prove the existence of solutions. Finally, in Section~\ref{sec:in:difference}, we deal with the analysis of the difference $K^{\bu}_0-K^{\bs}_0$.

\subsection{Banach spaces}\label{proof:Thm_inner_BS}
We introduce the Banach spaces we will use in order to find a solution to equation~\eqref{InnerEq}.  To this end, let $\kappa>0$ and $\phi : D^{\mathrm{in}, *}_{  \kappa} \times \T_\sigma \to \C$,  $*=\bu, \bs$, be an analytic function with Fourier series 
\[ 
\phi(z,\theta)=
\sum_{\ell \in\Z} \phi^{[\ell]}(z) e^{i\ell \theta}.
\]
Given $n \ge 0$,  we define the following norms
\begin{eqnarray*}
|\phi^{[\ell]}|^*_{n} &=& \sup_{u \in D^{\mathrm{in}, *}_{  \kappa}} \left| \phi^{[\ell]}(z) z^n\right|, \qquad |\phi|^*_{n, \sigma} = \sum_{\ell \in \Z} |\phi^{[\ell]}|^*_{n} e^{|\ell|\sigma},  \\
\lfloor \phi \rfloor_{n, \sigma}^* &=& |\phi|^*_{n, \sigma} + |\partial_z \phi|^*_{n+1, \sigma} + |\partial_\theta \phi|^*_{n+1, \sigma}
\end{eqnarray*}  
for $*=\bu, \bs$. 
Finally, we introduce the Banach spaces
\begin{align*}
\mathcal{X}_n^* =& \left\{ \phi : D^{\mathrm{in}, *}_{ \kappa}   \to \C : \mbox{$\phi$ is analytic and $|\phi|^*_n <\infty$ }\right\}\\
\mathcal{X}_{n, \sigma}^* =& \left\{ \phi : D^{\mathrm{in}, *}_{  \kappa}  \times \T_\sigma \to \C : \mbox{$\phi$ is analytic and $|\phi|^*_{n, \sigma} <\infty$ }\right\}\\
\mathcal{\widetilde X}_{n, \sigma}^* =& \left\{ \phi : D^{\mathrm{in}, *}_{ \kappa}  \times \T_\sigma \to \C : \mbox{$\phi$ is analytic and $\lfloor\phi\rfloor^*_{n, \sigma} <\infty$ }\right\},
\end{align*}
for $*= \bu,\bs$. Similarly to Section \ref{sec:domains_and_Banach_spaces}, we will use the same notation for vector-valued functions and matrices. We say that a vector-valued function or a matrix belongs to the above spaces  if it is the case for each component.  The norms of a vector-valued function or a matrix are defined as the maximum of those norms of its components.

The following lemma, whose proof is in~\cite{Cast15}, contains some properties that the spaces $\mathcal{X}_{n, \sigma}^*$ and $\widetilde{\mathcal{X}}_{n,\sigma}^*$ satisfy.
\begin{lemma}
\label{PropNormInner} 
We fix $n_1,n_2 \in \N$ and $*=\bu, \bs$. Then, there exists $\kappa_0\ge 1$ and a constant $C>0$ such that for any $\kappa \ge \kappa_0$
\begin{enumerate}
\item If $n_1 \le n_2$, then $\mathcal{X}_{n_2, \sigma}^*\subset \mathcal{X}_{n_1, \sigma}^*$ and for all $\phi \in \mathcal{X}_{n_2, \sigma}^*$
\begin{equation*}
|\phi|_{n_1, \sigma}^* \le {C \over  \kappa^{n_2 -n_1}} |\phi|^*_{n_2, \sigma}.
\end{equation*}
\item If $\phi_1 \in \mathcal{X}^*_{n_1, \sigma}$ and $\phi_2 \in \mathcal{X}^*_{n_2, \sigma}$, then $\phi_1\phi_2 \in \mathcal{X}^*_{n_1+n_2, \sigma}$ and  
\begin{equation*}
|\phi_1 \phi_2|^*_{n_1 + n_2, \sigma} \le C |\phi_1|^*_{n_1, \sigma}|\phi_2|^*_{n_2, \sigma}.
\end{equation*}
\end{enumerate}
\end{lemma}
We will extensively use the properties in the previous lemma without explicitly referring to them.  
\subsection{Inverting the operator $\Lin$}
\label{sec:in:Banach}
In order to rewrite equation~\eqref{InnerEq_proof} as a fixed point equation, we need to find a right inverse of the operator $\Lin$ acting in the spaces introduced in Section \ref{proof:Thm_inner_BS}.

We point out that the linear operator $\Lin (K)=\partial_z K +  \omega\partial_\theta K$ has been extensively studied in the literature~\cite{B06,B12,Cast15}. Let us recall here the main features of $\Lin$ and, more concretely, of its right inverse. 



For a given function $g^*: D^{\mathrm{in}, *}_{ \kappa} \times \T_\sigma \to \C$, we are interested in the solutions $K^*_0   :D^{\mathrm{in},*}_{ \kappa}\times \T_\sigma \to \mathbb{C}$ of the linear equation
\begin{equation*}
\Lin(K^*_0) = g^{*}, \qquad \mbox{with $*=\bu,\bs$,} 
\end{equation*}
satisfying the asymptotic condition
\begin{equation*}
\lim_{z \to +\infty}K_0^\bs(z, \theta)=0 \quad \mbox{and}  \quad \lim_{z \to -\infty}K_0^\bu(z, \theta)=0.
\end{equation*}
It is clear that formal right inverses $\G^{\bu},\G^{\bs}$ of $\Lin$ exist and can be expressed as 
\begin{equation*}
    \G^{\bu} (g) (z,\theta) = \int^{0}_{+\infty} g(z+\tau,\theta+\omega \tau)d\tau, \qquad \G^{\bs} (g) (z,\theta) = \int^{0}_{-\infty} g(z+\tau,\theta+\omega \tau)d\tau
\end{equation*}
or equivalently,  
\begin{equation}\label{proof:inner:def_G}
\G^* (g) (z, \theta) = \sum_{\ell \in \Z} \G^{*,[\ell]} (g) (z) e^{i \ell \theta}, \qquad \mbox{with $*=\bu,\bs$} 
\end{equation}
with 
\begin{equation*}
\G^{\bs,[\ell]} (g) (u) =   \int^z_{+\infty} e^{i  \omega\ell(\tau - z)} g^{[\ell]}(\tau)d\tau, \quad \G^{\bu,[\ell]} (g) (u) =  \int^z_{-\infty} e^{i\omega \ell(\tau - z)} g^{[\ell]}(\tau)d\tau
\end{equation*}
where, for $*=\bu, \bs$, $g^{ [\ell]}$ are the Fourier coefficients associated with $g$.  The following lemma, proved in~\cite{Cast15} (see also~\cite{B06}) provides some properties satisfied by $\G^*$.
\begin{lemma}
\label{PropGInner}
Let $n \in \Z$ such that $n\ge 1$ and $*=\bu, \bs$. There exists $\kappa_0 \ge 1$ a constant $C>0$ such that for all 
$g \in \mathcal{X}_{n,\sigma}^*$, $\mathcal{G}^* (g) \in \widetilde{\mathcal{X}}_{n-1,\sigma}^*$ and $\kappa \ge \kappa_0$
\begin{equation*}
\lfloor \G^*(g) \rfloor^*_{n-1, \sigma} \le C |g|^*_{n, \sigma}.
\end{equation*}
Moreover, letting $g^{[\ell]}$ be the Fourier coefficients associated with $g$, if $g^{[0]}=0$, then $\mathcal{G}^*(g) \in \mathcal{X}_{n,\sigma}^*$ and 
\begin{equation*}
\left| \G^*(g)\right|^*_{n, \sigma} \le C  |g|^*_{n, \sigma}.
\end{equation*}
for all $\kappa \ge \kappa_0$.
\end{lemma}

\subsection{The solutions $K^{\bs,\bu}_0$ of the inner equation}
\label{sec:in:contraction}

With the operator $\mathcal{G}^*$ defined by~\eqref{proof:inner:def_G}, we can rewrite the inner equation~\eqref{InnerEq_proof} as the fixed point equation
\begin{equation}\label{eqIn_fixedpoint}
K= \G^* \circ \Fin (K)
\end{equation}
with $*=\bu, \bs$ and $\Fin$ defined in~\eqref{def:Fin}. We point out that if $K^*_0\in \widetilde{\mathcal{X}}^{*}_{5,\sigma}$, with $*=\bu, \bs$, are solutions of the above fixed point equation then $K^*_0$ are also solutions of the inner equation~\eqref{InnerEq_proof} and, in addition, they satisfy the asymptotic conditions
\[
\lim_{\Re z \to \infty} K_0^{\bs}(z,\theta)=0, \qquad \lim_{\Re z\to -\infty} K_0^{\bu}(z,\theta)=0.
\]
In this section, we verify that the right-hand side of equation~\eqref{eqIn_fixedpoint} is a contraction on a suitable closed ball of $\mathcal{X}^*_{5,\sigma}$. The proof is divided into several lemmas. 
\begin{lemma} \label{InnerFirstIter}
There exists $\kappa_0 \ge 1$ and a constant $c_0>0$ such that for any $\kappa \ge \kappa_0$
\[
\left |\G \circ \Fin(0) \right |_{5,\sigma}^* \leq c_0, \qquad 
\]
with $*=\bu,\bs$. 
\end{lemma}
\begin{proof}
By the definition of $\Fin$ (see~\eqref{def:Fin}), we have 
\[
\Fin(0)(z,\theta)=\tilde{H}_1\big ( (-iz^{-2} \cos \theta , -iz^{-2} \sin \theta, -iz^{-1} \cos \theta,-iz^{-1} \sin \theta ); 0 \big ).
\]
Using that $\tilde{H}_1(x,y;0) = \mathcal{O}_6(x,y)$, and Lemma \ref{PropGInner} the claim follows immediately. 
\end{proof}

\begin{lemma} \label{ExSolInner1}
Given $c_1>0$ and $*=\bu, \bs$, there exist $\kappa_0\ge 1$ and a constant $c_2>0$ such that, for any $\kappa \ge \kappa_0$, the functional $\Fin : B_{c_1} \subset \mathcal{\widetilde X}_{5, \sigma}^* \to \mathcal{\widetilde X}_{6, \sigma}^*$ is Lipschitz with $\mathrm{lip}\, \Fin \le c_2 / \kappa^2$.
\end{lemma}
\begin{proof}
Let $K_1,K_2 \in B_{c_1}$. We note that 
\begin{equation}\label{bound:inner:help}
|\partial_z K_1 (z,\theta)|,\, |\partial_z K_2 (z,\theta)|, \, |\partial_\theta K_1 (z,\theta)|, \,|\partial_\theta K_1 (z,\theta)| \leq \frac{c_1}{|z|^6}. 
\end{equation}
From~\eqref{def:Fin}, and denoting $\mathbf{v}_\theta = R_\theta e_1 = (\cos \theta, \sin \theta)^\top$ we can write 
\begin{equation}\label{proof:in_contr_est}
    \Fin(K_1) - \Fin (K_2) = \mathcal{I}_1 + \mathcal{I}_2 + \mathcal{I}_3
\end{equation}
with
\begin{align*}
  \mathcal{I}_1 &=-z^4 (\partial_z K_1 + \partial_z K_2) (\partial_z K_1- \partial_z K_2) \\
  \mathcal{I}_2 &=-  \left [ (z^2 - \alpha) (\partial_\theta K_1 + \partial_\theta  K_2) - \frac{\beta}{z^2} \right ] (\partial_\theta  K_1- \partial_\theta  K_2) \\ 
  \mathcal{I}_3 &=\int_{0}^1 \left [\partial_{x} \tilde{H}_1 \big (-iz^{-2} \mathbf{v}_\theta+  K_\lambda, -iz^{-1} \mathbf{v}_\theta ;0\big ) \right ] R_\theta D_{-iz^2, iz} (\nabla K_1 - \nabla K_2)^\top \, d\lambda
\end{align*}
where $\partial_z \tilde{H}_1$ is the differential with respect to $x=(x_1,x_2)$, 
\[
K_\lambda = R_\theta D_{-iz^2, iz} (\nabla K_2 + \lambda (\nabla K_1 -\nabla K_2))
\]
and we recall that (see~\eqref{def:in:Rtheta})
\[
R_\theta = \left (\begin{array}{cc} \cos \theta & -\sin \theta \\ \sin \theta & \cos \theta \end{array} \right ), \qquad D_{-iz^2, iz} = 
\left (\begin{array}{cc} -iz^2  & 0 \\ 0 & iz \end{array} \right ).
\]
We estimate each term on the right-hand side of~\eqref{proof:in_contr_est} separately. For this purpose, using~\eqref{bound:inner:help}, we deduce that 
\begin{equation}\label{boundI1:inner}
|z^6 \mathcal{I}_1(z,\theta ) | \lesssim \frac{c_1}{|z|^2} |\partial_z K_1 - \partial_z K_2|_{6,\sigma}^* \lesssim \frac{1}{\kappa^2} \lfloor K_1 - K_2 \rfloor_{5,\sigma}^* 
\end{equation}
for $\kappa_0$ large enough.
Concerning the term $\mathcal{I}_2$ in~\eqref{proof:in_contr_est}, we observe that  
\begin{equation}\label{boundI2:inner}
|z^6 \mathcal{I}_2(z,\theta ) |\lesssim \left (\frac{1}{|z|^4} + \frac{1}{|z|^6} + \frac{1}{|z|^2} \right ) |\partial_\theta K_1 - \partial_\theta K_2|_{6,\sigma}^* \lesssim \frac{1}{\kappa^2} \lfloor K_1 - K_2 \rfloor_{5,\sigma}^*
\end{equation}
for $\kappa_0$ large enough.
Finally, we deal with $\mathcal{I}_3$. For $\varrho>0$, we first introduce $\mathbb{B}_{\varrho}:= \{(x,y) \in \mathbb{C}^4\,:\, |(x,y)|\leq \varrho \}$. Here $|(x,y)| = \max\{|x_1|, |x_2|, |y_1|, |y_2|\}$. Let $\varrho_0>0$ be such that $\tilde{H}_1$ is analytic on 
$\mathbb{B}_{\varrho_0}$. Then, for all, 
\[
|\partial_x \tilde{H}_1 (x,y;0)| \lesssim |(x,y)|^5, \qquad \mbox{for all $(x,y) \in \mathbb{B}_{\varrho_0/2}$}.
\]
Indeed, taking $\kappa_0$ large enough, the claim is straightforward from Cauchy's estimates and the fact that, for $(x,y) \in \mathbb{B}_{\varrho_0/2}$, 
the open ball centered at $(x,y)$ of radius $|(x,y)|$ is contained in $\mathbb{B}_{\varrho_0}$. In addition, it is clear that 
\[
K_\lambda \in \mathcal{X}_{4,\sigma}^* \times \mathcal{X}_{4,\sigma}^*.
\]

Therefore, there exists $\kappa_0$ large enough such that 
\[
\big (-iz^{-2} + K_\lambda , -iz^{-1} \big )\in \mathbb{B}_{\varrho_0/2} 
\]
and hence
$| z^5 \partial_x \tilde{H}_1 (-iz^{-2}\mathbf{v}_\theta + K_\lambda, -iz^{-1} \mathbf{v}_\theta;0)| \lesssim 1$. Then
\[
|z^6 \mathcal{I}_3 (z,\theta)| \lesssim \frac{1}{|z|^3} \left ( | \partial_z K_1 - \partial_z K_2 |_{6,\sigma}^* + 
| \partial_\theta K_1 - \partial_\theta  K_2 |_{6,\sigma}^* \right ) \lesssim \frac{1}{\kappa^3} \lfloor K_1 - K_2 \rfloor_{5,\sigma}^*.
\]
Combining the above estimate with~\eqref{boundI1:inner} and~\eqref{boundI2:inner}, we conclude the proof of this lemma.
\end{proof}

Let $c_0$ be the constant introduced in Lemma \ref{InnerFirstIter}. The following proposition provides the existence and main properties of the solutions of the inner equation~\eqref{eqIn_fixedpoint}.
\begin{proposition}\label{proof:inner_sol}
    There exist $\kappa_0 \ge 1$ and a constant $c_3>0$ such that, for any $\kappa \ge \kappa_0$, the functional $\G \circ \Fin: B_{2c_0}\subset \mathcal{\widetilde X}_{5, \sigma}^* \to B_{2c_0}$ is a contraction with Lipschitz constant $\mathrm{lip} \G \circ \Fin \le c_3 / \kappa^2$. Hence, equation~\eqref{eqIn_fixedpoint} admits two solutions $K_0^* \in \mathcal{\widetilde X}_{5, \sigma}^*$, with $*=\bu, \bs$, such that 
    \begin{equation*}
        \lfloor K_0^*\rfloor^*_{5,\sigma} \le 2c_0.
    \end{equation*}
\end{proposition}
\begin{proof}
    The proof is an immediate consequence of Lemmas \ref{PropGInner}, \ref{InnerFirstIter}, and \ref{ExSolInner1}.
\end{proof}

\subsection{The difference of the inner solutions}  
\label{sec:in:difference}
Consider now 
\[
\Delin  = K^{\bu}_0 - K^{\bs}_0
\]
where $K^{*}_0 \in \widetilde{\mathcal{X}}_{5,\sigma}^*$, with $*=\bu, \bs$, are the solutions of the inner equation $\Lin K = \Fin (K)$ (see~~\eqref{defLin} and~\eqref{def:Fin}) provided by Proposition \ref{proof:inner_sol}. By construction, $\Delin$ is analytic on the domain $E_{\kappa}$ (see~\eqref{DomainDiffInner}), namely
\begin{equation}\label{DomainDiffInner:proof}
E_{\kappa} := D_{\kappa}^{\mathrm{in},\bu} \cap D_{\kappa}^{\mathrm{in},\bs} \cap \{w \in \mathbb{C}\,:\, \Im z <0\}
\end{equation}
(see also Figure~\ref{fig:inner_domain}) and it satisfies the homogeneous linear equation defined by
\begin{equation}\label{EqDiffIn}
(1+\a(z,\theta)) \partial_z \Delin +( \omega +\b(z,\theta)) \partial_\theta \Delin =0
\end{equation}
with, denoting $\mathbf{v}_\theta = R_\theta e_1=(\cos \theta,\sin \theta )^\top$, $\mathbf{w}_\theta =R_\theta e_2= (-\sin \theta, \cos \theta )^{\top}$,
\begin{equation}\label{defabdiffin}
\begin{aligned}
\a(z,\theta) = & -z^4 (\partial_zK^{\bu} + \partial_z K^{\bs}) \\ & -iz^2  \int_{0}^1 \big (\partial_x \tilde{H}_1 (-iz^2 \mathbf{v}_\theta + K_\lambda, -iz \mathbf{v}_\theta ; 0) \big )^{\top} \mathbf{v}_\theta \, d\lambda  
\\
\b(z,\theta) = & - (z^2 -\alpha) (\partial_\theta K^{\bu} + \partial_\theta K^{\bs}) \\ & +  \frac{\beta}{z^2} + 
iz \int_{0}^1 \big (\partial_x \tilde{H}_1 (-iz^2 \mathbf{v}_\theta + K_\lambda,  -iz \mathbf{v}_\theta; 0) \big )^{\top}\mathbf{w}_\theta \, d\lambda
\end{aligned}
\end{equation}
and $K_\lambda$ stands for the known function
\begin{equation}\label{defKlambdadiff}
K_\lambda= R_\theta D_{-iz^2, iz}\left(\nabla K^{\bs} + \lambda (\nabla K^{\bu} - \nabla K^{\bs})\right).
\end{equation}

Now, equation~\eqref{EqDiffIn} is our starting point.  In the second part of this section, we introduce the Banach spaces we use to study equation~\eqref{EqDiffIn} and provide estimates for the terms $\a$ and $\b$.

For $\phi : E_{\kappa} \times \T_\sigma \to \C$ analytic with Fourier series $ \sum_{\ell \in\Z} \phi^{[\ell]}(z) e^{i\ell \theta}$ we define the following norms
\begin{equation}
\label{NormsDiffInner}
\begin{aligned}
|\phi|_{n, \sigma} & = \sum_{\ell \in \Z} |\phi^{[\ell]}|_{n} e^{|\ell|\sigma}, \qquad 
|\phi^{[\ell]}|_{n} = \sup_{z \in E_{\kappa} } \left| \phi^{[\ell]}(z) z^n\right|,  \\
\lfloor \phi \rfloor_{n, \sigma} &= |\phi|_{n, \sigma} + |\partial_z \phi|_{n+1, \sigma} + |\partial_\theta \phi|_{n+1, \sigma}
\end{aligned}
\end{equation}
where $n \in \Z$ with $n \ge 0$, and $E_{\kappa}$ is defined by~\eqref{DomainDiffInner:proof}.  These norms satisfy the same properties stated in Lemma \ref{PropNormInner}; we will use them without mentioning them explicitly.  
Furthermore, we consider the following Banach spaces
\begin{eqnarray*}
\mathcal{Y}_{n, \sigma} &=& \left\{ \phi : E_{\kappa}  \times \T_\sigma \to \C : \mbox{$\phi$ is analytic and $|\phi|_{n, \sigma} <\infty$ }\right\}\\
\mathcal{\widetilde Y}_{n, \sigma} &=& \left\{ \phi : E_{\kappa} \times \T_\sigma \to \C : \mbox{$\phi$ is analytic and $\lfloor\phi\rfloor_{n, \sigma} <\infty$ }\right\} 
\end{eqnarray*}

\begin{lemma}\label{lem:abdif}
Let $\a,\b$ be as in~\eqref{defabdiffin} and $\kappa_0$ be such that Proposition~\ref{proof:inner_sol} (and hence Theorem~\ref{thm:inner}) holds. Then, there exists a constant $c_3>0$ such that for all $\kappa\geq \kappa_0$,
\[
|\a|_{2,\sigma} \leq c_3, \qquad |\b|_{2,\sigma} \leq c_3. 
\]
\end{lemma}
\begin{proof}
We observe that $K^{\bu}, K^{\bs} \in \widetilde{\mathcal{Y}}_{5,\sigma}$ and, as a consequence, $K_\lambda$ in~\eqref{defKlambdadiff} belongs to $\mathcal{Y}_{6,\sigma}$. It follows that 
\[
\partial_x \tilde{H}_1 (-iz^2 \mathbf{v}_\theta + K_\lambda, -iz\mathbf{v}_\theta; 0) \in \mathcal{Y}_{5,\sigma}.
\]
Consequently, we have $\a,\b \in \mathcal{Y}_{2,\sigma}$. 
\end{proof}

Notice that, as a trivial consequence of Lemma~\ref{lem:abdif}, there exists $\kappa_1\ge \kappa_0$ such that $\omega +\b \ge \omega /2$ if $(z,\theta) \in E_{\kappa}$ with $\kappa \geq \kappa_1$. Then, equation~\ref{EqDiffIn} can be rewritten as 
\[  
    (1+ \tilde{\a}) \partial_z \Delin + \omega \partial_\theta \Delin=0, \qquad  \tilde{\a} =  \omega \frac{1+\a}{ \omega +\b} -1 \in \mathcal{Y}_{2,\sigma}
\]
and rescaling $\bar{z} = \omega z$ we obtain that $\overline{\Delin} (\bar{z},\theta) = \Delin ( \omega^{-1} \bar{z},\theta)$ 
satisfies
\begin{equation}
    \label{eq:Diff:inn}
    (1+\bar{\a}) \partial_{\bar z} \overline{\Delin} + \partial_\theta \overline{\Delin}=0, \qquad \bar{\a}(\bar{z},\theta) = \tilde{\a} (\omega^{-1} \bar z, \theta).                            
\end{equation}
The solutions of equation~\eqref{eq:Diff:inn} belonging to $\mathcal{Y}_{n,\sigma}$ for some $n\geq 0$ were studied  in~\cite{B06}, Section 5. Following the results in that work one can prove that all the solutions of~\eqref{eq:Diff:inn} are of the form $\bar{\chi} (\bar{z}-\theta + \bar g(\bar{z},\omega))$ with 
\[
\bar g\in \mathcal{Y}_{1,\sigma}, \qquad \bar{\chi} (\xi ) = \sum_{k<0} \chi^{[k]} e^{ik\xi}.  
\]
Undoing the rescaling $z=\omega^{-1} \bar{z}$, we obtain formula~\eqref{formula:diff:in} in Theorem~\ref{thm:inner}.

\section{Matching the inner and the outer scale} \label{sec:proof:matching}

In this section, we prove that the solutions $K^{\bu}_0$ and $K^{\bs}_0$ of the inner equation~\eqref{InnerEq} from Theorem~\ref{thm:inner} provide approximations of the parameterizations of the invariant manifolds $T_1^{\bu}$ and $T_1^{\bs}$ in complex subdomains of the domains $D^{\mathrm{in},\bu}_{\kappa}$ and $D^{\mathrm{in},\bs}_{\kappa}$ where the solutions of the inner equation are defined. 

Throughout this section, we fix $\sigma>0$, and $0 < \vartheta_1 < \vartheta_0 < \vartheta_2 < {\pi \over 2}$. These parameters are those used in the definition of the domains $D^{\mathrm{mch}, \bu}_{\kappa}$ and $D^{\mathrm{mch}, \bs}_{\kappa}$ (see~\eqref{Dmhcu} and~\eqref{Dmhcs}). Furthermore, we choose them so that the inclusions~\eqref{inclusion_domains} hold.

We also use $\kappa_0$ and $\varepsilon_0$ as threshold values and the notation $u\lesssim v$ if there exists a positive constant $C$ such that $u \le Cv$.

We prove Theorem~\ref{thm:matching} only for the $-\bu-$ case, being the $-\bs-$ case analogous. 
We introduce the domain
\[\mathcal{D}^{\mathrm{mch},\bu}_{\kappa} = \left \{ z \in \mathbb{C}\,:\, u(z) =  \varepsilon z+ i\frac{\pi}{2}  \in D^{\mathrm{mch,\bu}}_\kappa\right \} 
\]
the matching domain in the inner variables (see~\eqref{uz}) and the corresponding to $u_1,u_2$ (see~\eqref{defu1u2}) points
\begin{equation}\label{defz1z2}
z_1 = \frac{1}{\varepsilon} \left ( u_1 - i\frac{\pi}{2}\right ), \qquad z_2 = \frac{1}{\varepsilon} \left ( u_2 - i\frac{\pi}{2}\right )
\end{equation}
and we recall that 
\begin{equation}\label{propz1z2}
\frac{1}{\varepsilon^{1-\gamma}}  \lesssim |z_1|, |z_2| \lesssim \frac{1}{\varepsilon^{1-\gamma}} , \qquad \Im z_2 \leq \Im z \leq \Im z_1, \quad z\in \mathcal{D}_\kappa^{\mathrm{mch}, \bu}.
\end{equation}

Consider now the difference 
\begin{equation}\label{def:matching} 
K_1^\bu := K^\bu- K_0^\bu,
\end{equation}
where $K^\bu (z,\theta) = \varepsilon^3 T_1^\bu \left (\varepsilon z + i \frac{\pi}{2} ,\theta \right )$ is defined in~\eqref{phieta12} and $K^\bu_0$ is the solution of the inner equation provided by Theorem~\ref{thm:inner}. 
The functions, $K^\bu, K^\bu_0$ are analytic functions on $\mathcal{D}^{\mathrm{mch},\bu}_\kappa$ and we observe that by Theorems~\ref{thm:outer} and~\ref{thm:inner}, $K^\bu$ and $K^\bu_0$ satisfy the equations
\begin{equation}\label{proof:mch_eq_for_K_K0}
\Lin K^\bu= \varepsilon^4 \Fout (\varepsilon^3 K^\bu), \qquad \Lin K_0^\bu =\Fin (K_0^\bu)
\end{equation}
(recall the definitions~\eqref{defFouter_innervariables} and~\eqref{def:Fin} of $\varepsilon^4 \Fout$ and $\Fin$). 

The strategy to prove Theorem~\ref{thm:matching} is to analyze the nonhomogeneous linear equation satisfied by $K_1^{\bu}$.  
We note that Theorems~\ref{thm:outer} and~\ref{thm:inner} already provide a preliminary (non-sharp) estimate for $K_1^{\bu}$.  
Here, we show that in a smaller domain $\mathcal{D}^{\mathrm{mch},\bu}_{\kappa}$ a sharper upper bound can be obtained.  

In Section~\ref{sec:Mhc_eq_K1u}, we derive the explicit form of the equation satisfied by $K_1^{\bu}$.  
Then, in Section~\ref{sec:Mhc_bound_K1u}, we introduce suitable Banach spaces and establish the improved estimate for $K_1^{\bu}$.

\subsection{The equation for $K_1^\bu$}\label{sec:Mhc_eq_K1u}
In order to write the nonhomogeneous linear equation that $K_1^\bu$ satisfies, we introduce 
\begin{equation}\label{defP:matching}
\begin{aligned}
\mathcal{P}^{\mathrm{mch}} (\phi,\psi) =& -(z^2 \phi)^2 - (z \psi)^2 + \alpha (\psi)^2 + \beta \frac{1}{z^2} \psi \\ &+ \tilde{H}_1 \big (-iz^{-2} R_\theta e_1 + R_\theta  D_{-iz^2 , iz} (\phi,\psi)^{\top}, -iz^{-1} R_{\theta} e_1; 0\big ) 
\end{aligned}
\end{equation}
and we notice that, using definition~\eqref{def:Fin} of $\Fin$, 
\[
\Fin (K) = \mathcal{P}^{\mathrm{mch}} (\partial_z K, \partial_\theta K).
\]
Combining~\eqref{def:matching} and~\eqref{proof:mch_eq_for_K_K0} with the above identity, one can see that $K_1^\bu$  is a solution of the nonhomogeneous linear equation
\begin{equation}\label{eq:matching}
\Lin K_1^\bu = \mathcal{A} \nabla K_1^\bu + \mathcal{B}
\end{equation}
with 
\begin{equation}\label{defABmatching}
    \begin{aligned}
        \mathcal{A}(z,\theta) = &  \int_{0}^1 D\mathcal{P}^{\mathrm{mch}} (\nabla K_0^\bu  + \lambda \nabla (K^\bu - K^\bu_0))  (z,\theta) \, d\lambda  \\
        \mathcal{B}(z,\theta) = &  \varepsilon^4 \Fout (\varepsilon^{-3} K^\bu)(z,\theta) - \Fin (K^\bu)(z,\theta).
    \end{aligned}
\end{equation}
Here $D\mathcal{P}^{\mathrm{mch}}$ denotes $(\partial_\phi\mathcal{P}^{\mathrm{mch}}, \partial_\psi\mathcal{P}^{\mathrm{mch}})$ and we are using that $K^\bu, K_0^\bu$ are known functions. 

We observe that, since $K_1^\bu$ is a solution of~\eqref{eq:matching} defined on $\mathcal{D}^{\mathrm{mch},\bu}_\kappa \times \mathbb{T}_\sigma$, then 
\[
K_1^\bu (z,\theta ) = \sum_{k\in \mathbb{Z}} K^{\bu,[k]}_{1} (z) e^{ik\theta}
\]
with 
\[
K^{\bu,[k]}_1 (z) =  e^{i\omega k (z_{[k]}-z)} K_1^{\bu,[k]}(z_{[k]}) + \int_{z_{[k]}} ^z e^{i\omega k (\tau-z) } (\mathcal{A} \nabla K_1^\bu + \mathcal{B} )^{[k]} (\tau) \, d\tau
\]
for any initial conditions (that can depend on $k$) $z_{[k]} \in \mathcal{D}^{\mathrm{mch},\bu}_\kappa$. 
We choose 
\begin{equation}\label{defzkmatching}
z_{[k]}= z_1, \quad k> 0, \qquad z_{[k]} = z_2 , \quad k\le 0
\end{equation}
with $z_1,z_2$ defined in~\eqref{defz1z2} (see also~\eqref{defu1u2}). We then consider the following right inverse $\G$ of the linear operator $\Lin$, given by
\[
\mathcal{G}(g) (z,\theta) = \sum_{k\in \mathbb{Z}} \mathcal{G}^{[k]} (g) (z) e^{ik\theta}, \qquad 
\mathcal{G}^{[k]} (g) (z) = \int_{z_{[k]}} ^z e^{i\omega k (\tau-z) } g^{[k]}(\tau) \, d\tau.
\]
In addition, we define the following initial condition
\begin{equation}\label{def:K0:matching}
K^0(z,\theta)= \sum_{k\in \mathbb{Z}} K^{\bu, [k]}_1(z_{[k]}) e^{ik(z_{[k]} -z)}e^{ik\theta}.
\end{equation}
Hence, we can rewrite equation~\eqref{eq:matching} as
\begin{equation}\label{eqK1matching_1}
K_1^{\bu} = K^0  + \mathcal{G} (\mathcal{B}) + \mathcal{G} \circ \mathcal{A} (\nabla K_1^{\bu})
\end{equation}
where $\mathcal{A},\mathcal{B}$ defined in~\eqref{defABmatching}. 

\subsection{A posteriori bound for $K_1^u$}\label{sec:Mhc_bound_K1u}

We introduce suitable Banach spaces to work with. For $\phi : \mathcal{D}^{\mathrm{mch}, \bu}_{\kappa} \times \T_\sigma \to \C$ analytic with Fourier series $ \sum_{\ell \in\Z} \phi^{[\ell]}(z) e^{i\ell \theta}$ we define the following norms
\begin{eqnarray*}
|\phi|_{n, \sigma}^{\bu} &=& \sum_{\ell \in \Z} |\phi^{[\ell]}|^\bu_{n} e^{|\ell|\sigma}, \qquad 
|\phi^{[\ell]}|^\bu_{n} = \sup_{z \in \mathcal{D}^{\mathrm{mch}, \bu}_{\kappa} } \left| \phi^{[\ell]}(z) z^n\right| ,\\
\lfloor \phi \rfloor^\bu_{n, \sigma} &=& |\phi|^\bu_{n, \sigma} + |\partial_z \phi|^\bu_{n+1, \sigma} + |\partial_\theta \phi|^\bu_{n+1, \sigma}
\end{eqnarray*}
where $n \in \Z$ with $n \ge 0$.  Moreover, we consider the following Banach spaces
\begin{eqnarray*}
\mathcal{Z}^\bu_{n, \sigma} &=& \left\{ \phi : \mathcal{D}^{\mathrm{mch}, \bu}_{\kappa}   \times \T_\sigma \to \C : \mbox{$\phi$ is analytic and $|\phi|^\bu_{n, \sigma} <\infty$}\right\}\\
\mathcal{\widetilde Z}^\bu_{n, \sigma} &=& \left\{ \phi : \mathcal{D}^{\mathrm{mch}, \bu}_{\kappa}  \times \T_\sigma \to \C : \mbox{$\phi$ is analytic and $\lfloor\phi\rfloor^\bu_{n, \sigma} <\infty$ }\right\}.
\end{eqnarray*}
Similarly to Section \ref{proof:Thm_inner_BS}, we will use the same notation for vector-valued functions and matrices. 

An important remark is that, by Theorems~\ref{thm:outer} and \ref{thm:inner}, together with the definition~\eqref{phieta12} of $K^\bu$, we obtain
\begin{equation} \label{Kumatching}
    \lfloor K^{\bu} \rfloor_{5,\sigma}^\bu, \,\lfloor K^{\bu}_0 \rfloor_{5,\sigma}^\bu, \, \lfloor K^{\bu}_1 \rfloor_{5,\sigma}^\bu\lesssim 1.
\end{equation}
 Moreover, for any $z \in \mathcal{D}^{\mathrm{mch},\bu}_\kappa$, one has
\[
\kappa \cos \vartheta_2 \lesssim |z| \lesssim \frac{1}{\varepsilon^{1-\gamma}}.
\]
This estimate will be used repeatedly in what follows without an explicit mention. 
As a consequence, we can state the following result.
\begin{lemma}\label{lem:banach:matching}
We fix $n_1,n_2 \in \mathbb{N}$. There exists $\kappa_0\geq 1$ and a constant $C>0$ such that for all $\kappa\geq \kappa_0$
\begin{enumerate}
    \item If $\phi \in \mathcal{Z}_{n_1,\sigma}^{\bu}$, then $\phi\in \mathcal{Z}_{n_2,\sigma}^\bu$, for all $n_2\in \mathbb{Z}$ and when $n_1<n_2$ 
    $$
    |\phi|_{n_1,\sigma} \leq \frac{C}{\kappa^{n_2-n_1}} | \phi |_{n_2,\sigma}, \qquad 
    |\phi|_{n_2,\sigma}\leq \frac{C}{\varepsilon^{(n_2-n_1)(1-\gamma)}} |\phi|_{n_1,\sigma}.
    $$
    \item If $\phi_1\in \mathcal{Z}_{n_1,\sigma}^\bu$ and $\phi_2\in \mathcal{Z}_{n_2,\sigma}^\bu$, then 
    $\|\phi_1 \phi_2\|_{n_1 + n_2,\sigma} \leq C \|\phi_1\|_{n_1}\|\phi_2\|_{n_2,\sigma}$.
\end{enumerate}
\end{lemma}

The following result deal with $\mathcal{G}$ and $\mathcal{G} \circ \mathcal{A}$. 

\begin{lemma}\label{lem:propGAmatching} 
  We fix $n \in \N$. There exist $\kappa_0 \ge 1$, $0<\varepsilon_0<1$ and a constant $C>0$ such that for any $\kappa \ge \kappa_0$ and $0<\varepsilon \le \varepsilon_0$
  \begin{enumerate}
      \item \label{item1}If $n>1$, for all $g\in \mathcal{Z}_{n,\sigma}^\bu$, 
      $\lfloor \mathcal{G}(g) \rfloor_{n-1,\sigma}^\bu \leq C |g|_{n,\sigma}^\bu$. 
      \item For all $g\in \mathcal{Z}_{n,\sigma}^\bu$ such that 
      $g^{[0]}=0$, $|\mathcal{G}(g)|_{n,\sigma}^\bu \leq C |g|_{n,\sigma}^\bu$.
      \item \label{item3} For all $g_1,g_2\in  {\mathcal{Z}}_{n,\sigma}^\bu$, 
      $|\mathcal{A} (g_1,g_2)^\top |_{n+1,\sigma}^\bu \leq C \kappa^{-1} (|g_1|_{n,\sigma}^\bu + |g_2|_{n,\sigma}^\bu)$.
  \end{enumerate}
As a consequence, the operator 
$
\widetilde{\mathcal{G}}(g) := \mathcal{G} \circ \mathcal{A}(\nabla g)
$
satisfies that 
$
\widetilde{\mathcal{G}} : \widetilde{\mathcal{Z}}_{n,\sigma}^\bu \to \widetilde{\mathcal{Z}}_{n,\sigma}^\bu$ and 
\[
\lfloor \widetilde{\mathcal{G}}(g) \rfloor_{n,\sigma}^{\bu} \lesssim \frac{C}{\kappa} |g|_{n,\sigma}^{\bu}, \qquad \big |\big (\mathrm{Id} - \widetilde{\mathcal{G}}\big )^{-1} (g) \big |_{n,\sigma}^{\bu} \lesssim C .
\]
\end{lemma}
\begin{proof}
    The first two items, those corresponding to $\mathcal{G}$, are straightforwardly deduced from Lemma~20 of~\cite{BCS13} and from the proof of Lemma~6.2 in~\cite{GGSZ25}. We now deal with the third one. Observe that by~\eqref{Kumatching}, for all $\lambda\in [0,1] $ 
    \begin{equation}\label{Kumatchinglambda}
    K_\lambda:= K_0^\bu + \lambda  (K^\bu-K^\bu_0) \in \widetilde{\mathcal{Z}}_{5,\sigma}^\bu \quad \mbox{and \hspace{2mm}}\lfloor K_\lambda\rfloor_{5,\sigma}^\bu \lesssim 1.
    \end{equation}
Furthermore, from the definition~\eqref{defP:matching} of $\mathcal{P}^{\mathrm{mch}}$, we have
\begin{align*}
\partial_\phi \mathcal{P}^{\mathrm{mch}} ( \nabla K_\lambda) =&-2z^4 \partial_z K_\lambda \\ & +\partial_x \tilde{H}_1 (-iz^{-2} R_\theta e_1 + R_\theta \diagonal_{-iz^2 , iz} \nabla K_\lambda , -iz^{-1} R_\theta e_1;0) R_\theta e_1 (-iz^2) \\ 
\partial_{\psi} \mathcal{P}^{\mathrm{mch}} (\nabla K_\lambda) =& 2(-z^2 + \alpha) \partial_\theta K  + \beta \frac{1}{z^2}   \\&+\partial_x \tilde{H}_1 (-iz^{-2} R_\theta e_1 + R_\theta \diagonal_{-iz^2 , iz} \nabla K_\lambda, iz^{-1} R_\theta e_1;0) R_\theta e_2 (iz )
\end{align*}
where we recall that $\diagonal_{-iz^2,iz}$ stands for the diagonal $2\times 2$ matrix with entries $-iz^2, iz$, $e_1=(1,0)^{\top}$ and $e_2= (0,1)^{\top}$. Since $\partial_x \widetilde{H}_1(x,y;0) = \mathcal{O}(|(x,y)|^5)$, it follows that 
\[
|\partial_\phi \mathcal{P}^{\mathrm{mch}} (\nabla K_\lambda) |_{2,\sigma}^{\bu} \lesssim 1, 
\qquad 
|\partial_\psi \mathcal{P}^{\mathrm{mch}} (\nabla K_\lambda) |_{2,\sigma}^{\bu}\lesssim 1.
\]
Consequently, by definition~\eqref{defABmatching}
\[
|\mathcal{A} (g_1,g_2)^\top |_{n+1,\sigma}^{\bu} \lesssim \frac{1}{\kappa} (|g_1|_{n,\sigma}^\bu + |g_2|_{n,\sigma}^\bu ).
\]
Let now $g\in \widetilde{Z}_{n,\sigma}^\bu$. Then $\nabla g \in Z_{n+1,\sigma}^\bu \times Z_{n+1, \sigma}^\bu$. 
By item~\ref{item3}, $|\mathcal{A} \nabla g |_{n+1,\sigma}^\bu \lesssim \kappa^{-1} (|\partial_z g|_{n+1,\sigma}^\bu+|\partial_\theta g|_{n+1,\sigma}^\bu) \leq \kappa^{-1} \lfloor g \rfloor_{n,\sigma}^\bu$. Hence, by item~\ref{item1}, 
\[
\lfloor \widetilde{\mathcal{G}} (g)\rfloor_{n,\sigma}^\bu = \lfloor \mathcal{G}(\mathcal{A}(\nabla g)) \rfloor_{n,\sigma}^{\bu} \lesssim
|\mathcal{A}(\nabla g)|_{n+1,\sigma}^\bu \lesssim \frac{1}{\kappa} |g|_{n,\sigma}^{\bu}.
\]
This concludes the proof of this lemma.  
\end{proof}

We notice that from~\eqref{eqK1matching_1} we can write
\[
\big (\mathrm{Id} - \widetilde{\mathcal{G}}\big ) K_1^\bu = K^0 + \mathcal{G}(\mathcal{B}).
\]
Therefore, since by Lemma~\ref{lem:propGAmatching}, taking $\kappa\geq \kappa_0$ large enough, the map $\mathrm{Id} - \widetilde{\mathcal{G}}$ is invertible, 
\begin{equation}\label{exp:matching:end}
K_1^{\bu} = \big (\mathrm{Id} - \widetilde{\mathcal{G}}\big )^{-1} \big (K^0 + \mathcal{G}(\mathcal{B})).
\end{equation}
The following result analyzes $\mathcal{B}$ and $K^0$. 
\begin{lemma}\label{lem:independentterm:matching} 
Let $\kappa_0$ be such that Theorems~\ref{thm:outer} and~\ref{thm:inner} hold. 
There exists $\kappa\geq \kappa_0$ and a constant $c$ such that 
\[ 
|K^0|_{4,\sigma}^{\bu} \leq c \varepsilon^{1-\gamma} , \qquad 
|\mathcal{B} |_{5,\sigma}^{\bu} \leq c \varepsilon 
\]
with $K^0,\mathcal{B}$ defined in~\eqref{def:K0:matching} and~\eqref{defABmatching}, respectively. 
\end{lemma}
\begin{proof}
We first recall that, by definition~\eqref{defzkmatching} of $z_{[k]}$ and using also~\eqref{propz1z2}, 
\[
-k \Im (z_{[k]}-z) \leq 0, \qquad k\in \mathbb{Z}.
\]
Then,  using also Theorems~\eqref{thm:outer} and~\eqref{thm:inner}, $| K^0 (z , \theta) | \lesssim   |z_1|^{-5} \lesssim \varepsilon^{5(1-\gamma)}$ and from Lemma~\ref{lem:banach:matching}, $|K^0|_{4,\sigma}^\bu \lesssim \varepsilon^{1-\gamma}$.

Now we deal with $\mathcal{B}$. 
We introduce 
\begin{equation*}
    R^{\mathrm{mch}}(z;\varepsilon) = \varepsilon^2 R(u(z)) + \frac{i}{z^2} , \qquad 
    r^{\mathrm{mch}} (z;\varepsilon) = \varepsilon r(u(z)) + \frac{i}{z}
\end{equation*}
and 
\begin{equation*}
  \hat{\gamma}_0^{\mathrm{mch}} (z,\theta;\varepsilon) =  \hat{\gamma}_0 +  \frac{i}{z} R_\theta e_1, \qquad 
   \hat{\delta}_0^{\mathrm{mch}} (z,\theta;\varepsilon) =  \hat{\delta }_0 +  \frac{i}{z^2} R_\theta e_1, 
\end{equation*}
where $e_1=(1,0)^\top$ and $R_\theta$ and  $\hat{\gamma}_0, \hat{\delta}_0$ were defined in~\eqref{def:in:Rtheta} and~\eqref{def:gamma:delta:hat:inner}, respectively. It follows from~\eqref{hom:poles} and~\eqref{def:gamma:delta:hat:inner} that
\begin{equation}\label{boundsRrmatch}
|z^2 R^{\mathrm{mch}} (z;\varepsilon)|, \, |z r^{\mathrm{mch}} (z;\varepsilon)|, \,|z\hat{\gamma}_0^{\mathrm{mch}} (z,\theta;\varepsilon)| , \,|z^2 \hat{\delta}_0^{\mathrm{mch}} (z,\theta;\varepsilon) |\lesssim |\varepsilon z| \lesssim \varepsilon^{\gamma}.
\end{equation}
In addition, using~\eqref{dhatgamma0inner} and the definition~\eqref{defABmatching} of $\mathcal{B}$, we can write $\mathcal{B}= \mathcal{B}_1+ \mathcal{B}_2$, where
\begin{align*}
        \mathcal{B}_1(z,\theta)=& (\partial_z K^\bu)^2 \left ( \frac{1}{\varepsilon^4 R^2(u(z))} +z^4 \right ) + 
        (\partial_\theta K^\bu)^2 \left ( \frac{1}{\varepsilon^2 r^2(u(z))} +z^2 \right ) 
        \\ & -\beta \partial_\theta K^\bu \left (\varepsilon^2 r^2(u(z)) + \frac{1}{z^2} \right ) \\ 
        \mathcal{B}_2(z,\theta)=  &\int_{0}^1  \partial_{x} \tilde{H}_1 ( K_\lambda ; \lambda\varepsilon) \big [\hat \delta_0^{\mathrm{mch}} 
        + R_\theta P  \nabla K^\bu \big ]\, d\lambda 
        + \int_{0}^1  \partial_{y} \tilde{H}_1 (  K_\lambda ; \lambda \varepsilon )   \big 
        [\hat \gamma_0^{\mathrm{mch}}   \big ]\, d\lambda \\
        &+ \varepsilon \int_{0}^1 \partial_\varepsilon \tilde{H}_1 (  K_\lambda ; \lambda \varepsilon ) \, d\lambda
\end{align*}
with $K_\lambda = (K_\lambda^x, K_\lambda^y)^\top$ given by
\begin{equation}\label{defKlambdamatching}
\begin{aligned}
K_\lambda^x(z,\theta)=&  -iz^{-2} R_\theta e_1 + R_\theta D_{-iz^2 ,iz } 
        \nabla K^\bu +\lambda (\hat{\delta}_0^{\mathrm{mch}} + R_\theta P  \nabla K^\bu ) 
        \\ 
K_\lambda^y (z,\theta) = & -iz^{-1} R_\theta e_1 + \lambda \hat \gamma_0^{\mathrm{mch}}
\end{aligned}
\end{equation}
and $P$ defined in~\eqref{def:in:Rtheta}. From~\eqref{boundsRrmatch}, we obtain
\[
 \left |\frac{1}{\varepsilon^4 R^2(u(z))} +z^4 \right | =  \left |\frac{-z^4}{1-z^4 (R^{\mathrm{mch}}  (z))^2 + 2i z^2 R^{\mathrm{mch}} (z) } +z^4 \right | \lesssim |z|^5 \varepsilon
\]
and analogously
\[
\left | \frac{1}{\varepsilon^2 r^2(u(z))} +z^2 \right |\lesssim |z|^3 \varepsilon, \qquad 
\left |\varepsilon^2 r^2(u(z)) + \frac{1}{z^2} \right |\lesssim |z|^{-1} \varepsilon .
\]
Combining these estimates with the bound~\eqref{Kumatching} of $\lfloor K_1^\bu\rfloor_{5,\sigma}^\bu$, we have that 
$\mathcal{B}_1 \in \mathcal{Z}_{7,\sigma}^{\bu} \subset \mathcal{Z}_{5,\sigma}^{\bu}$ and 
\begin{equation}\label{boundB1matching}
|\mathcal{B}_1|_{5,\sigma}^\bu \lesssim \frac{\varepsilon}{\kappa^2} .
\end{equation}

To deal with $\mathcal{B}_2$, using~\eqref{boundsRrmatch}, \eqref{Kumatching} and definition~\eqref{def:in:Rtheta} of $P$, we first observe that $|K_\lambda^x|_{2,\sigma}^\bu \lesssim 1$ and $|K_\lambda^y|_{1,\sigma}^\bu \lesssim 1$ (see~\eqref{defKlambdamatching}). Then
since $ \partial_{x} \tilde{H}_1(x,y; \epsilon), \partial_y \tilde{H}_{1}(x,y;\epsilon)=  \mathcal{O}(|(x,y)|^5)$ and $\partial_\varepsilon \tilde{H}_1(x,y;\epsilon)= \mathcal{O}(|(x,y)|^6)$ uniformly in $\epsilon$, for $\lambda\in [0,1]$
\[
|\partial_{x} \tilde{H}_1 ( K_\lambda;\lambda\varepsilon )|_{5,\sigma}^\bu,\, |\partial_{y} \tilde{H}_1 ( K_\lambda;\lambda \varepsilon )|_{5,\sigma}^\bu,\, |\partial_{\varepsilon} \tilde{H}_1 ( K_\lambda;\lambda\varepsilon )|_{6,\sigma}^\bu  \lesssim 1.
\]
Therefore, again using~\eqref{boundsRrmatch} and~\eqref{def:in:Rtheta}, along with 
the estimate $|\partial_{\varepsilon} \tilde{H}_1 ( K_\lambda;\lambda\varepsilon )|_{5,\sigma}^\bu \lesssim \kappa^{-1} |\partial_{\varepsilon} \tilde{H}_1 ( K_\lambda;\lambda\varepsilon )|_{6,\sigma}^\bu$ (see Lemma~~\ref{lem:banach:matching}), we obtain that
\[
|\mathcal{B}_2 |_{5,\sigma}^{\bu} \lesssim \varepsilon .
\]
Then, using also~\eqref{boundB1matching}, the lemma is proved. 
\end{proof}

We now complete the proof of Theorem~\ref{thm:matching}. Indeed, recalling the expression~\eqref{exp:matching:end} of $K_1^\bu$ and applying Lemma~\ref{lem:propGAmatching}, we obtain
\[
|K_1^\bu|_{4,\sigma}^{\bu} \lesssim |K^0|_{4,\sigma}^{\bu} + |\mathcal{G} (\mathcal{B})|_{4,\sigma}^{\bu}. 
\]
Moreover, by Lemma~\ref{lem:independentterm:matching} and using also Lemma~\ref{lem:propGAmatching}, we deduce that 
\[
|K_1^\bu|_{4,\sigma}^{\bu} \lesssim \varepsilon^{1-\gamma}.
\]
This concludes the proof of Theorem~\ref{thm:matching} just recalling that, by definition~\eqref{def:matching}, $K_1^{\bu} = K^{\bu}-K_0^{\bu}$ with $K^{\bu}(z,\beta) = \varepsilon^3 T_1\left (\varepsilon z + i\frac{\pi}{2},\theta \right )$, and then 
\begin{align*}
\left |\left (u-\frac{i\pi}{2} \right )^4 \left (T_1^{\bu}(u,\theta) - \frac{1}{\varepsilon^3} K^{\bu}_0\left (\frac{u- i\frac{\pi}{2}}{\varepsilon}  ,\theta \right ) \right  ) \right | & =
\left |\left (u-\frac{i\pi}{2} \right )^4   \varepsilon^{-3} K_1^{\bu}\left (\frac{u- i\frac{\pi}{2}}{\varepsilon}  ,\theta \right ) \right  | \\ & \leq \varepsilon |K_1^\bu|_{4,\sigma}^{\bu} \lesssim \varepsilon^{2-\gamma}.
\end{align*}

\section{Straightening a variational equation} \label{sec:proof:difference}   

Consider the partial differential equation~\eqref{eq:diff:out:ab}, 
$$
(1+ \mathbf{a}) \partial_u \Delta + \left (\frac{\omega}{\varepsilon} + \mathbf{b}\right ) \partial_{\theta} \Delta =0, \qquad (u,\theta) \in E \times \mathbb{T}_{\sigma} 
$$
with $E$ defined in~\eqref{def:E:final} and $\mathbf{a}, \mathbf{b}$ defined in~\eqref{def:A:diff}. Using  Theorems~\ref{thm:ext} and~\ref{thm:outer} and that $\partial_x \tilde{H_1}(x,y; \varepsilon) = \mathcal{O}_5(x,y)$ uniformly in $\varepsilon$, one can easly check that   
\[  
    \sup_{(u,\theta) \in \E \times \T_\sigma} \left | \left (u^2 + \frac{\pi^2}{4}\right )^2\mathbf{a}(u,\theta) \right | \lesssim  \varepsilon^2, \quad \sup_{(u,\theta) \in \E \times \T_\sigma} \left | \left (u^2 + \frac{\pi^2}{4}\right )^2\mathbf{b}(u,\theta) \right | \lesssim  \varepsilon.
\]
Therefore, $|1+ \varepsilon \omega^{-1} \mathbf{b}| \geq 1/2$, taking $\kappa=s|\log \varepsilon| $ large enough, and equation~\eqref{eq:diff:out:ab} is equivalent to 
$$
\frac{\omega}{\varepsilon} \partial_{\theta } \Delta +   (1+ \overline{\mathbf{a}}) \partial_u \Delta = 0  
$$
with 
$$
\overline{\mathbf{a}} = \frac{1+\mathbf{a}}{ 1 + \varepsilon \omega^{-1} \mathbf{b}} -1 = \frac{\mathbf{a} - \varepsilon \omega^{-1} \mathbf{b}}{ 1 + \varepsilon \omega^{-1} \mathbf{b}}.
$$
Considering $\widetilde{\Delta} (u,\theta) = \Delta (u, \omega^{-1} \theta)$, we obtain
\begin{equation}\label{eqCdifference}
 \frac{1}{\varepsilon} \partial_\theta \widetilde {\Delta}   + (1+
  \widetilde{\mathbf{a}} )\partial_u\widetilde{\Delta} =0, \qquad 
  \widetilde{\mathbf{a}}(u,\theta) = \overline{\mathbf{a}} (u, \omega^{-1} \theta)
\end{equation}
and it is clear that 
\begin{equation}\label{bound:atilde}
    \sup_{(u,\theta) \in \E \times \T_\sigma} \left | \left (u^2 + \frac{\pi^2}{4}\right )^2\widetilde{\mathbf{a}}(u,\theta) \right | \lesssim  \varepsilon^2.
\end{equation}

Similar equations to~\eqref{eqCdifference} were also studied in~\cite{B12}. The first claim is that if $\xi(u,\theta)$ is a solution of 
\begin{equation}\label{eq:diff:psi}
\frac{1}{\varepsilon} \partial_\theta \psi + (1+\widetilde{\mathbf{a}})\partial_u \psi =0
\end{equation}
satisfying that $(\xi(u,\theta), \theta)$ is injective, then all the solutions of~\eqref{eqCdifference} are of the form $\chi(\xi(u,\theta))$. Indeed, if such solution $\xi$ exists, consider the change of variables 
$$
(v, \tau) = h(u,\theta):=(\xi(u,\theta), \theta), \qquad \varphi(v,\tau)=\psi (h^{-1}(v,\tau)).
$$
Then, $\psi(u,\theta) = \varphi(h(u,\theta))$ is a solution of~\eqref{eq:diff:psi} if and only if
$$
\frac{1}{\varepsilon} \big (\partial_{\tau} \varphi + \partial_{v} \varphi \partial_{\theta} \xi \big ) + (1+ \widetilde{\mathbf{a}}) \partial_v \varphi \partial_u \xi =0
$$
and, using that $\xi$ is a particular solution, we obtain that $\partial_{\tau} \varphi=0$. Therefore $\varphi(v,\tau) = \chi (v)$ and as a consequence
$$
\psi(u,\theta) = \varphi(h(u,\theta))=\varphi(\xi(u,\theta),\theta) = \chi(\xi(u,\theta)).
$$

We focus on solutions $\xi$ of the form $ \xi(u,\theta)= \varepsilon^{-1} (u + \mathcal{C}(u,\theta)) - \theta$ with $\mathcal{C}(u,\theta)$ being $2\pi$-periodic. 
We emphasize that, if such a solution exists, then 
$$
\Delta (u,\theta) = \Upsilon (\varepsilon^{-1} (u + \mathcal{C}(u,\theta)) - \theta) 
$$
with $\Upsilon$ a $2\pi-$ periodic function in its variable. In addition, it is clear that $\mathcal{C}$ has to satisfy 
\begin{equation}\label{eq:mathcalC:diff}
\frac{1}{\varepsilon} \partial_{\theta} \mathcal{C} + \partial_u \mathcal{C} =- \widetilde{\mathbf{a}} (1+ \partial_u \mathcal{C}).
\end{equation}

As a result of this preliminary analysis, Proposition~\ref{prop:C} is an straightforward consequence of the following lemma.
\begin{lemma}
Equation~\eqref{eq:mathcalC:diff} has a real analytic solution $\mathcal{C}:E \times \mathbb{T}_{\sigma} \to \mathbb{C}$. Moreover, defining
\begin{equation}\label{def:norm:diff:C}
| \mathcal{C} |_{1,\sigma}:=\sum_{k\in \mathbb{Z}} |\mathcal{C}^{[\ell]}|_1 e^{|\ell| \sigma}, \qquad |\mathcal{C}^{[\ell]}|_{1}:= \sup_{u \in \E} \left |\left (u^2 + \frac{\pi^2}{4}\right ) \mathcal{C}^{[\ell]}(u)\right |
\end{equation}
we have that $|\mathcal{C}|_{1,\sigma} \lesssim \varepsilon ^{2}$. 
\end{lemma}
\begin{proof}
Let $\kappa = s |\log \varepsilon|$ and consider
$u_{[k]}=\frac{\pi}{2} - \kappa\varepsilon$ for $k>0$, $u_{[k]}=-\frac{\pi}{2} + \kappa\varepsilon$ for $k<0$ 
and $u_{[0]}\in \mathbb{R}$ the element of $\E$ with the larger real part. We define the operators 
\begin{align*}
\mathcal{L}_{\varepsilon}(g)(u,\theta) &=  \frac{1}{\varepsilon} \partial_{\theta} g (u,\theta) + \partial_u g (u,\theta)\\
    \mathcal{G}_{\varepsilon}(g)(u,\theta) &= \sum_{k\in \mathbb{Z}} \mathcal{G}^{[k]}(g)(u) e^{ik\theta}, \qquad \mathcal{G}^{[k]}(u)=\int_{u_{[k]}}^{u} e^{ik\varepsilon^{-1}(v-u)} g^{[k]}(v) dv \\
    \mathcal{R}(g) (u,\theta) &= -\widetilde{\mathbf{a}}(u,\theta) (1+ \partial_u g (u,\theta)).
\end{align*}
The operator $\mathcal{G}_\varepsilon$ was also defined in~\cite{B12} (see Section 9.1). 
From its definition, it is immediate to check that $\mathcal{L}_{\varepsilon} \circ \mathcal{G}_{\varepsilon}(g) = g$ and therefore, if $\mathcal{C}$ is a solution of the fixed point equation
$$
\mathcal{C} = \mathcal{G}_\varepsilon \circ \mathcal{R}(\mathcal{C}),
$$
it is also a solution of equation~\eqref{eq:mathcalC:diff}. We introduce the norms
\begin{equation*}
| \mathcal{C} |_{n,\sigma}:=\sum_{k\in \mathbb{Z}} |\mathcal{C}^{[\ell]}|_n e^{|\ell| \sigma}, \qquad |\mathcal{C}^{[\ell]}|_{n}:= \sup_{u \in \E} \left |\left (u^2 + \frac{\pi^2}{4}\right )^n \mathcal{C}^{[\ell]}(u)\right |
\end{equation*}
and the Banach spaces
\[
\mathcal{Z}_{n, \sigma} = \left\{ \phi : \E\times \T_\sigma \to \C : \mbox{$\phi$ is real analytic and $\lfloor\phi\rfloor_{n,\sigma}:= |\phi|_{n, \sigma} + |\partial_u \phi|_{n+1,\sigma}<\infty$}\right\}.
\]
In~\cite{B12} (see Lemmas 9.1 and 9.2) was proven that, there exists a constant $c$ such that for all $g\in  \mathcal{Z}_{n ,\sigma}$ with $n>1$ and $g_1\in \mathcal{Z}_{n_1,\sigma}, g_2\in \mathcal{Z}_{n_2}$ with $n_1,n_2\geq 0$, 
$$
\lfloor\mathcal{G}_\varepsilon (g )\rfloor_{n-1,\sigma} \leq c |g|_{n,\sigma}, \qquad | g_1 g_2|_{n_1+n_2,\sigma} \leq c |g_1 |_{n_1,\sigma} |g_2 |_{n_2,\sigma}  .
$$

From~\eqref{bound:atilde} 
we deduce that $|\widetilde{\mathbf{a}}|_{2,\sigma} \leq \hat c \varepsilon^2$ for a positive constant $\hat c$. 
Let 
\[
\varrho=2 \lfloor \mathcal{G}_\varepsilon \circ \mathcal{R}(0) \rfloor_{1,\sigma}\lesssim |\mathcal{R}(0)|_{2,\sigma} \lesssim \varepsilon^2
\]
and $B_\varrho\subset \mathcal{Z}_{1,\sigma}$ be the closed ball of radius $\varrho$ centered at $0$. We notice that if $u\in \E$, 
\[
\left | u^2 + \frac{\pi^2}{4}\right | \gtrsim \varepsilon  \kappa  = \varepsilon  s  |\log \varepsilon| ,
\]

Then, for $g_1, g_2\in B_\varrho$
\begin{align*}
\lfloor \mathcal{G}_{\varepsilon} \circ \mathcal{R}(g_1) - \mathcal{G}_\varepsilon \circ \mathcal{R}(g_2)\rfloor_{1,\sigma}  
&\leq c |\mathcal{R}(g_1) - \mathcal{R}(g_2)|_{2,\sigma} \leq c \hat{c} s^{-2} |\log \varepsilon|^{-2} |\partial_u g_1- \partial_u g_2|_{2,\sigma} \\ &\leq c \hat{c}  s^{-2}|\log \varepsilon|^{-2}  \lfloor g_1-g_2\rfloor_{1,\sigma} \leq \frac{1}{2}\lfloor g_1-g_2\rfloor_{1,\sigma}
\end{align*}
if $\varepsilon$ is small enough. Moreover, for $g\in B_\varrho$, 
$$
\lfloor \mathcal{G}_{\varepsilon} \circ \mathcal{R}(g) \rfloor_{1,\sigma} \leq \lfloor \mathcal{G}_{\varepsilon} \circ \mathcal{R}(0) \rfloor_{1,\sigma} + \lfloor \mathcal{G}_{\varepsilon} \circ \mathcal{R}(g)  - \mathcal{G}_{\varepsilon} \circ \mathcal{R}(0)\rfloor_{1,\sigma} \leq \frac{\varrho}{2} + \frac{1}{2} \lfloor g\rfloor_{1,\sigma} \leq \varrho.
$$
As a consequence $\mathcal{G}_\varepsilon \circ \mathcal{R} : B_\varrho \to B_\varrho$ is a contraction and therefore the fixed point theorem assures the existence of a (unique) solution $\mathcal{C} \in B_\varrho$ of the fixed point equation $\mathcal{G}_\varepsilon \circ \mathcal{R}(\mathcal{C}) = \mathcal{C}$. This concludes the proof of this lemma.  
\end{proof}
 
\appendix

\section{Versal normal forms}\label{HVNF}
This section is divided into two parts. First, we discuss some ideas of the proof of Theorem \ref{Thm:NormalForm}, concerning the versal normal form associated with generic unfoldings of the Hamiltonian Hopf bifurcation; This is the content of Section \ref{app:NF}. In Section \ref{app:NF_L4}, we present some ideas of the proof of Theorem \ref{Thm:NormalForm:L4}, which concerns the versal normal form associated with the Hamiltonian of the RPC3BP. Obviously, in this case, the strategy used to bring the higher-order (greater than two) terms of the Hamiltonian of the RPC3BP close to $L_4$ into normal form is the same as that of the generic unfoldings of the Hamiltonian Hopf bifurcation. Therefore, in this section, we mainly provide an idea of how to put the quadratic terms into a versal normal form. We recall that the Hamiltonian of the RPC3BP close to $L_4$ is a particular unfolding of the Hamiltonian Hopf bifurcation.

\subsection{Versal normal form of the generic unfoldings}\label{app:NF}

Let us recall the notation. We consider a generic unfolding of the Hamiltonian Hopf bifurcation $\mathbf{H}_{\mu} : \R^2 \times \R^2 \to \R$ with the standard symplectic form $dx_1 \wedge dy_1 + dx_2 \wedge dy_2$. Furthermore, we assume that the quadratic terms $\mathbf{H}^2_\mu$ associated with $\mathbf{H}_\mu$ are
\begin{equation}\label{app:H_mu_2}
   \mathbf{H}^2_\mu (x_1, x_2, y_1, y_2) = \varpi S + {1 \over 2} N + {1 \over 2}\nu_0(\mu) Q, \quad \varpi>0
\end{equation}
with $S=x_1y_2 - x_2y_1$, $N=y_1^2 + y_2^2$, and $Q=y_1^2 + y_2^2$. Throughout this section, we adopt the compact notation $x = (x_1, x_2)^\top$ and $y= (y_1, y_2)^\top$. 

Following the proof contained in Section $3.5$ of~\cite{Meer85}, see also Section 5 of~\cite{Sch94} for the particular case where $\mathbf{H}_\mu$ is the Hamiltonian of the RPC3BP close to $L_4$, we aim to look for a Hamiltonian 
\begin{equation}\label{def:Wnu}
    W_\mu(x,y) =  W^1_\mu(x,y) +  W^2_\mu(x,y) + W^3_\mu(x,y)
\end{equation}
where $W^n_\mu$, for $n =1,2,3$, are homogeneous polynomials of degree $i+2$ in the four variables $(x,y)$, depending on $\mu$, such that the time $1$-map $\phi^1_{W_\mu}$ associated with $W_\mu$ is a symplectic change of coordinates satifying $\mathbf{H}_\mu\circ \phi^1_{W_\mu}=\check H_0 + \check H_1$ where $\check H_0$ and $\check H_1$ are as in Theorem \ref{Thm:NormalForm}. 

For this purpose, we introduce the following notation
\begin{equation}\label{app:def_H_mu}
    \mathbf{H_\mu}(x,y) = \sum_{n \ge 2} \mathbf{H}^n_\mu(x,y)
\end{equation}
where $\mathbf{H}^2_\mu$ is as in~\eqref{app:H_mu_2} and, for all $n \ge 3$,  $\mathbf{H}^n_\mu$ is an homogeneous polynomial of degree $n$ in the four variables $(x,y)$ with coefficients which are analytic functions in the parameter $\mu$ for $\mu$ sufficiently close to $0$. Furthermore,  we denote by $\{\cdot, \cdot\}$ the Poisson bracket associated with the standard symplectic form $dx_1 \wedge dy_1 + dx_2 \wedge dy_2$, and by $\mathcal{P}_n$ the space of the homogeneous polynomials of degree $n$ in the variables $(x,y)$.  

In Section 3.5 of~\cite{Meer85}, the author look for $W^n_\mu$ as a solution of the following equation
\begin{equation*}
    \mathbf{\tilde H}^{n+2}_\mu = K + \left\{\mathbf{H}^2_\mu, W^n_\mu\right\}
\end{equation*}
where $K$ contains only known terms, namely, those depending exclusively on $W^j_\mu$ with $j<n$, and $\mathbf{\tilde H}^{n+2}_\mu$ stands for the term of order $n+2$ of the new Hamiltonian expressed in the new variables. We have omitted the dependence on the variables $(x,y)$ for brevity, and we will do so throughout the rest of this section whenever no confusion may arise. Hence, deriving a normal form for $\mathbf{H}_\mu$ is equivalent to studying the following linear operator
\begin{equation}\label{app:lin_op_{H2mu}}
    \left\{\mathbf{H}_\mu^2, \cdot\right\}: \mathcal{P}_n \longrightarrow \mathcal{P}_n.
\end{equation}
We have the following decomposition
\begin{equation}\label{app:decomPn}
    \mathcal{P}_n = \mathrm{Im}\left(\left\{\mathbf{H}_\mu^2, \cdot\right\}\right) \oplus \mathcal{N}(\mathbf{H}_\mu^2)
\end{equation}
where $\mathcal{N}(\mathbf{H}_\mu^2) = \mathrm{Ker}\left(\{S, \cdot\}\right) \cap \mathrm{Ker}\left(\{Q, \cdot\}\right)$. We point out that $\mathrm{Ker}(\{S, \cdot\})$ and $\mathrm{Ker}(\{Q, \cdot\})$ denote for the kernel of the linear operators $\{S, \cdot\}$ and $\{Q, \cdot\}$, where $S$ and $Q$ are defined in~\eqref{app:H_mu_2}. In Section 3.5 of~\cite{Meer85}, it is proven that 
\begin{equation}\label{app:NH2mu}
    \mathcal{N}(\mathbf{H}_\mu^2) = \begin{cases} 0 \hspace{45.5mm} \mbox{if $n$ is odd},\\
    \mathrm{span}\left\{S^iQ^j: i+j = {n\over 2}\right\} \quad \mbox{if $n$ is even}.\end{cases}
\end{equation}
This means that the linear operator $\left\{\mathbf{H}_\mu^2, \cdot\right\}$ is invertible if $n$ is odd, otherwise its kernel is given by the homogeneous polynomials of degree ${n \over 2}$ in the variables $S$ and $Q$. This means that
\begin{equation}\label{app:tildeH345}
    \mathbf{\tilde H}^{3}_\mu = \mathbf{\tilde H}^{5}_\mu = 0, \quad \mbox{whereas} \hspace{3mm} \mathbf{\tilde H}^{4}_\mu = {1 \over 4} \check{\gamma}(\mu)   Q^2+ {1 \over 4} \check{\alpha}(\mu)    S^2 + \frac{1}{2}\check{\beta}(\mu)    Q  \,  S  
\end{equation}
for suitable functions $\check{\gamma}$, $\check{\alpha}$ and $\check{\beta}$ depending on $\mu$. It remains to verify that the regularity with respect to $\mu$. 

We point out that, in~\cite{Meer85}, the author works with the linear operator $\left\{\mathbf{H}_0^2, \cdot\right\}: \mathcal{P}_n \longrightarrow \mathcal{P}_n$. The proof is essentially the same and we obtain the decomposition $\mathcal{P}_n = \mathrm{Im}\left(\left\{\mathbf{H}_0^2, \cdot\right\}\right) \oplus \mathcal{N}(\mathbf{H}_0^2)$ with $\mathcal{N}(\mathbf{H}_0^2) = \mathrm{Ker}\left(\{S, \cdot\}\right) \cap \mathrm{Ker}\left(\{Q, \cdot\}\right)$. In this work, we preferred the approach used in~\cite{Sch94}.

In~\cite{Meer85}, the author works with $C^\infty$ functions, thus he proved that the time $1$-map $\phi^1_{W_\mu}$ is only $C^\infty$ with respect to the parameter $\mu$, at least near $0$. Here, we aim to verify that $\phi^1_{W_\mu}$ is, in fact, analytic in $\mu$.
\begin{lemma}\label{NFLemmaNonLinearTermsReg}
The change of coordinates $\phi^1_{W_\mu}$ is analytic with respect to $\mu$ for $\mu$ sufficiently close to zero.
\end{lemma}
\begin{proof}
The proof of this lemma is divided into three parts, each dedicated to analyzing a different component $W_\mu^n$ of $W_\mu$ in~\eqref{def:Wnu} for $n=1,2,3$.  It suffices to verify that each $W^n_\mu$ is analytic with respect to $\mu$, for $\mu$ small enough.  First, we recall the notation introduced in~\eqref{app:def_H_mu}, we denote by $B \subset \R$ a sufficiently small interval centered at the origin, and we point out that the space $\mathcal{P}_n$ of the homogeneous polynomials of degree $n$ in the variables $(x,y)$ is a Banach space endowed with a suitable norm. 

\vspace{2mm}
\textit{Analysis of $W^1_\mu$.} For all $\mu$ in a sufficiently small interval centered at the origin,  the first component $W^1_\mu$ is obtained as a solution of 
\begin{equation*}
\mathbf{H}^3_\mu + \left\{\mathbf{H}^2_\mu , W^1_\mu\right\}=0,
\end{equation*}
(see Section $3.2$ of~\cite{Meer85}). We define the following functional 
\begin{equation*}
\mathcal{F} : B \times \mathcal{P}_3 \longrightarrow \mathcal{P}_3, \quad \mathcal{F}(\mu,  W^1) = \mathbf{H}^3_\mu + \left\{\mathbf{H}^2_\mu , W^1\right\}.
\end{equation*}
For all $\mu$ sufficiently small, we already know the existence of a function $W^1_\mu$ such that $\mathcal{F}(\mu,  W^1_\mu)=0$.  Moreover,  
\begin{equation*}
\partial_{W^1} \mathcal{F}(0,  0): \mathcal{P}_3 \to \mathcal{P}_3,  \quad \partial_{W^1} \mathcal{F}(0,  0) \widehat W^1= \left\{\mathbf{H}^2_\mu, \widehat W^1\right\}
\end{equation*}
and by the previous analysis (see~\eqref{app:lin_op_{H2mu}} and also~\cite{Meer85}), it admits a right inverse.  Then, the implicit function theorem assures that for all $\mu$ in a small neighborhood of the origin,  $W^1_\mu$ is analytic with respect to $\mu$. 

\vspace{2mm}
\textit{Analysis of $W^2_\mu$.} For all $\mu$ close enough to zero, $W^2_\mu$ is obtained as a solution of the following equation 
\begin{equation*}
\mathbf{H}_\mu^4 + \left\{\mathbf{H}_\mu^3, W^1_\mu\right\} + {1 \over 2}\left\{\left\{\mathbf{H}^2_\mu, W^1_\mu\right\}, W^1_\mu\right\} + \left\{\mathbf{H}^2_\mu, W^2_\mu\right\} = {1 \over 4} \check{\gamma}(\mu)   Q^2+ {1 \over 4} \check{\alpha}(\mu)    S^2 + \frac{1}{2}\check{\beta}(\mu)    Q  \,  S  
\end{equation*}
see Sections 3.5 of~\cite{Meer85}. We point out that,  $W^1_\mu$ is the solution of the previous step, whereas $\check\alpha(\mu)$, $\check\beta(\mu)$ and $\check\gamma(\mu)$ are the ones in~\eqref{app:tildeH345}.  In this case, it is crucial to analyze the following linear operator
\begin{equation}
\label{H20P4}
\left\{\mathbf{H}^2_\mu, \cdot \right\}: \mathcal{P}_4 \longrightarrow\mathcal{P}_4. 
\end{equation}
Let
\begin{eqnarray*}
 \mathcal{\widehat P}_4 &=& \mathcal{N}(\mathbf{H}^2_\mu) = \mathrm{span}\left\{S^iQ^j: i+j = 2\right\} \subset \mathcal{P}_4.
\end{eqnarray*}
Thanks to~\eqref{app:decomPn} and~\eqref{app:NH2mu}, we can conclude that the image of $\{\mathbf{H}^2_0, \cdot \}$ is isomorphic to the quotient space $\mathcal{P}_4 / \widehat{\mathcal{P}}_4$, which, with a suitable norm, is a Banach space.

We define the following functional 
\begin{align*}
&\mathcal{F} : B \times \mathcal{P}_4 \longrightarrow \mathcal{P}_4 /  \mathcal{\widehat P}_4, \\
&\mathcal{F}(\mu,  W^2) = \mathbf{H}_\mu^4 + \left\{\mathbf{H}_\mu^3, W^1_\mu\right\} + {1 \over 2}\left\{\left\{\mathbf{H}^2_\mu, W^1_\mu\right\}, W^1_\mu\right\} + \left\{\mathbf{H}^2_\mu, W^2\right\} - \check\alpha(\mu) S^2\\
&\hspace{17mm}- \check\beta(\mu)SQ - \check\gamma(\mu) Q^2. 
\end{align*}
It is well-defined because we proved that the image of the operator~\eqref{H20P4} is isomorphic to $ \mathcal{P}_4 /  \mathcal{\widehat P}_4$ and the homogeneous polynomial $\check\alpha(\mu) S^2 + \check\beta(\mu)SQ + \check\gamma(\mu) Q^2$ is chosen in~\eqref{app:tildeH345}, see also~\cite{Meer85}, in such a way that $\mathbf{H}_\mu^4 + \left\{\mathbf{H}_\mu^3, W^1_\mu\right\} + {1 \over 2}\left\{\left\{\mathbf{H}^2_\mu, W^1_\mu\right\}, W^1_\mu\right\} - \check\alpha(\mu) S^2-\check\beta(\mu)SQ - \check\gamma(\mu) Q^2 \in  \mathcal{P}_4 /  \mathcal{\widehat P}_4$.

For all $\mu$ small enough, we know the existence of a function $W^2_\mu$ satisfying $\mathcal{F}(\mu,  W^2_\mu)=0$. Moreover, thanks to the previous analysis (see~\eqref{app:lin_op_{H2mu}}) of the operator defined by~\eqref{H20P4}, we have that 
\begin{equation*}
\partial_{W^2} \mathcal{F}(0,  0): \mathcal{P}_4 \to \mathcal{P}_4/ \mathcal{\widehat P}_4,  \quad \partial_{W^2} \mathcal{F}(0,  0) \widehat W^2= \left\{\mathbf{H}^2_\mu, \widehat W^2\right\}
\end{equation*}
admits a right inverse. Then, with the same argument of the previous case,  $W^2_\mu$ is analytic with respect to $\mu$.   

\vspace{2mm}
\textit{Analysis of $W^3_\mu$.} In this case,  for all $\mu$ small enough,  $W^3_\mu$ is the solution of 
\begin{align*}
\mathbf{H}^5_\mu &+ + \left\{\mathbf{H}^3_\mu, W^2_\mu\right\} + \left\{\mathbf{H}^4_\mu, W^1_\mu\right\} + {1 \over 2}\left\{\left\{\mathbf{H}^2_\mu, W^1_\mu\right\}, W^2_\mu\right\} + {1 \over 2}\left\{\left\{\mathbf{H}^2_\mu, W^2_\mu\right\}, W^1_\mu\right\} \\
&+ {1 \over 2}\left\{\left\{\mathbf{H}^3_\mu, W^1_\mu\right\}, W^1_\mu\right\} + \left\{\mathbf{H}^2_\mu, W^3_\mu\right\} = 0. 
\end{align*}
As in previous cases, we need to study the following operator 
\begin{equation*}
\left\{\mathbf{H}^2_\mu, \cdot \right\}: \mathcal{P}_5 \longrightarrow\mathcal{P}_5.
\end{equation*}
which is surjective.  Then, similarly to the case of $W^1_\mu$, one can prove that $W^3_\mu$ is analytic with respect to the parameter $\mu$ for $\mu$ small enough.   
\end{proof}

\subsection{Versal normal form of the RPC3BP around $L4$}\label{app:NF_L4}

In this section, we provide some ideas about the proof of Theorem \ref{Thm:NormalForm:L4}, which is essentially given in~\cite{Sch94}. We aim to study the versal normal form associated with the Hamiltonian $H$ given by~\eqref{H2} of the RPC3BP close to the Lagrangian point $L_4$. The proof in~\cite{Sch94} is divided into two parts. First, the author analyzes the quadratic terms of the Hamiltonian $H$, then the higher-order terms. 

To derive a versal normal form for the quadratic part of the Hamiltonian $H$, we introduce the following notation.  Let $\mathcal{M}_n(\C)$ be the set of the square matrices $n\times n$ with complex coefficients.  For each $A \in \mathcal{M}_n(\C)$, we denote by $A^j$ the $j$-th column of $A$ for all $j=1,...,n$ and by $\overline{A}^j$ the complex conjugate of $A^j$:
\[
A=\left (A^1, A^2, A^3, A^4\right ), \qquad A^j\in \mathbb{C}^4. 
\]
We have the following
\begin{lemma} There exists $\mu_0>0$ such that for all $\mu \in (-\mu_0,\mu_0)$, there exists an invertible symplectic matrix $A(\mu)\in\mathcal{M}_4(\C)$ in such a way that in the new complex position coordinates $\xi = (\xi_1, \xi_2)$ and complex momenta $\eta= (\eta_1, \eta_2)$ with 
\begin{equation*}
\begin{pmatrix} Q_1 \\ Q_2 \\ P_1 \\ P_2\end{pmatrix} = A(\mu) \begin{pmatrix} \xi_1 \\ \xi_2 \\ \eta_1 \\ \eta_2\end{pmatrix},
\end{equation*}
the Hamiltonian~\eqref{H2} takes the following form
\begin{equation}
\label{HafterT}
H(\xi, \eta; \mu)= i \hat\omega (\xi_1 \eta_1 - \xi_2 \eta_2) + \xi_1 \xi_2 + \mu \eta_1 \eta_2 + \mathcal{O}_3(\xi, \eta; \mu).
\end{equation}
Here $\mathcal{O}_3$ stands for terms,  depending uniformly on $\mu$,  of order at least $3$ in the new variables $(\xi,\eta)$.

Moreover $A(\mu)$ is analytic with respect to $\mu \in (-\mu_0,\mu_0)$ and it satisfies  $A^2(\mu) = \overline{A}^1(\mu)$ and $A^3(\mu) =  \overline{A}^4(\mu)$.
\end{lemma}
\begin{proof}
The proof of this lemma can be found in Section $4$ of~\cite{Sch94} where the author also provides an explicit formula for $A(\mu)$.  
\end{proof}
Using the following symplectic transformation with multiplier $2$
\begin{equation}
\label{XiEtaXY}
\xi_1 = x_1 + i x_2, \quad \eta_1 = y_1 - i y_2, \quad \xi_2 = x_1 - i x_2, \quad \eta_2=y_1 + i y_2,
\end{equation} 
we can rewrite the Hamiltonian~\eqref{HafterT} in the following real versal normal form 
\begin{equation}
\label{HafterT2}
H(x,y;\nu)= \hat \omega \left(x_1 y_2 - x_2 y_1\right) + {1 \over 2} \left(x_1^2 + x_2^2\right) + {1 \over 2}\nu \left(y_1^2 + y_2^2 \right)+ \mathcal{O}_3(x,y; \nu).
\end{equation}
This concludes the analysis of the quadratic terms associated with $H$. Now, following the lines of Section \ref{app:NF}, one can conclude the proof of Theorem \ref{Thm:NormalForm:L4}. We also point out that the proof of Theorem \ref{Thm:NormalForm:L4} is contained in Sections 5-8 of~\cite{Sch94}.

\bibliographystyle{amsalpha}
\bibliography{ref}
\end{document}